\documentclass[review]{elsarticle}
\usepackage{amsthm}
\usepackage{lineno,hyperref}
\usepackage{moreverb}
\usepackage{graphicx,subfig,wrapfig}
\usepackage{algorithm}
\usepackage{algpseudocode}
\usepackage{color}
\usepackage{amsmath}
\usepackage{txfonts}
\usepackage{afterpage}
\usepackage[margin=1in]{geometry}
\usepackage{mathrsfs}
\usepackage{tikz}
\usepackage{pstricks}
\usepackage{diagbox}
\usepackage{arydshln}
\newtheorem{theorem}{Theorem}
\newtheorem{mlemma}{Lemma}
\newtheorem{mydef}{Definition}
\newtheorem{myprop}{Proposition}
 \font\msbm=msbm10

\def\dsR{\hbox{{\msbm \char "52}}}
\def\div{\nabla \cdot}

\def\pL{\mathcal{L}}
\def\pB{\mathfrak{b}}

\def\zM{\mathcal{M}}
\def\zF{\mathcal{F}}
\def\zK{\mb{K}}
\def\zxi{\mb{\xi}}
\def\zX{\mb{\mathrm{X}}}
\def\zW{\mb{\mathrm{W}}}
\def\zi{\mb{i}}
\def\zgamma{\tilde{\gamma}}
\def\mgamma{\mb{\gamma}}
\def\zy{\mb{y}}
\def\zv{\mb{v}}
\def\zaph{\mb{\alpha}}
\def\zomega{\mb{\omega}}
\def\zC{\mb{C}}
\def\Cth{c_1}
\def\Ctha{c_2}
\def\Cthb{c_3}
\def\zXe{\mb{\mathrm{X}}_{\rm exact}}

\def\covl{l_c}

\renewcommand{\mathbf}{\boldsymbol}
\newcommand{\mb}{\mathbf}

\newcommand{\Se}{\mathbb{S}}
\newcommand{\norm}[2]{\left\| #1 \right\|_{#2}}

\newcommand{\Pe}{{\mathbb{P}}}
\newcommand{\Ne}{{\mathbb{N}}}
\newcommand{\Ee}{{\mathbb{E}}}

\modulolinenumbers[5]

\journal{Journal of Computational Physics}

\begin{document}

\begin{frontmatter}

\title{Rank adaptive tensor recovery based model reduction for partial differential equations with high-dimensional random inputs}


\author[mymainaddress]{Kejun Tang}
\ead{tangkj@shanghaitech.edu.cn}

\author[mysecondaryaddress]{Qifeng Liao\corref{mycorrespondingauthor}}
\cortext[mycorrespondingauthor]{Corresponding author}
\ead{liaoqf@shanghaitech.edu.cn}

\address[mymainaddress]{School of Information Science and Technology, ShanghaiTech University, Shanghai, China}
\address[mysecondaryaddress]{School of Information Science and Technology, ShanghaiTech University, Shanghai, China}

\begin{abstract}
This work proposes a systematic model reduction approach based on rank adaptive tensor recovery for partial differential equation (PDE) models with high-dimensional random parameters. Since the standard outputs of interest of these models are discrete solutions on given physical grids which are high-dimensional, we use kernel principal component analysis to construct stochastic collocation approximations in reduced dimensional spaces of the outputs. To address the issue of high-dimensional random inputs, we develop a new efficient rank adaptive tensor recovery approach to compute the collocation coefficients. 
Novel  efficient initialization strategies for non-convex optimization problems involved in tensor recovery are also developed in this work. We present a general mathematical framework of our overall model reduction approach, analyze its stability, and demonstrate its efficiency with numerical experiments.

\end{abstract}

\begin{keyword}
tensor recovery; model reduction; PDEs; uncertainty quantification
\end{keyword}

\end{frontmatter}


\section{Introduction}\label{section_intro}
During the last few decades there has been a rapid development in surrogate and reduced  order modelling 
for PDE systems with  random inputs. 
The PDE systems are fundamental mathematical models describing
complex physical and engineering problems, which can involve multiple disciplines, 
a large number of input parameters, and multiple sources of uncertainty. 
A main challenge of surrogate modelling for these PDE models is the so-called curse of dimensionality.
First, due to the high complexity of practical problems, the random input parameters are typically high-dimensional.
Second, the standard output of these PDE models is the spatial fields (e.g., temperature, pressure and velocity), 
and their fine resolution representation requires a large number of degrees of freedom, which make the output high-dimensional.

A type of widely used surrogate modelling approach for these PDE models is the stochastic spectral methods \cite{ghanem2003sfem,xiu2002wiener,Xiu2005High,maza2009adaptive,powelm09},
while the high dimensionality of  the random inputs causes difficulties in applying them. 
To alleviate the difficulty, modifications of these methods have been actively introduced by exploiting certain properties of the underlying 
problem. For example,
sparse (generalized) polynomial chaos (gPC) expansions \cite{doostan2011nonsparse,peng2014weighted,yan2012stochastic,jakeman2017generalized,lei2015constructing,guo2018gradient} 
are developed  through using the sparsity in spectral approximations.
Moreover, the stochastic collocation method \cite{Xiu2005High,babuvska2007stochastic,maza2009adaptive} is reformulated as a tensor style quadrature problem in
\cite{zhang2017bigdata}, which shows that the corresponding collocation coefficients 
can be efficiently computed through tensor recovery techniques. 
On the other hand, the high dimensionality in outputs poses challenges in both surrogate modelling and data storage. 
The surrogates proposed to resolve high-dimensional inputs as discussed above are typically restricted to problems with a single output. 
Naively extending them to high-dimensional outputs (building independent surrogates for multiple outputs) is computationally 
infeasible. For making progress, dimension reduction methods for the outputs gain a lot of interests.
For example,  principal component analysis (PCA) and kernel component analysis (kPCA) methods are successfully established for
Gaussian process surrogates \cite{contiohagan2010bayesian,xing2016manifold}. Especially, since kPCA captures highly nonlinear low-rank structures 
in the output space, it can provide dramatically tight representation of the outputs \cite{maza2011kernel,xing2016manifold}. 

In this work, we focus on tensor recovery based stochastic collocation. 
As discussed in \cite{zhang2017bigdata}, gPC coefficients in stochastic collocation can be  
computed through inner products of weight tensors and data tensors (see section \ref{section_col} for details),
where the weight tensors are given but the data tensors are expensive to obtain. Instead of directly evaluating the 
expensive data tensor, tensor recovery here is to use a small number of
entries of the data tensor to recover the whole tensor \cite{Acar2010Scalable,gandy2011tensor,Liu2013Tensor}. A popular recovery strategy is
developed in \cite{Acar2010Scalable} based on canonical polyadic (CP) decomposition. In this recovery approach, the CP rank
of the underlying tensor needs to be known a priori, 
which limits its application to our PDE models where the corresponding CP ranks are not given. 
For this purpose, we develop a novel rank adaptive tensor recovery (RATR) approach, which do not require any 
prior information for the tensor ranks. Moreover, as this kind of tensor recovery procedure requires solving 
a non-convex optimization problem \cite{zhang2017bigdata}, initialization strategies for this kind of optimization
problem are crucial  for successful recovery. In our RATR approach, 
new efficient initializaiton strategies are proposed based on  a hierarchical  rank-one updating procedure, 
and their stability is theoretically proven in this work. 

The aim of this paper is to develop a systematic model reduction framework to curb this challenging
high-dimensional input-output problem, and our overall procedure is as follows. 
First, kPCA is conducted for the outputs, which gives their reduced-dimensional representations. 
After that, for each kPCA mode, RATR based stochastic collocation  is proposed to construct sparse gPC expansions for each 
kPCA mode. 
The inverse mapping method introduced in \cite{maza2011kernel,xing2016manifold}   is finally adopted to construct an 
overall estimates of the outputs in the high-dimensional space. To summarize, 
the main contributions of this work are three-fold: first,  stochastic collocation methods are reformulated with manifold learning
for high-dimensional outputs;
second, a novel rank adaptive tensor recovery (RATR) approach is proposed to recover tensors without knowing their ranks a priori; 
third, new efficient initialization strategies for RATR are proposed and their stability is analyzed.

The rest of the paper is organized as follows. In the next section, we formulate stochastic collocation methods for 
stochastic PDEs based on manifold learning. Details of tensors and standard tensor recovery approaches  
are introduced in section \ref{section_tduq}.  
Our main algorithms and analysis for RATR and the overall RATR-collocation surrogate are presented in section \ref{section_hdoutput}.
In section \ref{section_test}, we demonstrate the efficiency of our RATR-collocation approach for stochastic diffusion and 
incompressible flow problems. Finally section \ref{section_conclude} concludes the paper.

\section{Problem setting and stochastic collocation based on manifold learning}\label{section_problem}
Let $D$ denote a spatial domain (in $\dsR^2$ or $\dsR^3$) which is bounded, connected and with a
polygonal boundary $\partial D$,
and $x$ denote a spatial variable.
Let $\mb{\xi}$
be a vector which collects a finite number of random variables. 
The dimension of $\mb{\xi}$ is denoted by $d$, i.e., we write $\mb{\xi}=[\xi_1,\ldots,\xi_d]^T$. The probability density function 
of $\mb{\xi}$ is denoted by $\pi(\mb{\xi})$.
In this paper, we restrict our attention to the situation that $\mb{\xi}$ has a bounded and connected support.
Without loss of generality, we next assume the support of $\mb{\xi}$ to be  $I^d$ where $I:=[-1,1]$,
since any bounded connected domain in $\dsR^{d}$ can be mapped to $I^d$.
The physics of problems considered in this paper
are governed by 
a PDE over the spatial domain $D$ and
boundary conditions on the boundary $\partial D$.
This PDE problem is stated as: find
$u(x,\mb{\xi}): D\times I^d\to \dsR$, such that
\begin{align}
\pL\left(x,\mb{\xi};u\left(x, \mb{\xi} \right)\right)=f(x)  \qquad
&\forall \left(x,\mb{\xi} \right) \in D\times I^d,
\label{spdexi1}\\
\pB\left(x,\mb{\xi};u\left(x,\mb{\xi} \right)\right)=g(x) \qquad
&\forall \left(x, \mb{\xi} \right)\in \partial D\times I^d,
\label{spdexi2}
\end{align}
where $\pL$ is a partial differential operator and $\pB$ is a boundary
operator, both of which can have random coefficients. $f$ is the source function and $g$ 
specifies the boundary conditions. 
We also define output quantities of interest. 
For each realization of $\mb{\xi}$, if the deterministic version of \eqref{spdexi1}--\eqref{spdexi2}
is solved using  a high-fidelity numerical scheme (simulator), for example finite element 
and difference methods, a natural definition of the output is the discrete solution. 
A high-fidelity discrete solution is also called a {\em snapshot} and can be represented as  
$\mb{y} = [u(x^{(1)},\mb{\xi}), \ldots, u(x^{(N_h)}, \mb{\xi})]^T \in \dsR^{N_h}$, where $u(x^{(i)}, \mb{\xi}), i = 1, \ldots, N_h$ denotes the value of $u(x, \zxi)$ at a specified location on a spatial grid
and $N_h$ refers to the  spatial degrees of freedom.
The manifold consisting of all snapshots is denoted by $\zM\subset\dsR^{N_h}$, and it is assumed to be smooth.
A PDE simulator can be viewed as a mapping $\chi: I^d \rightarrow \zM$, where $I^{d}$ and $\zM \in \dsR^{N_h}$ are the 
input space and the output manifold respectively,
and we denote it as $\chi (\zxi)= \zy = [u(x^{(1)},\mb{\xi}), \ldots, u(x^{(N_h)},\mb{\xi})]^T\in \zM$ for 
an arbitrary realization of the input $\mb{\xi} \in I^{d}$. 

The goal of this study is to build surrogates for conducting uncertainty qualification (UQ) of the output $\zy$, given limited 
training data points $\zy^{(j)} = \chi(\mb{\xi}^{(j)})$ for $j=1:N_t$ where $N_t$ is the size of a training data set. 
We focus on  
the challenging situations that the input and the output are both high-dimensional. To make progress,
we reformulate the stochastic collocation surrogates \cite{maza2009adaptive, Xiu2005High} based on manifold learning and
tensor recovery quadrature.
Manifold learning gives a reduced dimension representation for the output space through 
kernel  principal component analysis (kPCA) and inverse mappings \cite{vms2016gpca, kwok2004preimagekpca, xing2016manifold}, and tenor recovery 
provides estimates of collocation coefficients associated with high-dimensional random parameters 
through exploiting low rank structures in these coefficients \cite{zhang2017bigdata}.  
The rest of this section is to discuss the manifold learning based collocation and the setting of tensor formulation, 
while detailed tensor recovery methods and our new rank adaptive schemes are presented in the next two sections. 

\subsection{Kernel principal component analysis (kPCA)}\label{sec_kpca}
To simplify the presentation, the given training data are denoted by $\zy^{(j)}=\chi (\zxi^{(j)})$ for $j=1,\ldots,N_t$.
Following \cite{vms2016gpca, scholkopf1998kpca,xing2016manifold},  
the  kernel principal component analysis (kPCA) proceeds through two steps: 
mapping the training data to a higher-dimensional feature space, and performing linear
principal component analysis (PCA) in the feature space. 
Denoting the feature space by $\zF$, we define a mapping $\Gamma: \zM \to \zF$, 
which maps each training data point  $\zy^{(j)}\in \zM$ to $\Gamma(\zy^{(j)})\in \zF$ for $j=1,\ldots,N_t$.
A covariance matrix of the mapped data is defined as
\begin{equation} \label{kpca_cov}
\zC_{\zF}:= \frac{1}{N_t}\sum\limits_{j=1}^{N_t} { \widetilde{\Gamma}\left(\mb{y}^{(j)}\right)\widetilde{\Gamma}\left(\mb{y}^{(j)}\right)^T },
\end{equation}
where $\widetilde{\Gamma}(\mb{y}^{(j)}) = \Gamma(\mb{y}^{(j)}) - \bar{\Gamma}$ 
and $\bar{\Gamma} = (1/N_t) \sum_{j=1}^{N_t} {\Gamma (\mb{y}^{(j)})}$. 
Eigenvectors of  $\zC_{\zF}$ can give a new basis to represent the mapped data,
and the eigenvectors associated with dominate eigenvalues can provide an effective reduced dimensional  
representation for them.

However, the mapping $\Gamma$ in practice is typically  defined implicitly through kernel functions,
and the eigenvectors of  $\zC_{\zF}$ are always replaced by eigenvectors of some centred kernel 
matrices. 
A kernel function in this setting is a mapping from $\dsR^{N_h}\times \dsR^{N_h}$ to 
$\dsR$, which is denoted by  $\mb{k}(\cdot,\cdot)$. The kernel matrix associated with  $\mb{k}(\cdot,\cdot)$ 
is denoted by $\zK\in \dsR^{N_h\times N_h}$, of which each  entry is defined as
\begin{equation*}
\mb{K}_{ij} = \mb{k}(\mb{y}^{(i)}, \mb{y}^{(j)}) \quad \text{for} \quad i, j = 1,2,\ldots, N_t.
\end{equation*}  
A standard choice of the kernel function is the Gaussian kernel
\begin{equation*}
\mb{k}(\mb{y}^{(i)}, \mb{y}^{(j)}) = \mathrm{exp} \left( -\frac{ \norm{\mb{y}^{(i)}-\mb{y}^{(j)}}{2}^2 }{ 2 \sigma_g^2} \right),
\end{equation*}
where $\sigma_g$ is the bandwidth parameter.  
The centred kernel matrix is defined as
\begin{equation*}
\widetilde{\mb{K}}: = \left(\mb{K} - \mb{1}_{ \frac{1}{N_t} } \mb{K} - \mb{K} \mb{1}_{\frac{1}{N_t}} + \mb{1}_{\frac{1}{N_t}} \mb{K} \mb{1}_{ \frac{1}{N_t} }\right),
\end{equation*}
where $\mb{1}_{\frac{1}{N_t}}$ denotes the matrix with all entries equaling to $1/{N_t}$. 

Let $\zaph=[\zaph_1,\ldots,\zaph_{N_t}]\in\dsR^{N_t\times N_t}$ collect 
the eigenvectors of $\widetilde{\mb{K}}$ associated with 
eigenvalues $\lambda_1>\lambda_2\ldots>\lambda_{N_t}$.
Basis functions of the mapped data in $\zF$ are 
defined as $\zomega_e:=\sum^{N_t}_{j=1}\widetilde{\alpha}_{je}\widetilde{\Gamma}(\mb{y}^{(j)})$,
where $\widetilde{\alpha}_{je}=\alpha_{je}/\sqrt{\lambda_e}$
and $\alpha_{je}$ is the $j$-th component of $\zaph_e$, for $e=1,\ldots,N_t$.
A mapped training point $\widetilde{\Gamma}(\mb{y}^{(j)})$ for $j=1,\ldots,N_t$ can be represented as 
$\widetilde{\Gamma}(\mb{y}^{(j)}) = \sum_{e=1}^{N_t} \zgamma_e(\mb{y}^{(j)}) \zomega_e$, where 
each coefficient $\zgamma_e(\mb{y}^{(j)})$ is computed through
\begin{eqnarray}\label{kpca_co}
\zgamma_e\left(\mb{y}^{(j)}\right)  = \sum_{i=1}^{N_t} \widetilde{\alpha}_{ie} \widetilde{\mb{K}}_{ij}.
\end{eqnarray}
To result in dimension reduction, the first $N_r$ dominant eigenvectors of $\widetilde{\mb{K}}$ 
are selected as the the principal components, with the criterion 
$(\sum^{N_r}_{e=1} \lambda_e)/(\sum^{N_t}_{e=1} \lambda_e)>tol_{\rm PCA}$ where $tol_{\rm PCA}$ is
a given tolerance. The basis functions associate with the principal components are
then $\zomega_e$, $e=1,\ldots,N_r$, and each mapped training data point 
can be approximated as
$\widetilde{\Gamma}(\mb{y}^{(j)})\approx \sum_{e=1}^{N_r} \zgamma_e(\mb{y}^{(j)}) \zomega_e$.

It can be seen that the overall procedure of kPCA defines a mapping 
from the output manifold $\mathcal{M}$ to the reduced feature space 
$\mathrm{span}\{ \zomega_1 ,  \ldots, \zomega_{N_r}  \}$.
We denote this mapping as $\kappa(\zy) = \sum_{e=1}^{N_r} \zgamma_{e}\left(\zy\right) \zomega_e$,  
where each coefficient $\zgamma_e(\zy)$ is obtained from \eqref{kpca_co}.
The basis $\{\zomega_e\}^{N_r}_{e=1}$ discussed above depends on the mapping $\widetilde{\Gamma}$ which are defined implicitly. 
Collecting these coefficients, a reduced output vector is denoted 
by $\tilde{\mgamma}(\zy):=[\zgamma_1(\zy),\ldots,\zgamma_{N_r}(\zy)]^T\in \dsR^{N_r}$ for any $\zy\in \zM$.
The manifold consisting of all reduced output vectors is denoted by $\zM_r$.
We next denote
$\mgamma(\zxi):= \tilde{\mgamma}(\zy) \in \zM_r\subset \dsR^{N_r}$.
In summary, each training data point $\zy^{(j)}=\chi (\zxi^{(j)})\in \zM$ is mapped to  
$\mgamma(\zxi^{(j)}):=[\gamma_1(\zxi^{(j)}),\ldots,\gamma_{N_r}(\zxi^{(j)})]^T\in \zM_r$ for $j=1,\ldots,N_t$.
We next construct stochastic collocation surrogates for each component of  $\mgamma(\zxi)$. 

\subsection{Stochastic collocation}\label{section_col}
For each $\gamma_{e}(\mb{\xi})$, $e = 1, \ldots, N_r$, a truncated generalized polynomial chaos (gPC) 
approximation \cite{ghanem2003sfem,xiu2002wiener} can be written as
\begin{equation} \label{eq_gpcfs}
\gamma_{e}(\mb{\xi}) \approx \gamma^{\rm gPC}_{e}(\mb{\xi}):= \sum_{\norm{\mb{i}}{1}=0}^{p}{ c_{e\mb{i}}\Phi_{\mb{i}}(\mb{\xi}) },
\end{equation}
where $\mb{i}=[i_1,i_2,\ldots,i_d]^T \in \Ne^{d}$ is a multi-index, $\norm{\mb{i}}{1} = i_1 + i_2 + \cdots + i_d$,
and $p$ is a given oder for truncation. Denoting the set of the multi-indices by $\Upsilon:=\{\mb{i} |\, \mb{i} \in \Ne^{d} \textrm{ and } \norm{\mb{i}}{1}=0,\ldots,p\}$, which implies that the number of basis function is $|\Upsilon| = (p+d)!/(p!d!)$. 
The basis functions $\{\Phi_{\mb{i}}(\mb{\xi}) |\, \mb{i} \in \Upsilon \}$
are orthogonal polynomials with respect to the density function $\pi(\mb{\xi})$ 
\begin{equation*}
\langle \Phi_{\mb{i}}(\mb{\xi}), \Phi_{\mb{i}^{\prime}}(\mb{\xi})   \rangle = \int_{I^d} {\Phi_{\mb{i}}(\mb{\xi})\Phi_{\mb{i}'}(\mb{\xi}) \pi(\mb{\xi}) d\mb{\xi}  } =  \delta_{\mb{i},\mb{i}'},
\end{equation*}  
where $\delta$ denotes the Kronecker delta function, i.e., $\delta_{\mb{i}, \mb{i}'}=1$ if $\mb{i}$
is the same as $\mb{i}'$ and $\delta_{\mb{i}, \mb{i}'}=0$ otherwise. 
Each basis function $\Phi_{\mb{i}}(\mb{\xi})$ can be expressed as the product of a set of univariate orthogonal polynomials, $\Phi_{\mb{i}}(\mb{\xi}) = \prod_{k=1}^{d} {\phi_{i_k}(\xi_k)}$,
with each univariate orthogonal polynomial  
defined through a three term recurrence \cite{gautschi1982generating},
\begin{equation*}
\begin{aligned}
\phi_{j+1}(\xi) = (\xi - \zeta_j) & \phi_{j}(\xi) - \tau_j \phi_{j-1}(\xi), \quad j = 1,2,\ldots, p-1, \\
\phi_0(\xi) = 0, & \quad \phi_{1}(\xi) = 1,
\end{aligned}
\end{equation*}
where $\zeta_j = {\int_{-1}^{1} \xi \pi_{\xi}(\xi) \phi_j^2(\xi) d\xi} / {\int_{-1}^{1} \pi_{\xi}(\xi) \phi_j^2(\xi) d\xi }$,  
$\tau_j = {\int_{-1}^{1} \pi_{\xi}(\xi) \phi_j^2(\xi) d\xi}/{\int_{-1}^{1} \pi_{\xi}(\xi) \phi_{j-1}^2(\xi) d
	\xi }$, and $\pi_{\xi}(\xi)$ is the marginal density function of $\xi$ ($\xi$ denotes a component of $\zxi$).

According to orthogonality of the gPC basis functions, 
the coefficients in \eqref{eq_gpcfs} can be computed through 
\begin{equation} \label{eq_coeff_gpc}
c_{e\mb{i}} = \int_{I^d} \gamma_e(\mb{\xi}) \Phi_{\mb{i}}(\mb{\xi}) \pi(\zxi)d\mb{\xi}. 
\end{equation}
This integral can be computed through quadrature rules, 
and following \cite{zhang2017bigdata} we focus on the tensor style quadrature.
Let $\{\xi^{(j)},w^{(j)}\}^n_{j=1}$ denote $n$ quadrature nodes and weights on the interval $[-1,1]$.
The quadrature form of \eqref{eq_coeff_gpc} is
\begin{equation} \label{eq_quad_coeff}
c_{e\mb{i}} = \sum_{1\leq j_1,\ldots,j_d\leq n} {\gamma_e(\mb{\xi}_{j_1\ldots j_d}) \Phi_{\mb{i}}(\mb{\xi}_{j_1\ldots j_d}) w_{j_1\ldots j_d} },
\end{equation}
where 
\begin{eqnarray}
\zxi_{j_1\ldots j_d} &=& \left[\xi_1^{(j_1)}, \xi_2^{(j_2)}, \ldots, \xi_d^{(j_d)}\right]^T,\label{eq_xij}\\
w_{j_1\ldots j_d} &=& w^{(j_1)} w^{(j_2)} \cdots w^{(j_d)},
\end{eqnarray}
for $1\leq j_1,\ldots,j_d\leq n$ are the nodes and the weights spanned by 
the tensor product of the one-dimensional quadrature rule. 

Following \cite{zhang2017bigdata}, the quadrature form \eqref{eq_quad_coeff} can be
formulated as a tensor inner product as follows.
For each $e=1,\ldots,N_r$, the  values $\gamma_e(\mb{\xi}_{j_1\ldots j_d})$ for $1\leq j_1,\ldots,j_d\leq n$
form a $d$-th order {\em data tensor} $\zX_e \in \dsR^{n \times \cdots \times n}$,
of which each entry is 
\begin{eqnarray}
\mb{\mathrm{X}}_e(j_1,\ldots, j_d)=\gamma_e(\mb{\xi}_{j_1\ldots j_d}).\label{eq_data_tensor}
\end{eqnarray}
For each multi index $\mb{i}$ with $\norm{\mb{i}}{1}\leq p$,
the values $\Phi_{\zi}(\zxi_{j_1\ldots j_d}) w_{j_1\ldots j_d}$ for $1\leq j_1,\ldots,j_d\leq n$
form a  $d$-th order {\em weight tensor} $\mb{\mathrm{W}}_{\zi}\in \dsR^{n \times \cdots \times n}$ with 
\begin{eqnarray}
\mb{\mathrm{W}}_{\zi}(j_1,\ldots,j_d)=\Phi_{\zi}(\zxi_{j_1\ldots j_d}) w_{j_1\ldots j_d}. \label{eq_weigth_tensor}
\end{eqnarray}
Defining
\begin{equation} \label{eq_reweight}
\hat{\mb{w}}_k^{(i_k)} = \left[\phi_{i_k}\left(\xi^{(1)}\right) w^{(1)}, \phi_{i_k}\left(\xi^{(2)}\right) w^{(2)}, \ldots, \phi_{i_k}\left(\xi^{(n)}\right) w^{(n)} \right]^T \in \dsR^{n}, \quad \textrm{for} \quad k=1,\ldots,d,
\end{equation} 
each entry of  $\mb{\mathrm{W}}_{\mb{i}}$ can be written as
\begin{equation*}
\mb{\mathrm{W}}_{\mb{i}}(j_1,j_2,\ldots,j_d) = \hat{\mb{w}}_1^{(i_1)}(j_1) \hat{\mb{w}}_2^{(i_2)}(j_2) \cdots \hat{\mb{w}}_d^{(i_d)}(j_d) \quad \text{for all} \quad 1 \leq j_k \leq n, \ k = 1, \ldots, d,
\end{equation*}
and $\mb{\mathrm{W}}_{\mb{i}}$ can be expressed as
\begin{equation}
\mb{\mathrm{W}}_{\mb{i}} = \hat{\mb{w}}_1^{(i_1)} \circ \hat{\mb{w}}_2^{(i_2)} \circ \cdots \circ \hat{\mb{w}}_d^{(i_d)} \quad \textrm{with} \quad \zi =[i_1, \ldots, i_d]^T,
\label{eq_weight_tensor}
\end{equation}	  
where ``$\circ$" is the vector outer product. 
With the notation above, 
the coefficient $c_{e\mb{i}}$ in \eqref{eq_quad_coeff} can be rewritten as the tensor inner product
\begin{equation} \label{eq_teninner_coeff}
c_{e\mb{i}} = \left<\mb{\mathrm{X}}_e, \mb{\mathrm{W}}_{\mb{i}} \right>,
\end{equation}
where the tensor inner product  \cite{Lathauwer2000A,Kolda2009Tensor} is defined as,
\begin{equation}
\left< \mb{\mathrm{X}}_e, \mb{\mathrm{W}}_{\mb{i}} \right> = \sum_{j_1}^{n}\sum_{j_2}^{n}\cdots \sum_{j_d}^{n} {\mb{\mathrm{X}}_e(j_1,j_2,\ldots,j_d) \mb{\mathrm{W}}_{\mb{i}}(j_1,j_2,\ldots,j_d)}.
\label{eq_tensor_inner0}
\end{equation} 
The tensor norm induced by this inner product is denoted by 
$\norm{\cdot}{} = \langle\cdot, \cdot \rangle^{1/2}$.
Details of tensor decomposition and recovery are discussed  in section \ref{section_tduq}
and section \ref{section_hdoutput}.
\subsection{Inverse mapping}\label{section_inverse}
After the gPC approximation \eqref{eq_gpcfs} for each $\gamma_e$, $e=1,\ldots,N_r$, is constructed through 
the above collocation procedure, the reduced output $\gamma(\zxi)=\tilde{\mgamma}(\zy)=\tilde{\mgamma}(\chi (\zxi))\in \zM_r$ 
for an arbitrary realization of $\zxi$
can be cheaply estimated through this gPC surrogate. However, our goal is to quantify the uncertainties in the
output $\zy=\chi (\zxi)\in \zM$, which requires an inverse mapping $\kappa^{-1}$ from the reduced 
output manifold $\mathcal{M}_r$ to the original output manifold $\mathcal{M}$.
Following \cite{kwok2004preimagekpca,xing2016manifold}, an inverse mapping can be obtained through 
an interpolation of neighbouring points in the training data set $\{\zy^{(1)},\ldots,\zy^{(N_t)}\}$.
That is, the Euclid distance between an arbitrary output $\zy\in \zM$ and
each training point $\zy^{(j)}$ (for $j=1,\ldots,N_t$) is first
computed through 
\begin{equation*}
d_{j} = \sqrt{-2 \sigma_g^2 \mathrm{log} \left(1 - 0.5 \hat{d}_{j}^2 \right)},
\end{equation*}
where $\hat{d}_{j} = 1 + \tilde{\mgamma}(\zy)^T \mb{K} \tilde{\mgamma}(\zy) - 2 \tilde{\mgamma}(\zy)^T \mb{\mathrm{k}}_{\mb{y}^{(j)}}$ are computed through the kernel function, $\mb{\mathrm{k}}_{\zy^{(j)}} = [ \mb{k}(\mb{y}^{(j)},\mb{y}^{(1)}), \ldots, \mb{k}(\mb{y}^{(j)}, \mb{y}^{(N_t)}) ]^T$ and $\mb{k}(\cdot.\cdot)$ is the given
kernel function.
The distances $\{d_1,\ldots,d_{N_t}\}$ are sorted next. Given a positive integer $N_n$, the indices with
the smallest $N_n$ distances are collected in a set $\mathcal{J}\subset \{1,\ldots,N_t\}$, i.e., $d_j\leq d_i$ for any $j,i=1,\ldots,N_t$ 
with $j\in \mathcal{J}$ but $i\notin \mathcal{J}$. After that, $\zy$ can be approximated as
\begin{equation} \label{eq_inverse_map}
\zy \approx \sum_{j\in \mathcal{J} }  \frac{d_{j}^{-1}}{ \sum\limits_{j\in \mathcal{J}} d_{j}^{-1} } \zy^{(j)}.
\end{equation}

\section{Tensor recovery based quadrature}\label{section_tduq}
It is clear that the main computational cost of the above collocation procedure based on manifold learning comes from 
generating the gPC expansion \eqref{eq_gpcfs} for each kPCA mode $e=1,\ldots,N_r$,
where evaluating each collocation coefficient requires computing a tensor inner product \eqref{eq_teninner_coeff}.
When the input parameter $\zxi$ is high-dimensional ($d$ is large), each data tensor $\zX_e\in \dsR^{n\times \cdots\times n}$ is large (with $n^d$ entries). 
Evaluating each entry of $\zX_e$ requires computing a snapshot (see \eqref{eq_data_tensor}), and it is therefore expensive to form these data tensors through 
computing snapshots for all entries. As an alternative, tensor recovery methods provide efficient estimates of tensors using a small number of exact entries. 
For forward UQ problems with a single output, when tensor ranks are given, a tensor recovery based collocation approach
is developed in  \cite{zhang2017bigdata}, which can be applied to construct the gPC approximation for each kPCA component \eqref{eq_gpcfs}. 
We here review this  tensor recovery based collocation approach and provide new detailed computational cost assessments. 
Since computation procedures for generating the gPC surrogates for each $\gamma_{e}(\mb{\xi})$, $e=1,\ldots,N_r$, are identical,  
we generically denote the data tensor $\zX_e$ defined in  \eqref{eq_data_tensor} as $\zXe$ in this section
 (i.e., the subscript $e$ is temporally ignored).


\subsection{Canonical polyadic (CP) decomposition}
Following the presentation in  \cite{Kolda2009Tensor}, the CP decomposition is reviewed as follows.
For a $d$-th order tensor $\zX \in \dsR^{n \times \cdots \times n}$, its CP decomposition is expressed as
\begin{equation} \label{cp_decomp}
\mb{\mathrm{X}} = \sum_{r=1}^{R} {\zv_1^{(r)} \circ \zv_2^{(r)} \circ \cdots \circ \zv_d^{(r)} }
\end{equation}
where  $\mb{v}_k^{(r)}\in\dsR^{n}$ for $k=1,\ldots,d$, $R$ is the CP rank of $\zX$, and  ``$\circ$" is the vector outer product.
The CP rank is defined as 
\begin{equation*}
R:= \mathrm{rank}(\zX) := \text{min} \ \left\{R' \ \left| \ \mb{\mathrm{X}} = \sum_{r=1}^{R'} {\mb{v}_1^{(r)} \circ \mb{v}_2^{(r)} \circ \cdots \circ \mb{v}_d^{(r)} }\right. \right\}.
\end{equation*} 
Figure \ref{fig_cp_tensor} shows a third-order tensor with its CP decomposition.
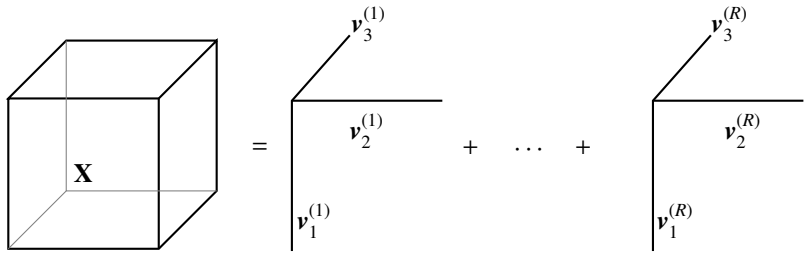
\begin{figure}[!htp]
	\centering
	\begin{tikzpicture}
	\draw[thick](2,2,0)--(0,2,0)--(0,2,2)--(2,2,2)--(2,2,0)--(2,0,0)--(2,0,2)--(0,0,2)--(0,2,2);
	\draw[thick](2,2,2)--(2,0,2);
	\draw[gray](2,0,0)--(0,0,0)--(0,2,0);
	\draw[gray](0,0,0)--(0,0,2);
	\draw(1,1,2) node{$\mb{\mathrm{X}}$}; 
	\draw(4.5,2.5,5) node{$=$};
	\draw(7.3,2.5,5) node{$+$};
	\draw(8.1,2.45,5) node{$\cdots$};
	\draw(8.8,2.5,5) node{$+$};
	\draw[thick](3,-0.8,0)--(3,1.2,0)--(3,1.3,-2);
	\draw(3.3,-0.4,0) node{$\mb{v}_1^{(1)}$};
	\draw[thick](3,1.2,0)--(5,1.2,0);
	\draw(4,0.8,0) node{$\mb{v}_2^{(1)}$};
	\draw(3.1,1.3,-2.4) node{$\mb{v}_3^{(1)}$};
	
	\draw[thick](7.8,-0.8,0)--(7.8,1.2,0)--(7.8,1.3,-2);
	\draw(8.1,-0.4,0) node{$\mb{v}_1^{(R)}$};
	\draw[thick](7.8,1.2,0)--(9.8,1.2,0);
	\draw(9,0.8,0) node{$\mb{v}_2^{(R)}$};
	\draw(7.9,1.3,-2.4) node{$\mb{v}_3^{(R)}$};
	\end{tikzpicture} 
	\caption{CP decomposition of a third-order tensor with rank $R$.  }
	\label{fig_cp_tensor}
\end{figure}  
For each $r=1,\ldots, R$, $\zv_1^{(r)} \circ \zv_2^{(r)} \circ \cdots \circ \zv_d^{(r)}$ in \eqref{cp_decomp} is called a rank-one component. 
For each $k=1,\ldots,d$,  the matrix $\mb{A}_k = [\mb{v}_k^{(1)}, \mb{v}_k^{(2)}, \ldots, \mb{v}_k^{(R)}] \in \dsR^{n \times R}$ is called the $k$th-order factor matrix. 
With these factor matrices, the CP decomposition \eqref{cp_decomp} can be rewritten as
\begin{equation}
\mb{\mathrm{X}} = [[ \mb{A}_1, \mb{A}_2, \ldots, \mb{A}_d ]].
\label{eq_cp_factor}
\end{equation}
From \eqref{eq_weight_tensor},  it is clear that each weight tensor is a rank-one tensor, and  we here generically denote it as 
$\mb{\mathrm{W}}:=\mb{w}_1 \circ \mb{w}_2 \circ \cdots \circ \mb{w}_d \in \dsR^{n \times \cdots \times n}$. 
Following \cite{Kolda2009Tensor}, the inner product of $\mb{\mathrm{X}}$, $\mb{\mathrm{W}}$   (for computing \eqref{eq_teninner_coeff})
can be efficiently computed as
\begin{eqnarray} 
\left< \mb{\mathrm{X}}, \mb{\mathrm{W}} \right> &=& \left<[[\mb{A}_1, \mb{A}_2, \ldots, \mb{A}_d]], \mb{w}_1 \circ \mb{w}_2 \circ \cdots \circ \mb{w}_d  \right> \nonumber\\
& =& \sum_{r=1}^{R} \left< \mb{v}_1^{(r)} \circ \mb{v}_2^{(r)} \circ \cdots \circ \mb{v}_d^{(r)}, \mb{w}_1 \circ \mb{w}_2 \circ \cdots \circ \mb{w}_d  \right> \nonumber\\
& = &\sum_{r=1}^R \sum_{j_1,j_2,\ldots,j_d}{\mb{v}_1^{(r)}(j_1)\mb{v}_2^{(r)}(j_2)\cdots \mb{v}_d^{(r)}(j_d) \mb{w}_1(j_1) \mb{w}_2(j_2)\cdots \mb{w}_d(j_d)} \nonumber\\
& = &\sum_{r=1}^R \left( \left(\sum_{j_1} \mb{v}_1^{(r)}(j_1) \mb{w}_1(j_1)\right) \left(\sum_{j_2} \mb{v}_2^{(r)}(j_2) \mb{w}_2(j_2)\right) \cdots \left(\sum_{j_d} \mb{v}_d^{(r)}(j_d) \mb{w}_d(j_d)\right) \right) \nonumber\\
& =& \sum_{r=1}^{R} \left( \prod_{k=1}^d \left<\mb{v}_k^{(r)}, \mb{w}_k \right> \right). \label{eq_teninner}
\end{eqnarray}
The cost for computing the tensor inner product $\langle \zX, \mb{\mathrm{W}} \rangle$ through \eqref{eq_teninner} is $O(dnR)$,
while the cost is $O(n^d)$ if the inner product is computed directly through its definition \eqref{eq_tensor_inner0}.

\subsection{Missing data tensor recovery} \label{sec_cpl1_reg}
Denoting the full index set for a $d$-th order tensor in $\dsR^{n \times \cdots \times n}$ as 
$\Theta_{\rm full}:=\{[j_1, j_2, \ldots, j_d]^T \in \Ne^d \, | \,  j_k=1,\ldots, n  \textrm{ for }  k = 1,\ldots, d\}$,
$\Theta\subset \Theta_{\rm full}$  denotes an observation index set, which consists of $|\Theta| \ll d^n$ indices
uniformly sampled from $\Theta_{\rm full}$. 
To numbering the elements in $\Theta$, the following sort operator  is introduced.
\begin{mydef}[Sort operator]\label{def:sort}
	For a given finite set $\Theta\subset \Ne^{d}$,  we first sort its elements in alphabetical order: for any two different indices $\mb{\hat{j}}=[\hat{j}_1,\ldots, \hat{j}_d]^T$ and 
	$\tilde{\mb{j}}=[\tilde{j}_1,  \ldots, \tilde{j}_d]^T$ belonging in $\Theta$, 
	$\mb{\hat{j}}$ is ordered before $\tilde{\mb{j}}$ if for the smallest number $k$ such that $\hat{j}_k \neq \tilde{j}_k$, we have $\hat{j}_k<\tilde{j}_k$.
	Then for any $\mb{j}\in \Theta$, $s(\Theta, \mb{j})\in\{1,\ldots,|\Theta|\}$ is defined to be the position of $\mb{j}$ in the sorted array.
\end{mydef}

A projection operator that takes tensor values over the observed indices  
are denoted by $\Pe_{\Theta}$: for any $d$-th order tensor $\zX\in \dsR^{n \times \cdots \times n}$,
$\Pe_{\Theta}(\zX):=[p_1,\ldots,p_{|\Theta|}]^T \in \dsR^{|\Theta|}$ where $p_{s(\Theta,\,\mb{j})}=\zX(j_1,\ldots , j_d)$ for all $\mb{j}\in\Theta$. 
Tensor recovery here is to find an approximation of the data tensor $\zXe$ based on the entries over the observation indices, i.e., 
$\Pe_{\Theta}(\zXe)$. Since it is assumed that $|\Theta| \ll d^n$, the cost for generating $\Pe_{\Theta}(\zXe)$ is small compared 
with the cost for generating the whole $\zXe$. 
When the CP rank of $\zXe$ (denoted by $R$) is given, Acar et al.\ \cite{Acar2010Scalable} formulate the tensor recovery problem
as the following optimization problem
\begin{equation} \label{cp_missvalue}
\begin{aligned}
\underset{\mb{\mathrm{X}}}{\textrm{min}} \quad & \frac{1}{2} \norm{\Pe_{\Theta}(\mb{\mathrm{X}})-\Pe_{\Theta}(\zXe)}{2}^2 \\
\text {s.t.} \quad  \mb{\mathrm{X}} &= \sum_{r=1}^{R} {\mb{v}_1^{(r)} \circ \mb{v}_2^{(r)} \circ \cdots \circ \mb{v}_d^{(r)} },
\end{aligned}
\end{equation}  
It is clear that evaluating $\Pe_{\Theta}({\zX})$ requires $O(|\Theta|dR)$ flops.

To take the sparsity of gPC coefficients \eqref{eq_gpcfs} into account (Cf. \cite{doostan2011nonsparse,yan2012stochastic} for
sparse gPC approximations), 
a $l_1$ regularized version of  \eqref{cp_missvalue} is formulated as follows  \cite{zhang2017bigdata}: 
\begin{equation} \label{cp_missvalue_reg}
\underset{\mb{A}_1, \mb{A}_2, \ldots, \mb{A}_d} {\text{min}} \quad J(\mb{A}_1, \mb{A}_2, \ldots, \mb{A}_d) = \frac{1}{2} \norm{\Pe_{\Theta} \left( [[\mb{A}_1, \mb{A}_2, \ldots, \mb{A}_d]]\right)-\Pe_{\Theta}\left(\zXe \right) }{2}^2 + \beta \sum\limits_{\norm{\mb{i}}{1}=0}^{p} |\left<[[\mb{A}_1, \mb{A}_2, \ldots, \mb{A}_d]], \mb{\mathrm{W}_{\mb{i}}} \right>|,     
\end{equation} 
where $\beta$ is a regularization parameter, $p$ is a given gPC order,  
and $\mb{A}_k$ for $k=1,\ldots,d$ are CP factor matrices of $\zX$ (see \eqref{eq_cp_factor}).
To solve (\ref{cp_missvalue_reg}), the alternative minimization iterative method can be  applied \cite{zhang2017bigdata}. 
Letting $\mb{A}_k^{(q)}$ for $k=1,\ldots,d$ be the CP factor matrices at $q$-th iteration step ($q\geq 0$ is an integer), each CP factor matrix $\mb{A}_k^{(q+1)}$ at iteration step $q+1$ is obtained through  
\begin{equation} \label{cpl1_altform}
\mb{A}_k^{(q+1)} = \text{arg} \underset{\mb{A}_k} {\text{ min}} \ J(\mb{A}_1^{(q+1)}, \ldots, \mb{A}_{k-1}^{(q+1)}, \mb{A}_k, \mb{A}_{k+1}^{(q)}, \ldots, \mb{A}_{d}^{(q)}),
\end{equation}
which leads to  a generalized lasso problem and is discussed next. 

\subsection{A generalized lasso problem}
Let $\mathrm{vec}(\mb{A})$ denote the vector form of a given matrix $\mb{A}$ (as implemented in the MATLAB function $\textbf{reshape}$).
Following the procedures discussed in \cite{zhang2017bigdata}, (\ref{cpl1_altform}) can be written as a generalized lasso problem 
\begin{equation} \label{glasso}
\mathrm{vec}\left(\mb{A}_k^{(q+1)}\right) = \text{arg} \underset{\mb{s}} {\text{ min}} \frac{1}{2} \norm{\mb{B}\mb{s} - \mb{b}}{2}^2 + \beta \norm{\mb{F}\mb{s}}{1},
\end{equation}
where $\mb{B}:=\mathcal{B}_{\Theta,k}([[\mb{A}_1^{(q+1)}, \ldots, \mb{A}_{k-1}^{(q+1)}, \mb{A}_k, \mb{A}_{k+1}^{(q)}, \ldots, \mb{A}_{d}^{(q)}]])$, $\mb{F}:=\mathcal{F}_{\Upsilon,k}([[\mb{A}_1^{(q+1)}, \ldots, \mb{A}_{k-1}^{(q+1)}, \mb{A}_k, \mb{A}_{k+1}^{(q)}, \ldots, \mb{A}_{d}^{(q)}]])$
with $\mathcal{B}_{\Theta,k}(\cdot)$ and $\mathcal{F}_{\Upsilon,k}(\cdot)$ defined as follows (in Definition \ref{def:B} and Definition \ref{def:F} respectively), and $\mb{b} = \Pe_{\Theta}(\zXe) \in \dsR^{|\Theta|}$.
\begin{mydef}\label{def:B}
	For a given observation index set $\Theta\subset \Ne^d$ and $k\in\{1,\ldots,d\}$,   
	the operator $\mathcal{B}_{\Theta,k}$ defines a mapping: 
	for a $d$-th order tensor $[[\mb{A}_1, \mb{A}_2, \ldots, \mb{A}_d]] \in \dsR^{n \times \cdots \times n}$ with rank $R$, entries of $\mb{B}:=\mathcal{B}_{\Theta,k}([[\mb{A}_1, \mb{A}_2, \ldots, \mb{A}_d]] )\in \dsR^{|\Theta| \times nR}$
	are zero except 
	\begin{equation} \label{eq_assemB}
	\mb{B}\left(s(\Theta, \mb{j}), (r-1)n+j_k\right) = 
	\prod\limits_{k^{\prime}=1}^{k-1} \mb{A}_{k^{\prime}}(j_{k^{\prime}},r) \prod\limits_{k^{\prime}=k+1}^{d} \mb{A}_{k^{\prime}}(j_{k^{\prime}},r),   \textrm{ for all } \mb{j} = [j_1,  \ldots, j_d]\in \Theta 
	\textrm{ and } r = 1,\ldots,R, \\
	\end{equation}
	where $s(\Theta,\cdot)$ is the sort operator defined in Definition \ref{def:sort}.
\end{mydef}

\begin{myprop} \label{prop_operator}
	Let $\Theta$ be an observation index set, and $[[\mb{A}_1, \mb{A}_2, \ldots, \mb{A}_d]] \in \dsR^{n \times  \cdots \times n}$ be a $d$-th order tensor with rank $R$. Letting $\mb{A}_{i \delta} := [\mb{A}_i, \mb{\delta a}] \in \dsR^{n \times (R+1)}$ for $i = 1, 2, \ldots, d$,
for any $k\in\{1,\ldots,d\}$,  we have that
	\begin{equation}
	\mathcal{B}_{\Theta, k}([[\mb{A}_{1 \delta}, \mb{A}_{2 \delta}, \ldots, \mb{A}_{d \delta}]]) = [
	\mathcal{B}_{\Theta, k}([[\mb{A}_1, \mb{A}_2, \ldots, \mb{A}_d]]), \mathcal{B}_{\Theta, k}([[\mb{\delta a}, \mb{\delta a}, \ldots, \mb{\delta a}]])
	].
	\end{equation}
\end{myprop}
\begin{proof}
	This proposition is straightforward since the matrix $\mb{B}$ is constructed by \eqref{eq_assemB}.
\end{proof}

\begin{mydef}\label{def:F}
	Denoting the set of the multi-indices in \eqref{eq_gpcfs} as $\Upsilon:=\{\mb{i} |\, \mb{i} \in \Ne^{d} \textrm{ and } |\mb{i}|=0,\ldots,p\}$,
	where $p$ is a given gPC order.
	The operator $\mathcal{F}_{\Upsilon, k}$ defines a mapping: 
	for a $d$-th order tensor $[[\mb{A}_1, \mb{A}_2, \ldots, \mb{A}_d]] \in \dsR^{n \times \cdots \times n}$ with rank $R$ and $k\in\{1,\ldots,d\}$, entries of 
	$\mb{F}:=\mathcal{F}_{\Upsilon, k}([[\mb{A}_1, \mb{A}_2, \ldots, \mb{A}_d]] )\in \dsR^{|\Upsilon| \times nR}$
	are specified as 
	\begin{eqnarray} \label{eq_assemF}
	\mb{F}(s(\Upsilon, \mb{i}), (r-1)n+m_k) = \hat{\mb{w}}_k^{(i_k)}(m_k) \prod\limits_{k^{\prime}=1}^{k-1} \left< \mb{A}_{k^{\prime}}(:,r), \hat{\mb{w}}_{k^{\prime}}^{(i_{k^{\prime}})} \right> \prod\limits_{k^{\prime}=k+1}^{d} \left< \mb{A}_{k^{\prime}}(:,r), \hat{\mb{w}}_{k^{\prime}}^{(i_{k^{\prime}})} \right>, 
	\end{eqnarray}
	for all $ \mb{i} = [i_1, i_2, \ldots, i_d]^T \in \Upsilon$, $m_k = 1, \ldots, n$ and $ r = 1, \ldots, R, $
	where $s(\Upsilon,\cdot)$ is the sort operator defined in Definition~\ref{def:sort} and each $\hat{\mb{w}}_k^{(i_k)}$ is defined in \eqref{eq_reweight}. 
\end{mydef}

The operation numbers to construct $\mb{B}$ and $\mb{F}$ 
are $O(|\Theta| Rd)$ and $O(|\Upsilon|Rn^2d)$ respectively.  
The generalized lasso problem \eqref{glasso} can be solved 
by alternating direction method of multipliers (ADMM) \cite{boyd2011distributed} as follows. 
First, \eqref{glasso} is rewritten as  
\begin{equation} \label{glasso_admm}
\begin{aligned}
&\underset{\mb{s}, \mb{t}} {\text{ min}} \quad  \frac{1}{2} \norm{\mb{B}\mb{s} - \mb{b}}{2}^2 + \beta \norm{\mb{t}}{1} \\
& \ \text{s.t.} \quad \mb{F} \mb{s} = \mb{t}.
\end{aligned}
\end{equation}
The augmented Lagrangian function of (\ref{glasso_admm}) is given by
\begin{equation*}
L(\mb{s},\mb{t},\mb{z},\varrho) = \frac{1}{2} \norm{\mb{B}\mb{s} - \mb{b}}{2}^2 + \beta \norm{\mb{t}}{1} + \left< \mb{z}, \mb{F}\mb{s} - \mb{t} \right> + \frac{\varrho}{2} \norm{\mb{F} \mb{s} - \mb{t}}{2}^2,
\end{equation*}
where $\mb{z} \in \dsR^{|\Upsilon|}$ is the Lagrange multiplier, and $\varrho > 0$ is an augmented Lagrange multiplier. 
Following  \cite{bertsekas1999nonlinear}, the optimization problem \eqref{glasso_admm} can be solved as
\begin{equation*}
\begin{aligned}
\mb{s}^{(i+1)} &= \text{arg} \underset{\mb{s}} {\text{ min}} \ L(\mb{s},\mb{t}^{(i)},\mb{z}^{(i)}, \varrho^{(i)}) \\
\mb{t}^{(i+1)} &= \text{arg} \underset{\mb{t}} {\text{ min}} \  L(\mb{s}^{(i+1)},\mb{t},\mb{z}^{(i)}, \varrho^{(i)}) \\
\mb{z}^{(i+1)} &= \mb{z}^{(i)} + \varrho^{(i)}(\mb{F}\mb{s}^{(i+1)} - \mb{t}^{(i+1)}). 
\end{aligned}
\end{equation*}

Following \cite{boyd2011distributed}, details of the ADMM algorithm for (\ref{glasso_admm}) are summarized Algorithm \ref{alg_glasso}, 
and the soft-thresholding operator  $\Se_{\frac{\beta}{\varrho}}$ on line 4 of  Algorithm \ref{alg_glasso} is defined as   
\begin{equation*}
\Se_{\frac{\beta}{\varrho}} (h) = \begin{cases}
h - \frac{\beta}{\varrho} \ & \text{when} \quad h \geq \frac{\beta}{\varrho}, \\
0 \ & \text{when} \quad |h| < \frac{\beta}{\varrho}, \\
h + \frac{\beta}{\varrho}\ & \text{when} \quad h \leq -\frac{\beta}{\varrho}.
\end{cases}
\end{equation*} 
The cost of using Algorithm \ref{alg_glasso} to solve \eqref{glasso_admm} is analyzed as follows:
\begin{itemize}
	\item updating $\mb{s}^{(k+1)}$ line 3 of Algorithm \ref{alg_glasso} requires a matrix inversion and matrix-vector products, of 
	which the total cost is $O(n^3R^3 + n^2R^2(|\Theta| + |\Upsilon|))$.
	\item updating the soft-thresholding operator and $\mb{z}^{(k+1)}$ requires $O(|\Upsilon|nR)$ operations.
\end{itemize}
Therefore, the total cost of Algorithm \ref{alg_glasso} is 
$C_{Alg\ref{alg_glasso}}: = O(n^3R^3 + n^2R^2(|\Theta|+|\Upsilon|)) + O(|\Upsilon|Rn^2d + |\Theta| Rd)$. 

\begin{algorithm}
	\caption{ADMM for generalized lasso \cite{boyd2011distributed}}
	\label{alg_glasso}
	\begin{algorithmic}[1]
		\Require $\mb{B}, \mb{F}, \mb{b}, \beta, \nu \geq 1 \ (\text{for augment Lagrange multiplier})$
		\State Initialize $\varrho^{(0)}, \mb{s}^{(0)}, \mb{t}^{(0)} = \mb{F}\mb{s}^{(0)},$ and $\mb{z}^{(0)}$, $i = 0$
		\While {not converged}
		\State $\mb{s}^{(i+1)} = (\mb{B}^T \mb{B} + \varrho \mb{F}^T \mb{F} )^{\dagger}(\mb{B}^T \mb{b} + \varrho \mb{F}^T \mb{t}^{(i)} - \mb{F}^T \mb{z}^{(i)})$ 
		\State $\mb{t}^{(i+1)} = \Se_{\frac{\beta}{\varrho}}(\mb{F}\mb{s}^{(i+1)}+\frac{1}{\varrho}\mb{z}^{(i)})$ (element-wise)
		\State $\mb{z}^{(i+1)} = \mb{z}^{(i)} + \varrho (\mb{F} \mb{s}^{(i+1)} - \mb{t}^{(i+1)})$
		\State $\varrho^{(i+1)} = \varrho^{(i)} \nu$
		\State $i = i + 1$
		\EndWhile
		\Ensure $\mb{s}^* = \mb{s}^{(i)}$ and $\mb{A}$ (the matrix form of $\mb{s}^*$).
	\end{algorithmic}
\end{algorithm}

\begin{algorithm}
	\caption{Fixed-rank tensor recovery \cite{zhang2017bigdata}}
	\label{alg_ftr}
	\begin{algorithmic}[1]
		\Require CP rank $R$, initial rank $R$ factor matrices $\mb{A}_k^{(0)}$ for $k = 1, \ldots, d$, $\Theta$, $\Upsilon$ and $\Pe_{\Theta}(\zXe)$.
		\State Let $q=0$ and $\mb{b} = \Pe_{\Theta}(\zXe)$.
	         \State Initialize   $\epsilon_{\text{factor}}\geq \delta,  \epsilon_{J}\geq \delta, \epsilon_{c}\geq \delta $.
		\While {$\epsilon_{\text{factor}} \geq \delta$, 
			$\epsilon_{J} \geq \delta$ or $\epsilon_{c} \geq \delta$} 
		\State $ q = q + 1$.     	  
		\For {$k = 1:d$}
		\State $\mb{B}:=\mathcal{B}_{\Theta,k}([[\mb{A}_1^{(q+1)}, \ldots, \mb{A}_{k-1}^{(q+1)}, \mb{A}_k, \mb{A}_{k+1}^{(q)}, \ldots, \mb{A}_{d}^{(q)}]])$.
		\State $\mb{F}:=\mathcal{F}_{\Upsilon,k}([[\mb{A}_1^{(q+1)}, \ldots, \mb{A}_{k-1}^{(q+1)}, \mb{A}_k, \mb{A}_{k+1}^{(q)}, \ldots, \mb{A}_{d}^{(q)}]])$.
		\State Obtain factor matrix $\mb{A}_k^{(q+1)}$ by Algorithm \ref{alg_glasso}.
		\EndFor
		\State Compute $\epsilon_{\text{factor}},  \epsilon_{J}, \epsilon_{c} $. 
		\EndWhile
		\Ensure CP factor matices $\mb{A}_k = \mb{A}_k^{(q)}$ for $k = 1,2,\cdots,d$ and the recovered tensor $\zX=[[\mb{A}_1, \ldots, \mb{A}_d]]$.
	\end{algorithmic}
\end{algorithm}     

The stopping criterion for the overall optimization problem \eqref{cp_missvalue_reg} is specified through three parts in \cite{zhang2017bigdata}: the relative changes of factor matrices, objective function values, and gPC coefficients. 
The relative change of factor matrices between iteration step $q$ and $q+1$ is defined as
$\epsilon_{\text{factor}}: = {(\sum_{k=1}^d {\|{\mb{A}_k^{(q+1)} -\mb{A}_k^{(q)}}\|_F^2)^{1/2}} }/
( \sum_{k=1}^{d}  \|\mb{A}_k^{(q)} \|_F^2 )^{1/2}$ where $\|\cdot\|_F$ denotes the matrix Frobenius norm. 
The complexity of computing $\epsilon_{\text{factor}}$ is $O(dnR)$. 
Similarly, the relative changes of objective function values and gPC coefficients are defined as
$\epsilon_{J}: = |J^{(q+1)}-J^{(q)}|/|J^{(q)}|$ and $\epsilon_{c}:= \|c^{(q+1)}-c^{(q)}\|_1/\|c^{(q)}\|_1$ respectively,
where $\|\cdot\|_1$ denotes the vector $l_1$ norm and $c^{(q)}$ collects the collocation coefficients \eqref{eq_teninner_coeff} 
obtained with $[[\mb{A}^{(q)}_1, \ldots, \mb{A}^{(q)}_d]]$.
Since evaluating the objective function of \eqref{cp_missvalue_reg}  includes computing the projection $\Pe_{\Theta}$ 
and the tensor inner products,  the cost of computing $\epsilon_J$ and $\epsilon_c$ are $O(|\Theta|dR + |\Upsilon|ndR)$.
For a given tolerance $\delta$, 
the optimization iteration for \eqref{cp_missvalue_reg} terminates if $\epsilon_{\text{factor}} < \delta$, 
$\epsilon_{J} < \delta$ and $\epsilon_{c} < \delta$. The details for solving \eqref{cp_missvalue_reg} is 
summarized Algorithm \ref{alg_ftr}, which is proposed in \cite{zhang2017bigdata}. 
The total cost of Algorithm \ref{alg_ftr} is  $ C_{Alg\ref{alg_ftr}}: = dC_{Alg1} + O((n+|\Theta|+|\Upsilon|n)dR)$.

\section{Rank adaptive tensor recovery for stochastic collocation}\label{section_hdoutput}
Our goal is to perform  uncertainty propagation from a high-dimensional random input vector $\zxi$
to the snapshot $\mb{y} = [u(x^{(1)},\mb{\xi}), \ldots, u(x^{(N_h)}, \mb{\xi})]^T \in \dsR^{N_h}$ which is also high-dimensional.  
For this purpose, we develop a novel rank adaptive tensor recovery collocation 
(RATR-collocation) approach in this section. We first present our new general  rank adaptive tensor recovery (RATR) algorithm
for a general tensor, 
and then analyze its stability. 
After that, we present our main 
algorithm for this high-dimensional forward UQ problem.

\subsection{Rank adaptive tensor recovery (RATR)}\label{sec_ratr}
As discussed in section \ref{section_tduq}, the standard tensor recovery quadrature requires a given rank of the data tensor,
which causes difficulties for problems where the tensor rank is not given a priori. 
Especially in our setting, tensor recovery quadratures are applied to compute the gPC coefficients for each kPCA mode \eqref{eq_gpcfs},
where the ranks of data tensors \eqref{eq_data_tensor} are not given and the data tensor ranks associated with  
different kPCA modes can be different. 
To address this issue, we develop a new rank adaptive tensor recovery (RATR) approach.

Our idea is that, starting with setting the CP rank $R=1$, we gradually increase the CP rank until the recovered tensor $\zX$ 
approximates the exact tensor $\zXe$ well. To measure the quality of the recovered tensor,  
the following quantity of error is introduced 

\begin{equation} 
\varepsilon_{\Theta^{\prime}} (\mb{\mathrm{X}}) = \frac{\norm{\Pe_{\Theta^{\prime}} (\mb{\mathrm{X}}) - \Pe_{\Theta^{\prime}} (\zXe)}{2}}{\norm{\Pe_{\Theta^{\prime}} (\zXe)}{2}},
\label{eq_cv_cpl2}
\end{equation}
where $\Theta$ is the observation index set,  $\Theta'$ is a  validation index set (see \cite{arlot2010cvsurvey} for validation) randomly sampled from 
$\Theta_{\rm full}:=\{[j_1, j_2, \ldots, j_d]^T \in \Ne^d \, | \,  j_k=1,\ldots, n  \textrm{ for }  k = 1,\ldots, d\}$, 
such that $ \Theta'  \cap \Theta = \emptyset $  and $|\Theta'|<|\Theta_{\rm }|\ll d^n$. 
Evaluating the relative error $\varepsilon_{\Theta^{\prime}}(\zX)$ requires $O((|\Theta|+|\Theta^{\prime}|)dR)$ flops, 
which is discussed in section \ref{sec_cpl1_reg}. 

Since the optimization problem \eqref{cp_missvalue_reg} is non-convex, initial factor matrices in Algorithm \ref{alg_ftr} need
to be properly chosen. We provide a detailed analysis of the initialization strategy 
in section \ref{section_hdoutput_initana}.
Here, supposing the tensor  $\zX^{(R)}=[[\mb{A}_1, \ldots, \mb{A}_d]]$ is obtained, where $\{\mb{A}_1, \ldots, \mb{A}_d\}$ is 
the solution of \eqref{cp_missvalue_reg} with rank $R\geq1$,  
we consider one higher rank, i.e., $R+1$. While an analogous approach for tensor completion using tensor train decomposition 
can be found in \cite{MS2016Tensor}, we here focus on CP decomposition and give the following scheme of rank-one update. The initial  factor matrices for rank $R+1$ is set to the rank-one updates 
of the factor matrices of $\zX^{(R)}$, i.e.,
\begin{equation} \label{eq_roupdate}
\mb{A}^{(0)}_k = [
\mb{A}_k , \mb{\delta a}
], \quad \text{for }  k = 1,\ldots,d,
\end{equation}
where $\mb{\delta a} \in \dsR^{n}$ is a random perturbation vector and $\mb{A}_1, \ldots, \mb{A}_d$ are
the factor matrices of $\zX^{(R)}$. 
With these new initial factor matrices, the  recovered tensor 
$\zX^{(R+1)}$
for rank $R+1$ are obtained using Algorithm \ref{alg_ftr}.
To assess the progress obtained through this update  of the CP rank,  the difference between the recovery errors of $\zX^{(R+1)}$
and $\zX^{(R)}$ are assessed through 
${\Delta \varepsilon_{\Theta^{\prime}} := \varepsilon_{\Theta'}(\zX^{(R)}) - \varepsilon_{\Theta'}(\zX^{(R+1)})},$
where $\varepsilon_{\Theta'}(\cdot)$ is computed through \eqref{eq_cv_cpl2}. 
After that, we update the CP rank $R:=R+1$, and the above procedure is repeated until $\Delta \varepsilon_{\Theta'}<0$, i.e.,  $\varepsilon_{\Theta'}(\zX^{(R+1)})> \varepsilon_{\Theta'}(\zX^{(R)})$.

Details of our RATR method are presented in Algorithm \ref{alg_rankadaptive_cpl1}. 
The initial  rank-one matrices $\mb{A}_1^{(0)},\ldots,\mb{A}_d^{(0)}$ in the input are discussed in the next section.
The other inputs are the observation index set $\Theta$, the validation index set $\Theta'$, and entries of the
data tensor (see \eqref{eq_data_tensor}) on these index sets ($\Pe_{\Theta}(\zXe)$ and $\Pe_{\Theta'}(\zXe)$). 
To start the {\bf While} loop of this algorithm, $\Delta \varepsilon_{\Theta'}$ is initially set to an arbitrary number that is larger than $0$ on line 1. The output of this algorithm gives an estimation of the data tensor and its estimated  rank.  
The cost of this algorithm is $C_{Alg\ref{alg_ftr}} + O(|\Theta^{\prime}|dR)$.

\begin{algorithm}
	\caption{Rank adaptive tensor recovery (RATR)}
	\label{alg_rankadaptive_cpl1}
	\begin{algorithmic}[1]
		\Require $\Theta, \Theta^{\prime}, \Upsilon, \Pe_{\Theta}(\zXe), \Pe_{\Theta'}(\zXe)$, and initial rank-one matrices $\mb{A}_k^{(0)}$ 
		for $k = 1,\cdots,d$.
		\State Initialize the CP rank $R:=1$ and set $\Delta \varepsilon_{\Theta'}>0$.
		\State Run Algorithm \ref{alg_ftr} to obtain $\zX^{(1)}$.
		\State Compute $\varepsilon_{\Theta'} \left(\zX^{(1)}\right)$ by (\ref{eq_cv_cpl2}). 
		\While {$\Delta \varepsilon_{\Theta'} \geq 0$}
		\State Initialize $\mb{A}^{(0)}_k:= [\mb{A}_k , \mb{\delta a}]$ for $k = 1,\ldots,d$,
		where each $\mb{A}_k$ is a factor matrix of $\zX^{(R)}$. 
		\State Update the CP rank $R:=R+1$.
		\State Run Algorithm \ref{alg_ftr} to obtain $\zX^{(R)}$.
		\State Compute $\varepsilon_{\Theta'} \left(\zX^{(R)}\right)$ by (\ref{eq_cv_cpl2}).
		\State Compute the relative change in errors 
		$     {\Delta \varepsilon_{\Theta^{\prime}} := \varepsilon_{\Theta'}\left(\zX^{(R-1)}\right) - \varepsilon_{\Theta'}\left(\zX^{(R)}\right)}.
		$
		\EndWhile
		\State Let $\zX := \zX^{(R-1)}$.
		\State Let $R := R-1$.
		\Ensure the CP rank $R$ and the recovered tensor $\zX$. 
	\end{algorithmic}
\end{algorithm}

\subsection{Numerical stability analysis for RATR} \label{section_hdoutput_initana}
While the tensor recovery problem \eqref{cp_missvalue_reg} is a non-convex optimization problem,
the initial guesses for the factor matrices need to be chosen properly. As discussed in 
section \ref{section_tduq}, \eqref{cp_missvalue_reg} is solved using the alternative minimization iterative method, 
where the generalized lasso problem \eqref{glasso} needs to be solved at each iteration step. 
As studied in  \cite{ywt_admmlasso_convergence}, 
\eqref{glasso} becomes ill-defined if  $\mb{B}$
is ill-conditioned. 
Therefore, a necessary condition for the initial factor matrices in \eqref{cp_missvalue_reg} is that 
the resulting matrix $\mb{B}$ (see Definition \ref{def:B}) needs to be well-conditioned. 
In this section, we first show that  if the initial factor matrices 
are sampled through some given distributions, the condition number of $\mb{B}$ is bounded with high 
probability for the case of rank $R=1$. Next, we focus on the rank-one update procedure 
in our RATR approach (on line 5 of Algorithm \ref{alg_rankadaptive_cpl1}),
and show that the condition number of  $\mb{B}$ in \eqref{glasso}  associated this update procedure 
is bounded under certain conditions. 
We begin our analysis with introducing the following definitions.  

\begin{mydef}[Uniform observation index set]\label{def:uniform}
	An observation index set $\Theta$ is uniform if and only 
	if $|\{\mb{j}| \mb{j}=[j_1,\ldots,j_d]\in\Theta \textrm{ and } j_k = i \}| = |\Theta|/n$ for each $i = 1, \ldots, n$ and $k = 1, \ldots, d$, 
	where $|\cdot|$ denotes  the size of a set.  
\end{mydef}

\begin{mydef}[$d$-th order ratio with $m$ degrees of freedom] \label{def:ratio}
	Let $\psi_1, \ldots, \psi_{d-1}$ form a random sample from a given distribution $\Pe$ and for a given positive integer $d$, 
	let $\Psi:= \prod_{j=1}^{d-1} {\psi_j}$. 
	For a given positive integer $m$, let $\Psi_1, \ldots, \Psi_m$ form a random sample from the sampling distribution of $\Psi$, 
	and let $\Xi := \sum_{k=1}^{ m } {\Psi_k ^2}$. 
	The $d$-th order ratio with $m$ degrees of freedom of $\Pe$ is 
	\begin{eqnarray}\label{eq:mu}
	\mu = \frac{\sqrt{\mathrm{Var}(\Xi)}}{\Ee(\Xi)},
	\end{eqnarray}
	where $\Ee(\Xi)$ and $\mathrm{Var}(\Xi)$ are the expectation and the variance of $\Xi$ respectively.
\end{mydef}

Note that $\mu$ in \eqref{eq:mu} is typically called the coefficient of variation of $\Xi$ \cite[p.~845-846]{degroot2012probability}.  
\begin{theorem} \label{Theorem_var_init}
	Let $[[\mb{A}_1, \ldots, \mb{A}_d]] \in \dsR^{n \times \cdots \times n}$ be a $d$-th order rank-one tensor 
	and $\Theta$ be an observation index set. Suppose that for $i = 1, \ldots, d$, all entries of $\mb{A}_i$ form a random sample from distribution $\Pe$.
	Assume that the observation index set $\Theta$ is uniform. 
	For a given constant $\Cth>1$, if the $d$-th order ratio with $|\Theta|/n$ degrees of freedom of $\Pe$ 
	satisfies $\mu<1$, $Q: = (\Cth-1)/(\Cth\mu+\mu) > 1$, then the condition number of $\mb{B} =\mathcal{B}_{\Theta,k}([[\mb{A}_1, \mb{A}_2, \ldots, \mb{A}_d]])$ for $k=1,2,\ldots,d$ (see Definition \ref{def:B}) 
	satisfies  $\mathrm{cond}^2(\mb{B})  \leq \Cth $ with probability at least $1 - 1/Q^2$.
\end{theorem}
\begin{proof}
	Since the entries of $\mb{B}$ are zero except $\mb{B}\left(s(\Theta, \mb{j}), j_k\right)$  for $\mb{j} = [j_1,  \ldots, j_d]^T\in \Theta$ 
	(see Definition \ref{def:B}) where $s(\Theta, \cdot)$ is the sort operator,  
	each row of $\mb{B}$ can have at most one nonzero entry. 
	Therefore, $\mb{B}^T\mb{B}$ is a diagonal matrix. 
	
	Since  all entries of $\mb{A}_i$ form a random sample from distribution $\Pe$ for $i = 1, \ldots, d$,
	based on  Definition \ref{def:B}, 
	for $\mb{j} = [j_1,  \ldots, j_d]^T\in \Theta$, $k=1,\ldots,d$ and $j_k=1,\ldots,n$,
	we can express $\mb{B}\left(s(\Theta, \mb{j}), j_k\right):=\Psi:= \prod_{j=1}^{d-1} {\psi_j}$, where 
	$\psi_1, \ldots, \psi_{d-1}$ form a random sample from $\Pe$. 
	We next denote $\Psi_{s(\Theta, \mb{j}),j_k}:=\mb{B}\left(s(\Theta, \mb{j}), j_k\right)$,
	where  $\Psi_{s(\Theta, \mb{j}),j_k}$ form a random sample from the sampling distribution of $\Psi$ 
	for $\mb{j} = [j_1,  \ldots, j_d]^T\in \Theta$, $k=1,\ldots,d$ and $j_k=1,\ldots,n$. 
	Therefore, with the assumption that  $\Theta$ is uniform, 
	we have $\mb{B}^T\mb{B}=\mathrm{diag}(\Xi_1,\ldots,\Xi_n)$ with
	$\Xi_{i}:=\sum_{j=1}^{|\Theta|/n } {(\Psi_j)^2}$ for $i=1,\ldots,n$,
	where  $\Psi_{j}$ for $j=1,\ldots,\Theta|/n$ 
	form a random sample from the sampling distribution of $\Psi$. 

	According to the Chebyshev inequality,
	\begin{equation*}
	\mathrm{Prob} \left( |\Xi - \Ee(\Xi)| > Q \sqrt{\mathrm{Var}(\Xi)} \right ) \leq \frac{\mathrm{Var}(\Xi)}{Q^2 \mathrm{Var}(\Xi) },
	\end{equation*}
	which is equivalent to 
	\begin{equation}\label{theorem1-1}
	\mathrm{Prob} \left( \Ee(\Xi) - Q \sqrt{\mathrm{Var}(\Xi)} \leq \Xi \leq \Ee(\Xi) + Q \sqrt{\mathrm{Var}(\Xi)}  \right ) \geq 1 - \frac{1}{Q^2}.
	\end{equation}
	Using  $\mathrm{diag}(\Xi_1,\ldots,\Xi_n)=\mb{B}^T\mb{B}$ gives 
	\begin{equation}\label{theorem1-2}
	\mathrm{cond}^2(\mb{B})  = \frac{\mathrm{max} \{\Xi_1, \ldots, \Xi_n\}}{\mathrm{min} \{\Xi_1, \ldots, \Xi_n\}} 
	\leq \frac{\Ee(\Xi) + Q \sqrt{\mathrm{Var}(\Xi)}}{\Ee(\Xi) - Q \sqrt{\mathrm{Var}(\Xi)}} 
	= \frac{1+Q\mu}{1-Q\mu}.
	\end{equation}
	Noting that $Q := (\Cth-1)/(\Cth\mu+\mu) > 1$ which gives $\Cth=(1+Q\mu)/(1-Q\mu)$, combining \eqref{theorem1-1} and \eqref{theorem1-2} establishes
	\begin{equation*}
	\mathrm{cond}^2(\mb{B})  \leq \Cth
	\end{equation*}
	with probability at least $1 - 1/Q^2$. 
\end{proof}

\begin{table}[!htp]
	\caption{Examples of the $d$-th order ratio with $|m|$ degrees of freedom for several standard distributions
	with $d=48$ and $m=33$ and the corresponding values of $\mathrm{cond}(\mb{B})$.} 
	\centering	
	\begin{tabular}{cccc}  
		\hline
		Distribution $\Pe$  & Estimated $\mu$  & Average $\mathrm{cond}(\mb{B})$ \\ \hline
		$\mathrm{U}(1, 2)$ & 0.0394 & 1.9031 \\
		$\mathrm{N}(9, 0.1)$ & 0.0242 & 1.1491 \\
		$\mathrm{U}(1, 3)$ & 0.1051 & 3.7348 \\
		$\mathrm{N}(9, 0.5)$ & 0.1390 & 1.1830 \\	
		\hdashline
		$\mathrm{U}(0, 1)$ & 231.4262 & $2.27\times 10^{3}$\\ 
		$\mathrm{N}(0, 1)$ & 163.5804 & $1.89\times 10^{4}$ \\ 
		\hline
	\end{tabular}
	\label{table_ratio_mu}
\end{table}

The conditions in Theorem \ref{Theorem_var_init} require that 
$\mu<1$ ($\mu$ is defined in Definition \ref{def:ratio}) and imply $Q\mu<1$, such that  $\mathrm{cond} (\mb{B})$
is bounded above with probability at least $1-1/Q^2$. 
To achieve a high probability for a bounded $\mathrm{cond} (\mb{B})$, $Q$ should be large and $\mu$ should 
then be small. 
So, the initial factor matrices (inputs of Algorithm \ref{alg_ftr})  should be generated using realizations of 
a distribution $\Pe$ of which $\mu$ is small. 
As an example,  we show the estimated 
$\mu$ (the $d$-th order ratio with $|m|$ degrees of freedom) for several standard distributions in
Table \ref{table_ratio_mu}, where we set $d=48$ and $m=33$. 
Here, $\mathrm{U}(a_1,a_2)$ refers to a uniform distribution on the interval $[a_1,a_2]$,
and $\mathrm{N}(a_1,a_2)$ refers to a normal distribution with mean $a_1$ and standard deviation $a_2$.
To compute the estimated $\mu$  in Table \ref{table_ratio_mu}, 
$\Ee(\Xi)$ and $\mathrm{Var}(\Xi)$ in \eqref{eq:mu} are computed using the sample mean and the sample
variance of $10^5$ samples of $\Xi$ (note that the relationship between $\Xi$ and $\Pe$ is stated in Definition \ref{def:ratio}).
In the procedure of generating each sample of $\Xi$, 
the corresponding $\mb{B}$ is formulated and its condition number $\mathrm{cond} (\mb{B})$
is stored (see  Definition \ref{def:ratio} and Theorem \ref{Theorem_var_init} for the relationship between $\Xi$ and
$\mb{B}$). The $10^5$ samples of $\Xi$ are associated with $10^5$ samples of $\mathrm{cond} (\mb{B})$,
and Table \ref{table_ratio_mu} also shows the average of these samples of $\mathrm{cond} (\mb{B})$ 
associated with each distribution. 
As shown in Table \ref{table_ratio_mu}, the distributions listed above the dash line have $\mu<1$, and 
they therefore can be used to generate initial factor matrices, while $\mathrm{U}(0,1)$ and  $\mathrm{N}(0,1)$
should not be used. 

Next, our analysis proceeds through induction. 
That is, supposing for a rank $R$ tensor $\zX^{(R)}=[[\mb{A}_1, \ldots, \mb{A}_d]]$,
its corresponding $\mb{B}$ (see Definition \ref{def:B}) is well-conditioned, we show that 
the matrix $\mb{B}$ associated with $\zX^{(R+1)}=[[\mb{A}^{(0)}_1, \ldots, \mb{A}^{(0)}_d]]$
is also well-conditioned, where $\mb{A}^{(0)}_k = [\mb{A}_k , \mb{\delta a}]$ for  $k = 1,\ldots,d$ are the
rank-one updates of the factor matrices and $\mb{\delta a} \in \dsR^{n}$ is a perturbation vector.  
Before introducing our main theorem (Theorem \ref{Theorem_bound_R}), the following lemma is given. 

\begin{mlemma} \label{lemma}
	Given two matrices $\mb{X}_1 \in \dsR^{n_1 \times n_2}$ and $\mb{X}_2 \in \dsR^{n_1 \times n_3}$ with 
	full column ranks where $n_1 > n_2 \geq n_3$, let their singular value 
	decompositions be $\mb{X}_1 = \mb{U}_1 \mb{\Sigma}_1 \mb{V}_1^T$ and $\mb{X}_2 = \mb{U}_2 \mb{\Sigma}_2 \mb{V}_2^T$, where $\mb{U}_1 \in \dsR^{n_1 \times n_1}$, $\mb{\Sigma}_1\in \dsR^{n_1 \times n_2}$, $\mb{V}_1\in \dsR^{n_2 \times n_2}$, $\mb{U}_2\in \dsR^{n_1 \times n_1}$, 
	$\mb{\Sigma}_2 \in \dsR^{n_1 \times n_3}$ and $\mb{V}_2\in \dsR^{n_3 \times n_3}$,
	and let $\mb{u}_1^{(i)}\in \dsR^{n_1}$ for $i=1,\ldots,n_1$ and $\mb{u}_2^{(j)}\in\dsR^{n_1}$ for $j=1,\ldots,n_1$ denote the left 
	singular vectors of $\mb{X}_1$ and $\mb{X}_2$ respectively. 
	Assume the following two conditions hold: 
	first
	\begin{eqnarray}\label{lemma1-1}
	\mb{\Sigma}_2 \mb{\Sigma}_2^T = \mathrm{diag}(\underbrace {{\lambda, \ldots ,\lambda} }_{n_3}, \underbrace {{0, \cdots ,0} }_{n_1-n_3}),
	\end{eqnarray}
	where $\lambda$ is a positive constant; second
	\begin{eqnarray}\label{lemma1-2}
	\langle \mb{u}^{(j)}_2,  \mb{u}^{(i)}_1 \rangle = 0, \textrm{ for }  j = n_3+1, \ldots, n_1 \textrm{ and }  i = 1, \ldots, n_1.
	\end{eqnarray} 
	Then $\mb{X}_1 \mb{X}_1^T$ and $\mb{X}_2 \mb{X}_2^T$ commute, i.e., 
	$\mb{X}_1 \mb{X}_1^T \mb{X}_2 \mb{X}_2^T = \mb{X}_2 \mb{X}_2^T \mb{X}_1 \mb{X}_1^T$.
	
\end{mlemma}
\begin{proof}
	Let $\mb{P} = \mb{U}_2^T \mb{U}_1\in\dsR^{n_1\times n_1}$, $\mb{X} = \mb{P} \mb{\Sigma}_1 \mb{\Sigma}_1^T \mb{P}^T\in \dsR^{n_1\times n_1}$ and 
	$\mb{Z} = \mb{X} \mb{\Sigma}_2 \mb{\Sigma}_2^T \in \dsR^{n_1 \times n_1}$. 
	Using \eqref{lemma1-1},  $\mb{Z}(i,j)=0$ for $j=n_3+1,\ldots,n_1$ and $i=1,\ldots,n_1$, 
	while $\mb{Z}(i,j)=\lambda \mb{X}(i,j)$ for $i=1,\ldots, n_1$ and $j=1,\ldots,n_3$.
	Denoting the singular values of $\mb{X}_1\in \dsR^{n_1\times n_2}$ by $\sigma_1,\ldots,\sigma_{n_2}$ (note that $n_1>n_2$) gives 
	\begin{eqnarray}\label{lemma1-3}
	\mb{X} = \sum_{i=1}^{n_2} \sigma_i^2 \mb{p}_i \mb{p}_i^T,
	\end{eqnarray}
	where $\mb{p}_i$ for $i=1,\ldots,n_1$ are columns of $\mb{P}$. 
	Since $\mb{p}_i=\mb{U}^T_2\mb{u}^{(i)}_1$ for $i=1,\ldots,n_1$, 
	each element of $\mb{p}_i$ is  $\mb{p}_i(j)=\langle\mb{u}^{(j)}_2,\mb{u}^{(i)}_1\rangle$.
	Using \eqref{lemma1-2} gives $\mb{p}_i(j)=0$ for $j=n_3+1,\ldots,n_1$ and $i=1,\ldots,n_1$.
	Therefore, \eqref{lemma1-3} gives $\mb{X}(i,j)=0$ for $j=n_3+1,\ldots,n_1$ and $i=n_3 + 1,\ldots,n_1$.
	In summary, each entry of $\mb{Z}$ is
	\begin{equation} \label{lemma1z1}
	\mb{Z}(i,j) = \begin{cases}
	\lambda\mb{X}(i,j) & \textrm{ for } i \leq n_3, j \leq n_3, \\
	0& \textrm{otherwise}.
	\end{cases} 
	\end{equation}
	Similarly, each entry  of $\mb{Z}^{\prime} :=  \mb{\Sigma}_2 \mb{\Sigma}_2^T \mb{X} \in \dsR^{n_1 \times n_1}$ is 
	\begin{equation}\label{lemma1z2}
	\mb{Z}^{\prime}(i,j) = \begin{cases}
	\lambda\mb{X}(i,j) & \textrm{ for } i \leq n_3, j \leq n_3, \\
	0 & \text{otherwise}.
	\end{cases} 
	\end{equation}
	Combing \eqref{lemma1z1}--\eqref{lemma1z2} 
	gives $\mb{Z}=\mb{Z}'$, and thus $\mb{X} \mb{\Sigma}_2 \mb{\Sigma}_2^T = \mb{\Sigma}_2 \mb{\Sigma}_2^T \mb{X}$, which leads to  
	\begin{equation*}
	\mb{U}_2^T \mb{U}_1 \mb{\Sigma}_1 \mb{\Sigma}_1^T \mb{U}_1^T \mb{U}_2 \mb{\Sigma}_2 \mb{\Sigma}_2^T = \mb{\Sigma}_2 \mb{\Sigma}_2^T  \mb{U}_2^T \mb{U}_1 \mb{\Sigma}_1 \mb{\Sigma}_1^T \mb{U}_1^T \mb{U}_2.
	\end{equation*}
	Left multiplying both sides of the above equation by $\mb{U}_2$ and right multiplying them by $\mb{U}_2^T$ give 
	\begin{equation*}
	\mb{X}_1 \mb{X}_1^T \mb{X}_2 \mb{X}_2^T = \mb{X}_2 \mb{X}_2^T \mb{X}_1 \mb{X}_1^T.
	\end{equation*}
\end{proof}

\begin{theorem} \label{Theorem_bound_R}
	Let $\Theta$ be an observation index set, and $[[\mb{A}_1 \ldots, \mb{A}_d]] \in \dsR^{n \times n \times \cdots \times n}$ be a $d$-th order tensor with rank $R$. 
	Suppose that $\mb{A}_{i_{\delta}} = [\mb{A}_i, \mb{\delta a}]$  with $\mb{\delta a} = [1, \ldots, 1]^T \in \dsR^{n}$ is 
	a rank-one update of $\mb{A}_i$ for $i=1,\ldots,d$.
	Let $\mb{B} := \mathcal{B}_{\Theta,k}([[\mb{A}_{1} \ldots, \mb{A}_{d}]])$, 
	$\mb{\delta B} := \mathcal{B}_{\Theta, k}([[\mb{\delta a}, \mb{\delta a}, \ldots, \mb{\delta a}]])$,
	and $\mb{B}_{\delta} = \mathcal{B}_{\Theta, k}([[\mb{A}_{1 \delta}, \ldots, \mb{A}_{d \delta}]])$,
	and let $\mb{u}^{(j)}_{\mb{\delta B}}$ and $\mb{u}^{(l)}_{\mb{B}}$ are the $j$-th and the $l$-th left singular vectors of $\mb{\delta B}$ and $\mb{B}$ 
	respectively for $j,l=1,\ldots,|\Theta|$.
	If the following three conditions hold: 
	\begin{enumerate}
		\item[$1)$] the observation index set $\Theta$ is uniform as defined in Definition \ref{def:uniform},
		
		\item[$2)$] there exist positive constants $\Ctha$ and $\Cthb$ which are independent of $\mb{B}$, such that $\mathrm{cond}^2(\mb{B}) = s^2_{\mathrm{max}}(\mb{B}) / s^2_{\mathrm{min}}(\mb{B}) \leq \Ctha$ and $|\Theta|/(ns^2_{\mathrm{min}}(\mb{B})) \leq \Cthb$, where $s_{\mathrm{max}}$ and $s_{\mathrm{min}}$ are the largest and the smallest singular values of $\mb{B}$ respectively,
		
		\item[$3)$] $\langle \mb{u}^{(j)}_{\mb{\delta B}},  \mb{u}^{(l)}_{\mb{B}} \rangle = 0$, for $j = n+1, \ldots, |\Theta|$  and $l = 1, \ldots, |\Theta|$,
		
	\end{enumerate}
	then the condition number of $\mb{B}_{\delta}$ satisfies that
	\begin{equation}
	\mathrm{cond}^2(\mb{B}_{\delta}) \leq \Ctha + \Cthb.
	\end{equation}   
\end{theorem}

\begin{proof}
	By Proposition \ref{prop_operator}, we have $\mb{B}_{\delta} = [\mb{B}, \mb{\delta B}]$.
	Note that the largest singular value of $\mb{B}_{\delta}$ satisfies that $s^2_{\text{max}}(\mb{B}_{\delta}) = \lambda_{\text{max}} (\mb{B}\mb{B}^T + \mb{\delta B} \mb{\delta B}^T )$. Using the min-max theorem, 
	\begin{equation*}
	\begin{aligned}
	s^2_{\text{max}}(\mb{B}_{\delta}) &= \underset{\norm{\mb{x}}{2}=1}{\textrm{max}} \mb{x}^T(\mb{B}\mb{B}^T + \mb{\delta B} \mb{\delta B}^T) \mb{x} \\
	& \leq \underset{\norm{\mb{x}}{2}=1}{\textrm{max}} \mb{x}^T(\mb{B} \mb{B}^T) \mb{x} + \underset{\norm{\mb{x}}{2}=1}{\textrm{max}} \mb{x}^T (\mb{\delta B} \mb{\delta B}^T) \mb{x} \\
	& = s^2_{\text{max}}(\mb{B}) + s^2_{\text{max}}(\mb{\delta B})
	\end{aligned}
	\end{equation*} 
	Next we consider the smallest singular value of $\mb{B}_{\delta}$ under the above conditions. Let $\mb{B} = \mb{U}_1 \mb{\Sigma}_1 \mb{V}_1^T$ and $\mb{\delta B} = \mb{U}_2 \mb{\Sigma}_2 \mb{V}_2^T$ be the singular value decompositions of $\mb{B}$ and $\mb{\delta B}$, respectively. 
	
	By condition $1)$ and noting that $\mb{\delta a} = [1, \ldots, 1]^T$, we have $\mb{\Sigma}_2 \mb{\Sigma}_2^T = \mathrm{diag}(\underbrace {{|\Theta|/n, \ldots ,|\Theta|/n} }_{n}, \underbrace {{0, \cdots ,0} }_{|\Theta|-n}) $.  Since $\mb{\Sigma}_2 \mb{\Sigma}_2^T = \mathrm{diag}(\underbrace {{|\Theta|/n, \ldots , |\Theta|/n} }_{n}, \underbrace {{0, \cdots ,0} }_{|\Theta|-n})$ and $\langle \mb{u}^{(j)}_{\mb{\delta B}},  \mb{u}^{(l)}_{\mb{B}} \rangle = 0, j = n+1, \ldots, |\Theta|, \text{and} \ l = 1, \ldots, |\Theta|$. By Lemma \ref{lemma}, $\mb{B} \mb{B}^T$ and $\mb{\delta B} \mb{\delta B}^T$ commute. Therefore, $\mb{B} \mb{B}^T$ and $\mb{\delta B} \mb{\delta B}^T$ can be simultaneously diagonalizable, i.e., there exists an orthonormal matrix $\mb{Q}$ such that
	\begin{equation*}
	\begin{aligned}
	\mb{B} \mb{B}^T &= \mb{Q} \mb{\Lambda}_1 \mb{Q}^T, \ \mb{\Lambda}_1 = \mathrm{diag}(s^2_{\text{max}}(\mb{B}), \ldots, s^2_{\text{min}}(\mb{B}), \underbrace {{0, \cdots ,0} }_{|\Theta|-nR} ), \\
	\mb{\delta B} \mb{\delta B}^T &= \mb{Q} \mb{\Lambda}_2 \mb{Q}^T, \ \mb{\Lambda}_2 = \mathrm{diag}(s^2_{\text{max}}(\mb{\delta B}), \ldots, s^2_{\text{min}}(\mb{\delta B}), \underbrace {{0, \cdots ,0} }_{|\Theta|-n} ).
	\end{aligned}
	\end{equation*}
	Then it follows that $\mb{B} \mb{B}^T + \mb{\delta B} \mb{\delta B}^T = \mb{Q} (\mb{\Lambda}_1 + \mb{\Lambda}_2) \mb{Q}^T$, and $s^2_{\text{min}}(\mb{B}_{\delta}) = s^2_{\text{min}}(\mb{B})$, which gives that
	\begin{equation*}
	\begin{aligned}
	\mathrm{cond}^2(\mb{B}_{\delta}) &= \frac{s^2_{\text{max}}(\mb{B}_{\delta})}{s^2_{\text{min}}(\mb{B}_{\delta})} \\
	& \leq \frac{s^2_{\text{max}}(\mb{B}) + s^2_{\text{max}}(\mb{\delta B})}{s^2_{\text{min}}(\mb{B})} = \mathrm{cond}^2(\mb{B}) + |\Theta|/(ns^2_{\text{min}}(\mb{B})) \leq \Ctha + \Cthb.
	\end{aligned}
	\end{equation*}
\end{proof}

In summary, in our RATR algorithm (Algorithm \ref{alg_rankadaptive_cpl1}), the
fixed-rank tensor recovery algorithm (Algorithm \ref{alg_ftr}) is invoked.
The stability of  Algorithm \ref{alg_ftr} is dependent on the observation index set $\Theta$
and the initial factor matrices. From our above analysis,  
if the observation index set $\Theta$ is uniform, and the initial rank-one factor matrices
are sampled from the distributions given in Table \ref{table_ratio_mu} with $\mu<1$, 
the first tensor recovery step in RATR (on line 2 of Algorithm \ref{alg_rankadaptive_cpl1}) is stable
with high probability.  
In the rank adaptive procedure, 
our analysis shows that the initial factor matrices specified on line 5 of Algorithm \ref{alg_rankadaptive_cpl1}
can lead to stable tensor recovery on line 7 of Algorithm \ref{alg_rankadaptive_cpl1}, 
if each $\mb{B}$ (see Definition \ref{prop_operator}) associated with the data tensor obtained in the 
previous iteration step is well-conditioned. While the overall tensor recovery problem 
\eqref{cp_missvalue_reg} is solved using the alternative minimization iterative method,
our analysis is restricted to the first iteration step. To analyze the stability for the generalized lasso problem
\eqref{glasso} for arbitrary iterations steps during the alternative minimization procedure remains an open problem.
Nevertheless, our analysis here gives a systematic guidance to initialize the factor matrices for 
Algorithm \ref{alg_rankadaptive_cpl1}
(also for Algorithm \ref{alg_ftr}), 
and our numerical results in section \ref{section_test} show that  our RATR approach is stable and efficient.

\subsection{RATR-collocation algorithm}  \label{section_ratrcollocation}  
Our goal is to efficiently conduct uncertainty propagation from the random input $\mb{\xi} \in I^{d}$ to the  
discrete solution (which is high-dimensional) $\zy=\chi (\zxi) = [u(x^{(1)},\mb{\xi}), \ldots, u(x^{(N_h)},\mb{\xi})]^T\in \zM$ 
of \eqref{spdexi1}--\eqref{spdexi2}. 
The overall procedure of RATR-collocation approach is presented as  the following three steps: 
generating data, processing data to construct RATR-collocation model, and 
conducting predictions using the RATR-collocation model.

For generating data, a tensor style quadrature rule \cite{tensorquad} is first specified 
with $n$ quadrature nodes in each dimension. The full index set is then defined as 
$\Theta_{\rm full}:=\{[j_1, j_2, \ldots, j_d]^T \in \Ne^d \, | \,  j_k=1,\ldots, n  \textrm{ for }  k = 1,\ldots, d\}$
and quadrature nodes are denoted as $\{\zxi_{j_1\ldots j_d}, \textrm{ for }  \mb{j}=[j_1, j_2, \ldots, j_d]^T\in \Theta_{\rm full}\}$. 
A observation index set $\Theta$, and  a validation index set  $\Theta'$ are  randomly selected from 
$\Theta_{\rm full}$, 
such that $ \Theta'  \cap \Theta = \emptyset $  and $|\Theta'|<|\Theta_{\rm }|\ll n^d$. 
After that, snapshots $\chi (\zxi_{j_1\ldots j_d} )$ for 
$\mb{j}=[j_1, j_2, \ldots, j_d]^T\in \Theta\cup \Theta'$
are 
computed through solving deterministic versions of \eqref{spdexi1}--\eqref{spdexi2} with high-fidelity numerical schemes.
At the end of this step, the snapshots are stored in a data matrix $\mb{Y} = [\mb{y}^{(1)}, \mb{y}^{(2)}, \cdots , \mb{y}^{(N_t)}]$, 
where $\zy^{(s(\Theta\cup \Theta', \mb{j}))} := \chi(\zxi_{j_1\ldots j_d})$ for $\mb{j}=[j_1, j_2, \ldots, j_d]^T\in \Theta\cup \Theta'$
and $s(\cdot,\cdot)$ is the sort operator defined in Definition \ref{def:sort}. 

To process the data,  kPCA (see section \ref{sec_kpca}) is first applied to result in a reduced-dimensional representation of 
$\mb{Y}$---each $\zy^{(l)}=\chi (\zxi^{(l)})\in \zM$ is mapped to $\mgamma(\zxi^{(l)})=[\gamma_1(\zxi^{(l)}),\ldots,\gamma_{N_r}(\zxi^{(l)})]^T\in \zM_r$ 
for $l=1,\ldots,N_t$.  After that, for each kPCA mode $e=1,\ldots,N_r$, an estimated data tensor \eqref{eq_data_tensor} 
is generated through our RATR approach presented in section \ref{sec_ratr}. That is, through 
setting the observed data $\Pe_{\Theta}({\zXe}):=\mb{p}$ with $\mb{p}=[p_1,\ldots,p_{|\Theta|}]^T$, where $p_{s(\Theta, \mb{j})}=\gamma_e\left(\zxi^{(s(\Theta, \mb{j}))}\right)$ 
for $\mb{j}\in \Theta$, and the validation data $\Pe_{\Theta'}({\zXe}):=\mb{p}$ with $\mb{p}=[p_1,\ldots,p_{|\Theta'|}]^T$, where $p_{s(\Theta', \mb{j})}=\gamma_e\left(\zxi^{(s(\Theta', \mb{j}))}\right)$ for $\mb{j}\in \Theta'$, Algorithm \ref{alg_rankadaptive_cpl1} gives an approximation of $\zX_e$, which is
denoted by $\tilde{\zX}_e$. With this estimated data tensor, each gPC approximation (see \eqref{eq_gpcfs}) 
$\gamma_{e}(\mb{\xi}) \approx \gamma^{\rm gPC}_{e}:=\sum_{|\mb{i}|=0}^{p}{ c_{e\mb{i}}\Phi_{\mb{i}}(\mb{\xi}) }$ for $e=1,\ldots,N_r$  
is obtained with coefficients computed through $c_{e\mb{i}}:=\langle \tilde{\zX}_e, \zW_{\zi} \rangle$,
 where $\zW_{\zi}$ is defined in \eqref{eq_weight_tensor}. 
In the following, we call these gPC approximations $\{\gamma^{\rm gPC}_{e}(\zxi):=\sum_{|\mb{i}|=0}^{p}{ c_{e\mb{i}}\Phi_{\mb{i}}(\mb{\xi}) }\}^{N_r}_{e=1}$ 
the RATR-collocation model.

The above two steps for generating data and constructing the RATR-collocation model are summarized in Algorithm~\ref{alg_cpl1_uq_offline}. 
For conducting a prediction of the snapshot for an arbitrary realization of $\zxi$, we first use RATR-collocation model 
to compute
the output $[\gamma^{\rm gPC}_{1}(\zxi),\ldots,\gamma^{\rm gPC}_{N_r}(\zxi)]^T$ in the reduced-dimensional manifold $\zM_r$. 
With the reduced output  $[\gamma^{\rm gPC}_{1}(\zxi),\ldots,\gamma^{\rm gPC}_{N_r}(\zxi)]^T$ and the data matrix $\mb{Y}$ (generated in Algorithm~\ref{alg_cpl1_uq_offline}),  an estimation of the snapshot is obtained through the inverse mapping (see section \ref{section_inverse} and \cite{xing2016manifold}),
which is denoted as 
$\zy_{\rm RATR}:=\chi_{\rm RATR} (\zxi) \in \dsR^{N_h}$. 

\begin{algorithm}
	\caption{RATR-colocation in the reduced-dimensional manifold $\mathcal{M}_r$}
	\label{alg_cpl1_uq_offline}
	\begin{algorithmic}[1]
		\Require a full index set $\Theta_{\rm full}$, 
		quadrature nodes $\left\{\zxi_{j_1\ldots j_d}, \textrm{ for }  \mb{j}=[j_1, j_2, \ldots, j_d]^T\in \Theta_{\rm full}\right\}$,
		an observation index set $\Theta$, 
		a validation index set $\Theta^{\prime}$, and a gPC order $p$. 
		\State Generate a data matrix $\mb{Y} = \left[\mb{y}^{(1)}, \mb{y}^{(2)}, \cdots , \mb{y}^{(N_t)}\right]$,
		where $\zy^{\left(s(\Theta \cup \Theta',\mb{j})\right)} := \chi\left(\zxi_{j_1\ldots j_d}\right)$ for $\mb{j}\in \Theta\cup \Theta'$ 
		are obtained through high-fidelity simulations
		for deterministic versions of \eqref{spdexi1}--\eqref{spdexi2} and $s(\cdot, \cdot)$ is defined in Definition \ref{def:sort}.
		\State Perform kPCA for $\mb{Y}$ to obtain 
		$\mgamma\left(\zxi^{(l)}\right)=\left[\gamma_1\left(\zxi^{(l)}\right),\ldots,\gamma_{N_r}\left(\zxi^{(l)}\right)\right]^T$ for $l=1,\ldots,N_t$.
		
		\For {$e = 1:N_r$}
		\State Define $\Pe_{\Theta}({\zXe}):=\mb{p}$ with $\mb{p}=[p_1,\ldots,p_{|\Theta|}]^T$, where $p_{s(\Theta,\mb{j})}=\gamma_e\left(\zxi^{(s(\Theta, \mb{j}))}\right)$ for $\mb{j}\in \Theta$.
		\State Define $\Pe_{\Theta'}({\zXe}):=\mb{p}$ with $\mb{p}=[p_1,\ldots,p_{|\Theta'|}]^T$, where $p_{s(\Theta', \mb{j})}=\gamma_e\left(\zxi^{(s(\Theta', \mb{j}))}\right)$ for $\mb{j}\in \Theta'$.
		\State Generate an estimated data tensor $\zX$ using Algorithm \ref{alg_rankadaptive_cpl1}, and  define ${\tilde{\zX}_e:=\zX}$.
		\State Generate the gPC approximation  
		$\gamma_{e}(\mb{\xi}) \approx \gamma^{\rm gPC}_{e}:=\sum_{\norm{\mb{i}}{1}=0}^{p}{ c_{e\mb{i}}\Phi_{\mb{i}}(\mb{\xi}) }$ with $c_{e\mb{i}}:=\left<\tilde{\zX}_e, \zW_{\zi} \right>$ for $\zi \in \Upsilon:=\{\mb{i} |\, \mb{i} \in \Ne^{d} \textrm{ and } \norm{\mb{i}}{1}=0,\ldots,p\}$, where $\zW_{\zi}$ is defined in \eqref{eq_weight_tensor}.
		\EndFor
		\Ensure  gPC approximations $\gamma^{\rm gPC}_{1}(\mb{\xi}), \ldots,\gamma^{\rm gPC}_{N_r}(\mb{\xi})$ and 
		the data matrix $\mb{Y} = \left[\mb{y}^{(1)}, \mb{y}^{(2)}, \cdots , \mb{y}^{(N_t)}\right]$.
	\end{algorithmic}
\end{algorithm}

\section{Numerical study}\label{section_test}
In this section, 
we first consider diffusion problems in section \ref{sec_numexp_diff1} and 
section \ref{sec_numexp_diff2}, and consider a Stokes problem in section \ref{sec_numexp_stokes}. 
The governing equations of the diffusion problems are
\begin{eqnarray}
-\div \left[a\left(x,\mb{\xi}\right)\nabla u\left(x,\mb{\xi}\right)\right]=1 
& \quad \textrm{in}\quad & D\times I^d, \label{diff1} \\
u\left(x,\mb{\xi}\right)=0                                  
& \quad\textrm{on}\quad & \partial D_D \times I^d,\label{diff2}\\
\frac{\partial u\left(x,\mb{\xi}\right)}{\partial n}=0 
& \quad\textrm{on}\quad & \partial D_N \times I^d,\label{diff3}
\end{eqnarray}
where $\partial u/\partial n$ is the outward normal derivative of $u$ on the boundaries,
$\partial D_D\cap \partial D_N=\emptyset$ and $\partial D=\partial D_D\cup \partial D_N$.
In the following numerical studies, the spatial domain  is taken to be $D=(0,1)\times(0,1)$. 
The condition (\ref{diff2}) is 
applied on the left ($x=0$) and right ($x=1$) boundaries, and (\ref{diff3}) 
is applied on the top and bottom boundaries. 
Defining $H^1(D):=\{u: D \to \dsR, \int_{D} u^2\, {\rm d} D<\infty, \int_{D} (\partial u/ \partial x_l)^2\, {\rm d} D<\infty,  l=1,2\}$
and $H_0^1(D):=\{v\in H^1(D)\,| \,  v=0 \textrm{ on } \partial D_D\}$, 
the weak form of (\ref{diff1})--(\ref{diff3}) is to find
$u(x,\xi)\in H_0^1(D)$ such that
$(a \nabla u, \nabla v) = (1,v)$ for all $v\in H_0^1(D)$. 
We discretize in space using a bilinear finite element approximation \cite{Elman2014}, 
with a uniform $65\times 65$ grid ($N_h=4225$).

The diffusion coefficient $a(x,\mb{\xi})$ in our numerical studies is assumed to be a random 
field with mean function $a_0(x)$, standard deviation $\sigma$ and 
covariance function $Cov(x,y)$,
\begin{eqnarray} \label{eq_cov}
Cov(x,y)=\sigma^2 \exp\left(-\frac{|x_1-y_1|}{\covl}-\frac{|x_2-y_2|}{\covl}
\right),\label{covariance}
\end{eqnarray}
where $x=[x_1,x_2]^T, y=[y_1,y_2]^T\in \dsR^2$ and $\covl$ is the correlation length. 
This random field can be approximated by a truncated Karhunen--Lo\`eve (KL)
expansion \cite{ghanem2003sfem}
\begin{eqnarray}
a(x,\mb{\xi})\approx 
a_0(x)+\sum_{i=1}^d\sqrt{\lambda_i}a_i(x)\xi_i, \label{kl}
\end{eqnarray}
where $a_i(x)$ and $\lambda_i$ are eigenfunctions and eigenvalues 
of (\ref{covariance}), $d$ is the number of KL modes retained, and 
$\{\xi_i\}^d_{i=1}$ are uncorrelated random variables.
We set the random variables $\{\xi_i\}^d_{i=1}$ to be 
independent uniform distributions with range $I=[-1,1]$,
and set $a_0(x) = 1$ and $\sigma^2 = 0.25$.
For test problem 1 (in section \ref{sec_numexp_diff1}), 
we set $l_c = 0.8$ and  $d=48$, such  that at least $95\%$ of the total variance is captured,
i.e., $(\sum_{i=1}^{d}\lambda_i)/(|D|\sigma^2) > 0.95$, where $|D|$ is the area of $D$. 
For test problem 2 (in section \ref{sec_numexp_diff2}), 
we set $l_c=1/16$ and again set $d=48$.

For all test problems, we set the gPC order $p=2$ (see section \ref{section_col}) and 
take $n = 3$ Gaussian quadrature points for each dimension,
while $\Theta_{\rm full}$ is constructed by the tensor product of these three points ($|\Theta_{\rm full}|=3^{48}$). 
As in the input of Algorithm \ref{alg_cpl1_uq_offline}, an observation index set $\Theta$ and a validation index $\Theta'$ are required.
We test three cases of $\Theta$ uniformly sampled from $\Theta_{\rm full}$ with sizes $|\Theta|=100$, $300$ and $600$ respectively,
and generate $\Theta'$ using $20$ samples uniformly sampled from $\Theta_{\rm full}$, such that  $\Theta \cap \Theta^{\prime} = \emptyset$. Note that the number of high-fidelity simulations 
(the finite element methods here) in our RATR  is $|\Theta|+|\Theta^{\prime}|$, while that in standard tensor grid 
collocation \cite{babuvska2007stochastic} 
is $|\Theta_{\rm full}|=3^{48}$ and that in sparse grid collocation \cite{Xiu2005High,maza2009adaptive} 
is still around $4705$  (for a comparable grid level). 
So, the cost of RATR-collocation is much smaller than the costs
of both tensor and sparse grid collocation methods for high-dimensional problems. 

For the diffusion test problems. 
The regularization parameter $\beta$ in (\ref{cp_missvalue_reg}) is set to $0.01$, the tolerance
in Algorithm \ref{alg_ftr}  is set to $\delta = 10^{-5}$, and the initial rank-one matrices for Algorithm~\ref{alg_rankadaptive_cpl1}
are generated with samples of $\mathrm{U}(1,2)$ which is an optimal initializaiton strategy as discussed in section \ref{section_hdoutput_initana}. 
For kPCA as reviewed in section \ref{sec_kpca},  we set the criterion for selecting  principal components to $tol_{\rm PCA}=90\%$,
and set the bandwidth to $\sigma_g=5$ for the diffusion test problems.
For a given realization of $\mb{\xi}$,  $\zy:=\chi(\mb{\xi})$ denotes the finite element solution, and $\zy_{\rm RATR}:=\chi_{\rm RATR}(\mb{\xi})$
refers to a RATR-collocation approximation solution (see section \ref{section_ratrcollocation}). 
A relative error is then defend as  
\begin{equation}\label{diff_error}
\text{Relative error} = \frac{\norm{ \zy - \zy_{\text{RATR}}}{2} }{\norm{\zy}{2} }. 
\end{equation}
\subsection{Test problem 1: diffusion problem with $\covl=0.8$ and $d=48$} \label{sec_numexp_diff1}
For each case of the observation index set $\Theta$ (with $N_t:=|\Theta|=100$, $300$ and $600$ respectively),
we first generate the corresponding data matrix $\mb{Y}$ and apply kPCA for dimension reduction.
For the given tolerance $tol_{\rm PCA}=90\%$,
the number of kPCA modes retained is $N_r=4$ for the three cases here (see section \ref{sec_kpca} for
the definitions of $N_r$ and $tol_{\rm PCA}$). 
For each kPCA mode, our RATR algorithm gives an estimation $\tilde{\zX}_e$ of the data tensor $ \zX_e$ 
for $e=1,\ldots,N_r$ (see line 6 of Algorithm \ref{alg_cpl1_uq_offline}), 
where $\zX_e$ is defined in \eqref{eq_data_tensor}.
Tabel~\ref{table_cprank1} shows the estimated CP ranks 
of $\zX_e$ generated through Algorithm~\ref{alg_rankadaptive_cpl1}.
It is clear that, these estimated ranks of each $\zX_e$ are similar for the three cases of $\Theta$,
and they are very small---the maximum estimated CP rank for this test problem is four.

\begin{table}[!htp]
	\caption{Estimated CP ranks of each data tensor $\zX_e$ for $e=1,\ldots,4$, test problem 1.} 
	\centering	
	\begin{tabular}{c|ccccc}  
		\hline
		\diagbox{$|\Theta|$}{rank}{$e$} 
		& 1 & 2 & 3 & 4 \\ \hline
		100 &  4 & 2 & 3 & 1 \\
		300 &  2 & 1 & 1 & 3 \\
		600 &  4 & 1 & 1 & 2 \\
		\hline
	\end{tabular}
	\label{table_cprank1}
\end{table}

To assess the efficiency of our RATR  procedure,
we compare Algorithm~\ref{alg_rankadaptive_cpl1} with the standard
fixed-rank tensor recovery approach (Algorithm~\ref{alg_ftr}) to recover $\zX_1$ with $|\Theta|=600$ for this test problem.
As discussed above, the initial rank-one factor matrices for RATR are generated thorough the distribution $\mathrm{U}(1,2)$.
For Algorithm~\ref{alg_ftr}, for each given rank $R=1,\ldots,4$, two distributions are tested for generating  
the initial matrices: $\mathrm{U}(1,2)$ and $\mathrm{N}(0,1)$.
Note that, as discussed in section  \ref{section_hdoutput_initana}, $\mathrm{U}(1,2)$ is an optimal choice and $\mathrm{N}(0,1)$ 
is a non-optimal choice for the situation that the CP rank is one. 
In the following, the fixed-rank tensor recovery approach (Algorithm~\ref{alg_ftr}) with initial factor matrices generated through the optimal choice $\mathrm{U}(1,2)$ is denoted by FRTR-O, and
that with initial factor matrices generated through the non-optimal choice $\mathrm{N}(0,1)$ 
is denoted by FRTR-N. 
Figure~\ref{fig_vlidation_error1}(a) shows the validation errors \eqref{eq_cv_cpl2} of 
the recovered tensor generated by RATR, FRTR-O and FRTR-N respectively, where it is clear that 
for each rank $R=1,\ldots,4$, our RATR has the smallest validation error.
As discussed in section \ref{sec_cpl1_reg}, the overall tensor recovery problem \eqref{cp_missvalue_reg} is solved through
the alternative minimization iterative method (see \eqref{cpl1_altform}). 
Looking more closely, the validation errors at each iteration step of the alternative minimization iterative method 
for $R=1,2,4$ are shown in Figure~\ref{fig_vlidation_error1}(b), Figure~\ref{fig_vlidation_error1}(c) and Figure~\ref{fig_vlidation_error1}(d)
respectively (since the results of $R=3$ and $R=4$ similar,  we only show the results of $R=4$). 
For $R=1$ (Figure~\ref{fig_vlidation_error1}(b)), there is no rank adaptive procedure preformed in RATR, and 
the validation errors of RATR and FRTR-O are the same, while it is clear that they are much smaller than 
the errors of FRTR-N.
Moreover, the validation error of FRTR-N can even become larger as the iteration step increases for $R=1$, 
which is consistent with Theorem \ref{Theorem_var_init}. For $R=2,4$ (Figure~\ref{fig_vlidation_error1}(c) and Figure~\ref{fig_vlidation_error1}(d)),
it can be seen that RATR has the smallest validation errors at each iteration step, which shows that our rank-one 
updating procedure (on line 5 of Algorithm \ref{alg_rankadaptive_cpl1}) gives efficient initial factor matrices for the 
generalized lasso problem \eqref{glasso}. 


\begin{figure}[!ht]
	\centering
	\subfloat[][Validation errors w.r.t CP ranks ]{\includegraphics[width=.4\textwidth]{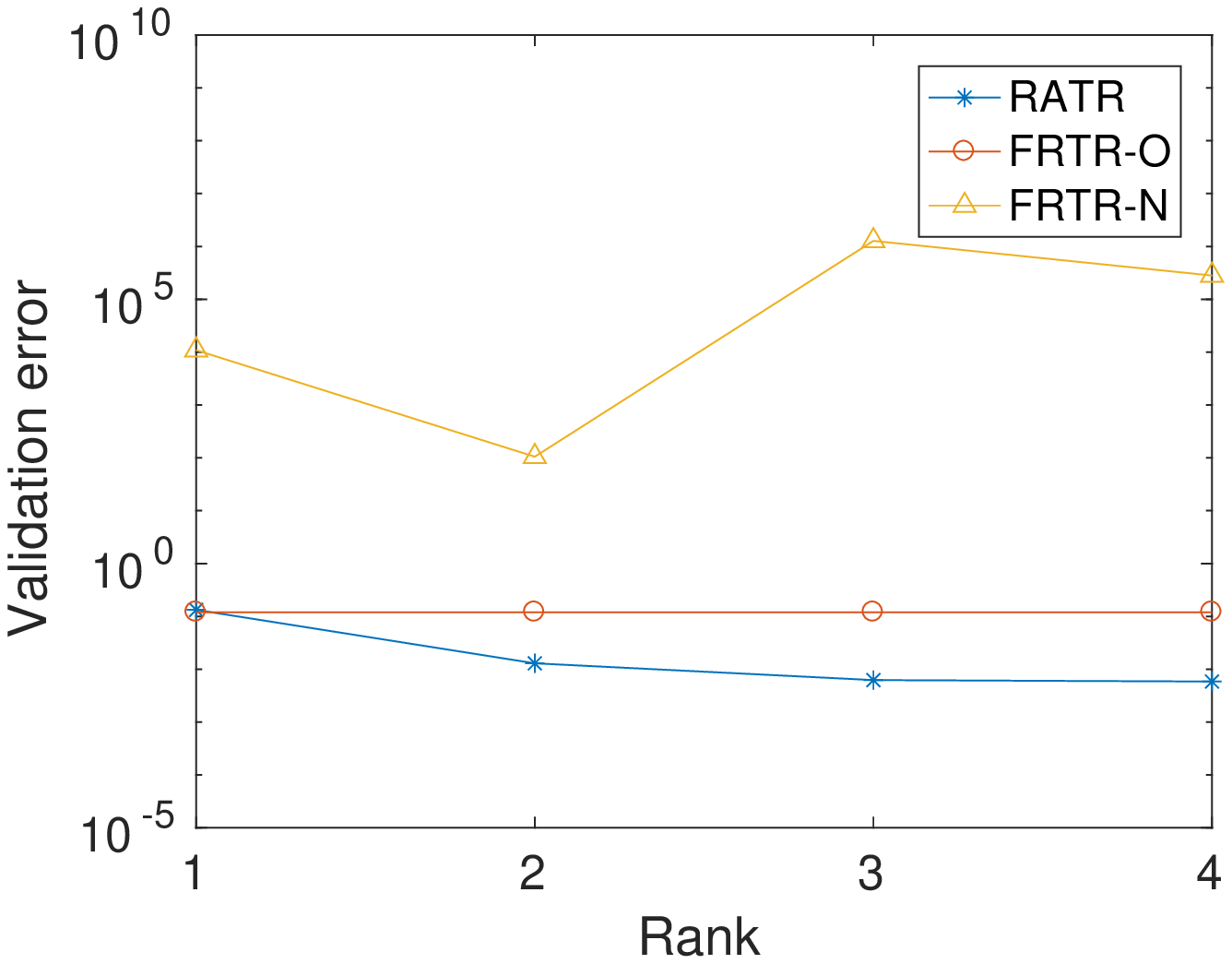}}\quad
	\subfloat[][Validation errors at $R=1$ ]{\includegraphics[width=.4\textwidth]{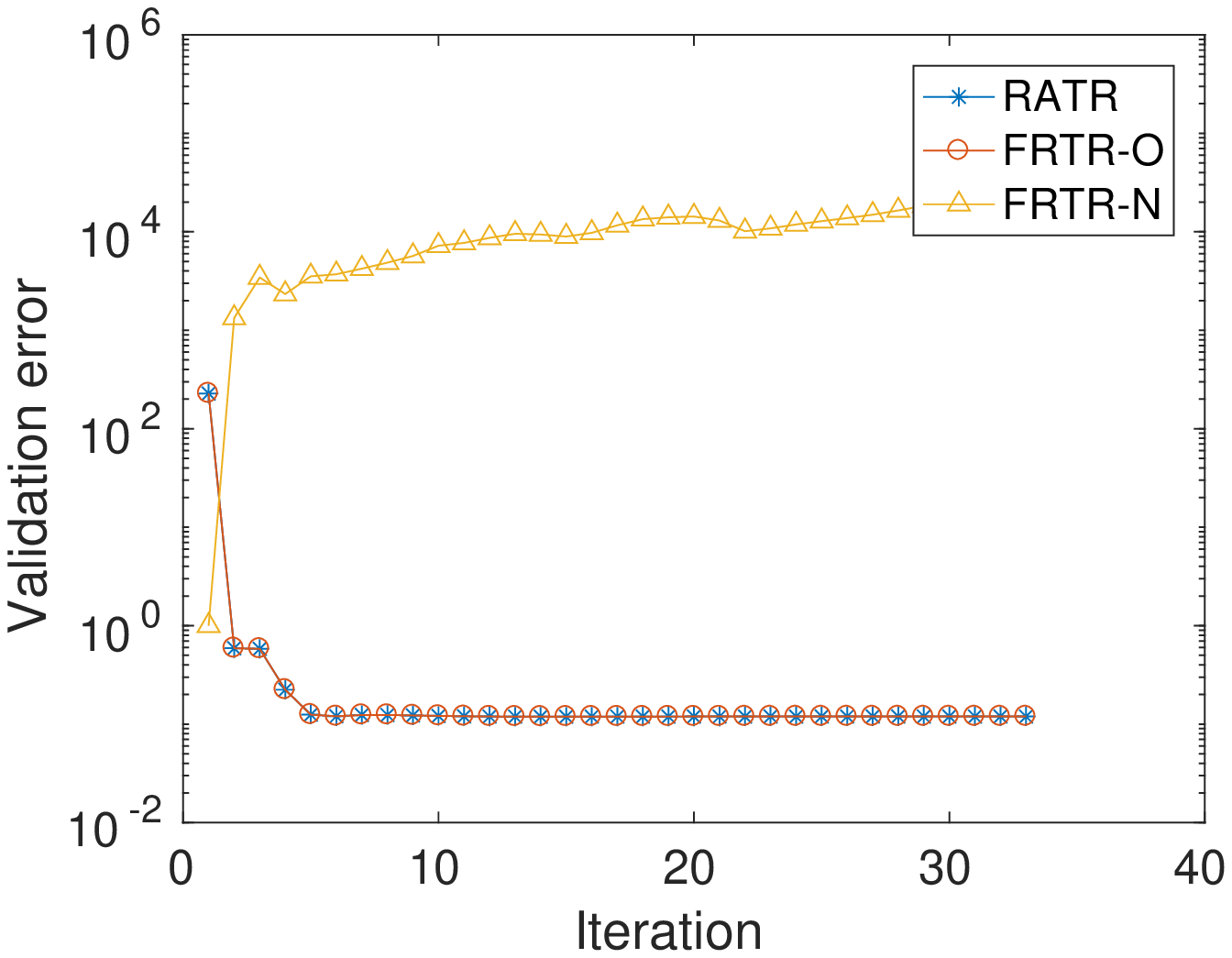}}\\
	\subfloat[][Validation errors at $R=2$ ]{\includegraphics[width=.4\textwidth]{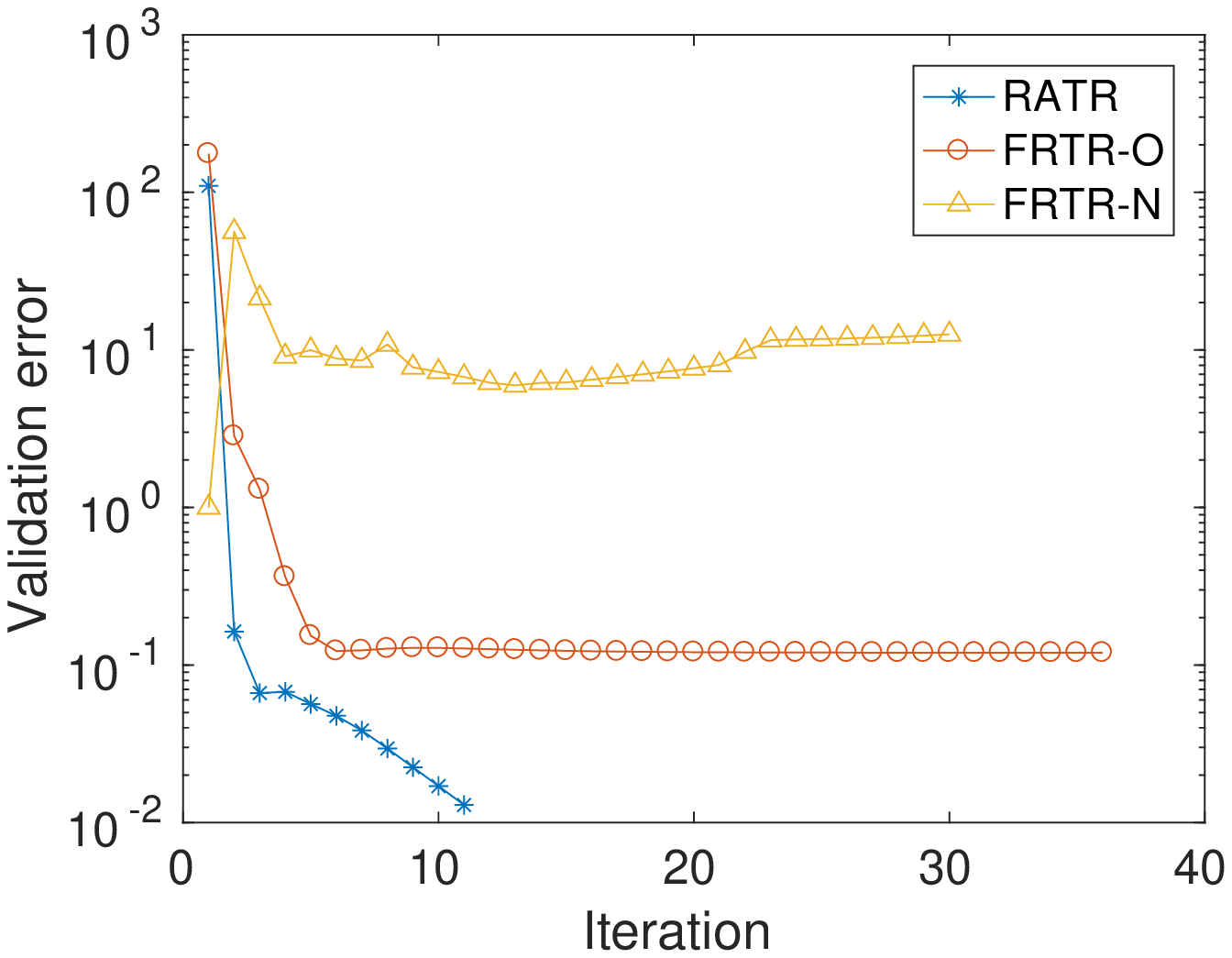}}\quad
	\subfloat[][Validation errors at $R=4$ ]{\includegraphics[width=.4\textwidth]{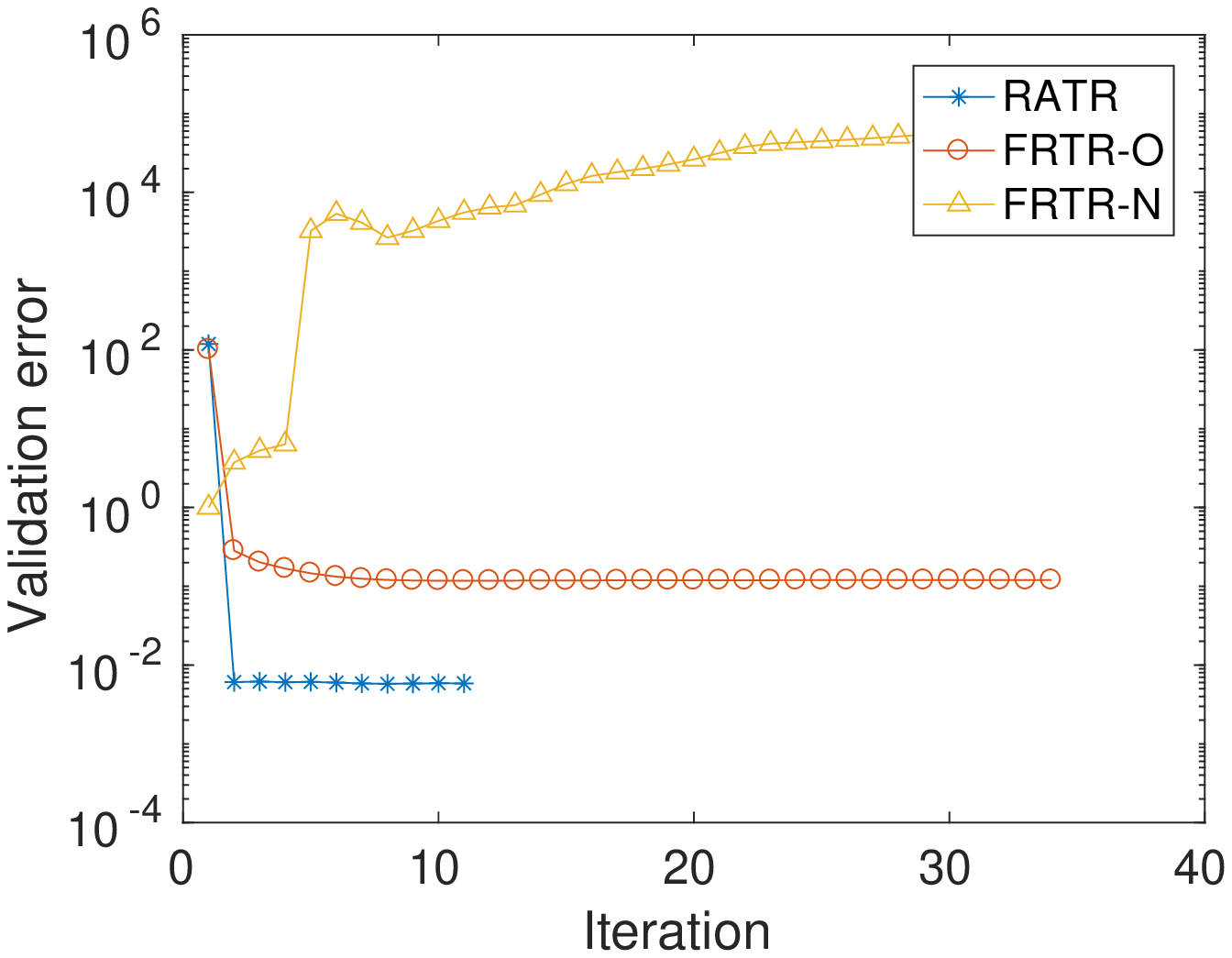}}
	\caption{Validation errors of rank adaptive tensor recovery (RATR),
		fixed-rank tensor recovery with initial factor matrices generated through $\mathrm{U}(1,2)$  (FRTR-O),
		and fixed-rank tensor recovery with initial factor matrices generated through $\mathrm{N}(0,1)$  (FRTR-N), 
		test problem 1. }
	\label{fig_vlidation_error1}
\end{figure}


While the sparsity of the gPC coefficients is taken into account in the tensor recovery problem \eqref{cp_missvalue_reg},
we show the absolute value of each the gPC coefficient $c_{e\mb{i}}$ (see \eqref{eq_coeff_gpc}) for each kPCA mode $e=1,\ldots,4$ and 
each gPC multi-index $\mb{i}\in \Upsilon$ (see section \ref{section_col}) in Figure~\ref{fig_coeff1}.
In Figure~\ref{fig_coeff1}, the gPC multi-index set is labeled as 
$\Upsilon=\{\mb{i}^{(1)},\ldots,\mb{i}^{|\Upsilon|}\}$, where the indices are 
sorted by the sort operator $s(\Upsilon,\cdot)$ (see section \ref{sec_cpl1_reg}, Definition \ref{def:sort}). 
From Figure~\ref{fig_coeff1}, it is clear that the gPC coefficients are sparse---absolute values of 
most coefficients are smaller than $10^{-4}$, which is consist with the results in \cite{doostan2011nonsparse}.

\begin{figure}[!ht]
	\centering
	\subfloat[][1st kPCA mode]{\includegraphics[width=.4\textwidth]{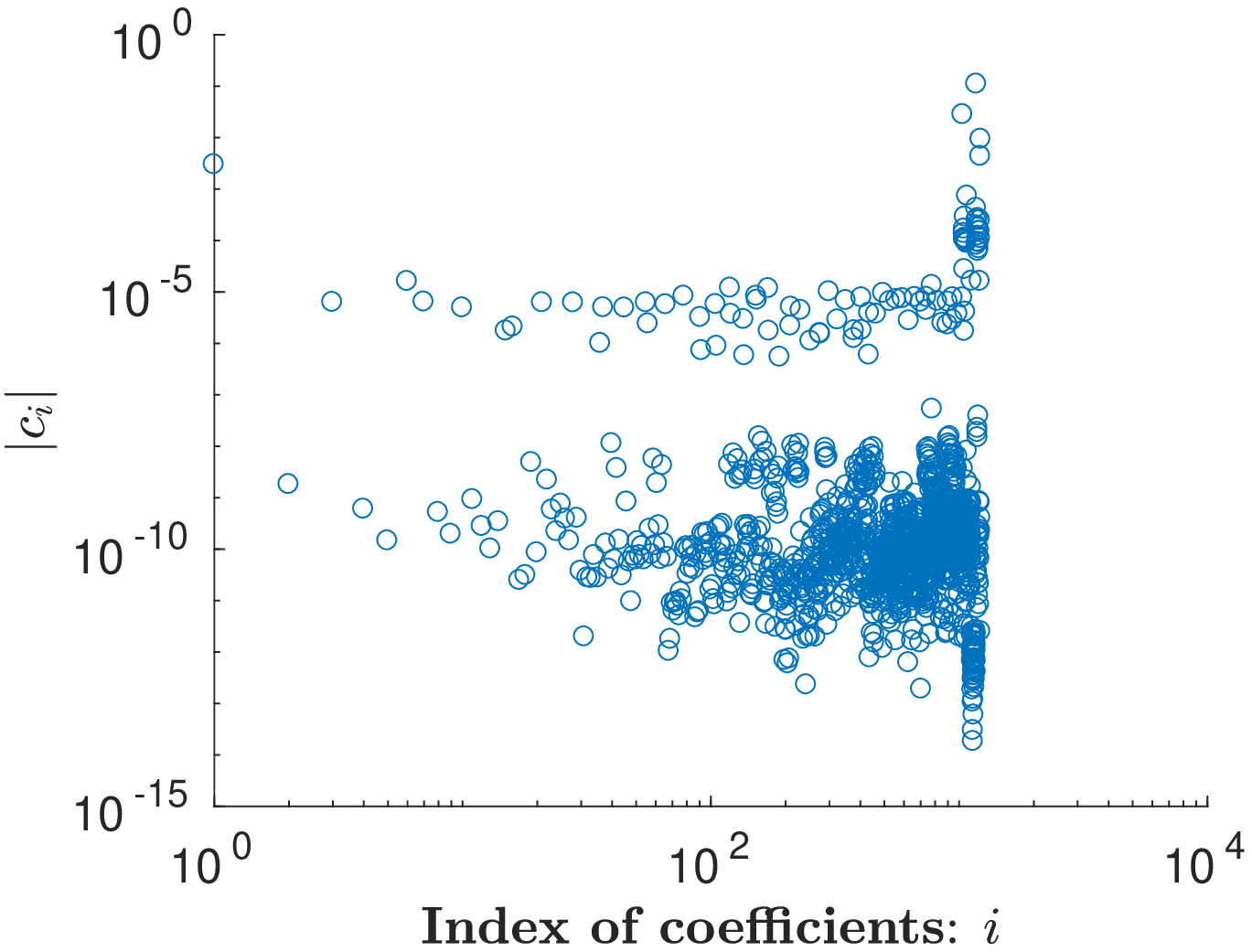}}\quad
	\subfloat[][2nd kPCA mode]{\includegraphics[width=.4\textwidth]{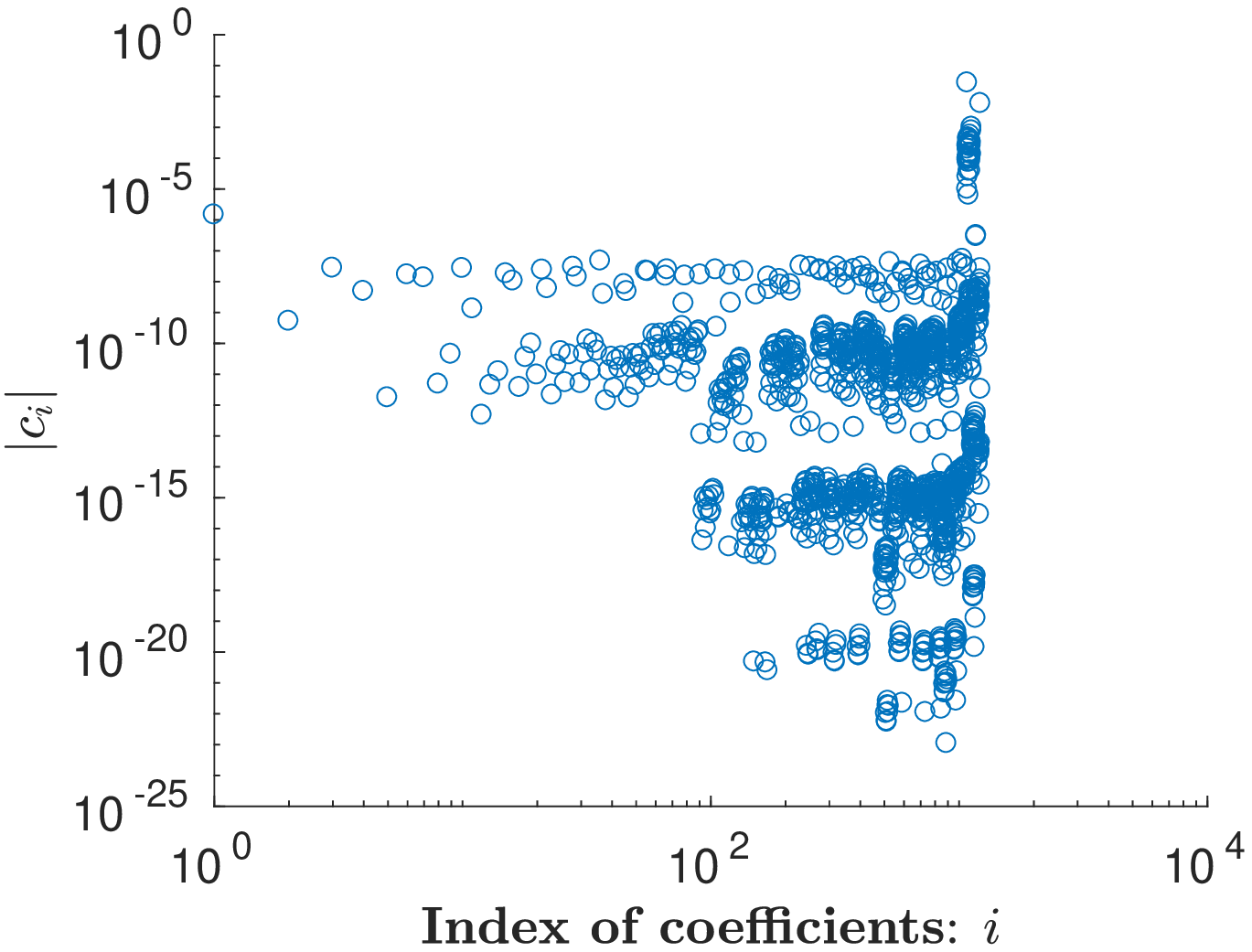}}\\
	\subfloat[][3rd kPCA mode]{\includegraphics[width=.4\textwidth]{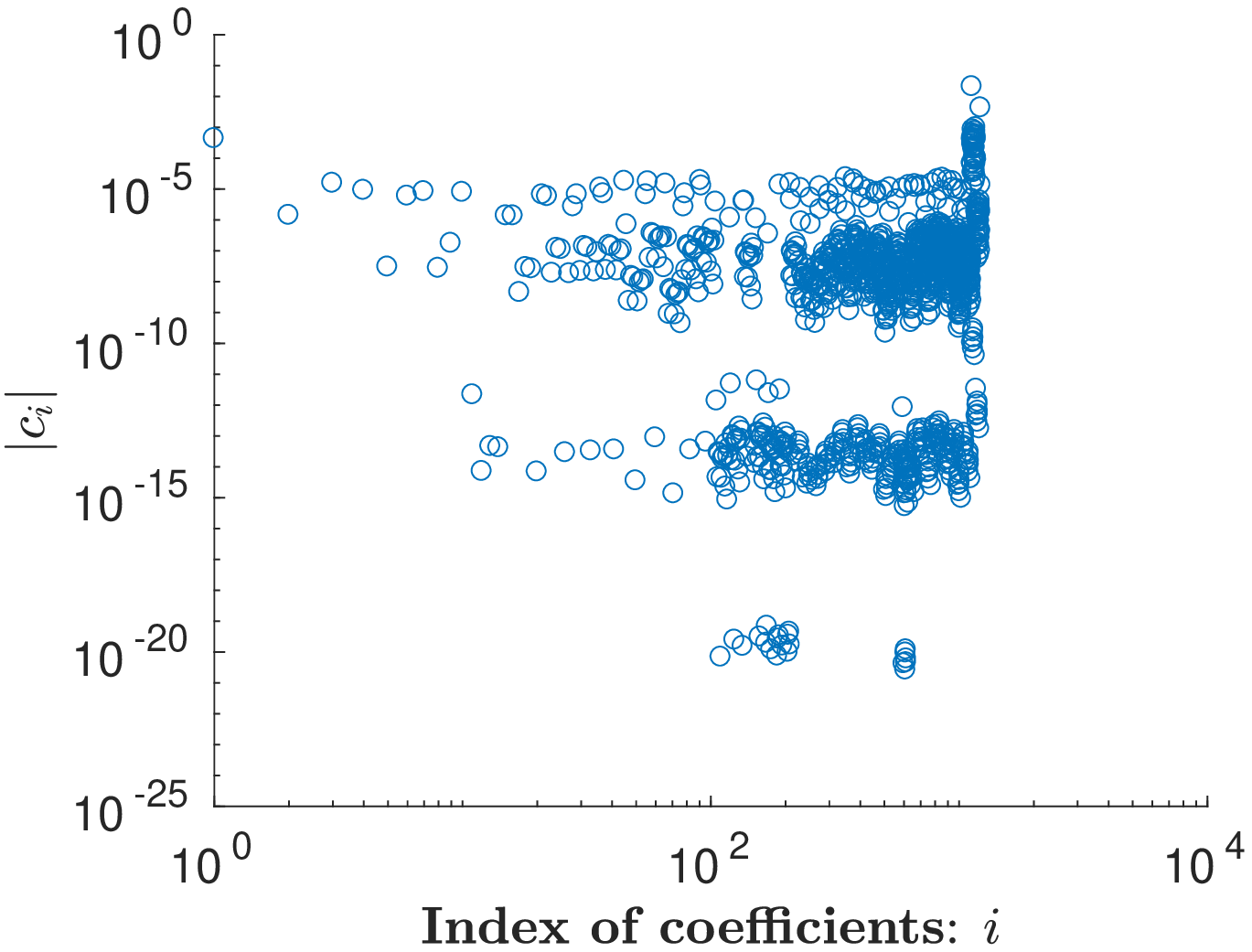}}\quad
	\subfloat[][4th kPCA mode]{\includegraphics[width=.4\textwidth]{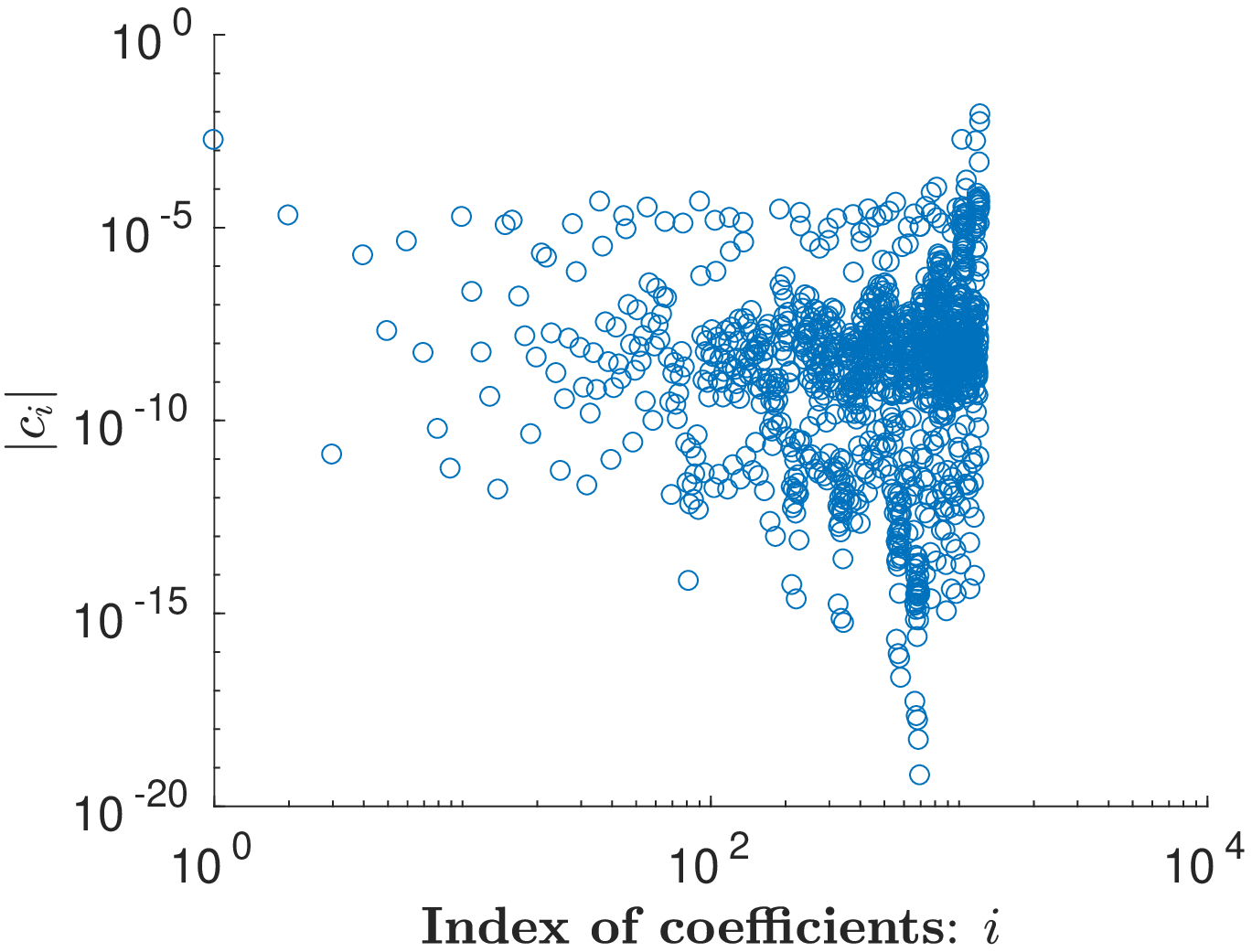}}
	\caption{Sparsity of gPC  coefficients for each kPCA mode, test problem 1.}
	\label{fig_coeff1}
\end{figure}

Figure \ref{fig_sol1} shows the finite element solution $\mb{y}$ and 
the RATR-collocation approximation $\mb{y}_{\text{RATR}}$ 
responding to a given realization of $\mb{\xi}$, where it can be seen that they are visually indistinguishable.   
Finally, we generate $500$ samples of $\mb{\xi}$, and compute the relative error \eqref{diff_error} 
for the three cases ($|\Theta|=100$, $300$ and $600$ respectively).
Figure~\ref{fig_diff_box1} shows Tukey box plots of these errors. 
Here, the central line in each box is the median, the lower and the upper edges are the first and the third
quartiles respectively, and the red crosses are the outliers where the relative errors are large. 
From Figure~\ref{fig_diff_box1}, it is clear that as the size of the observation index set ($|\Theta|$)
increases, values of the median, the first and the third quartiles of the errors decrease.  


\begin{figure}[!htp]
	\centering
	\includegraphics[width=0.48\linewidth, height = 0.36\linewidth]{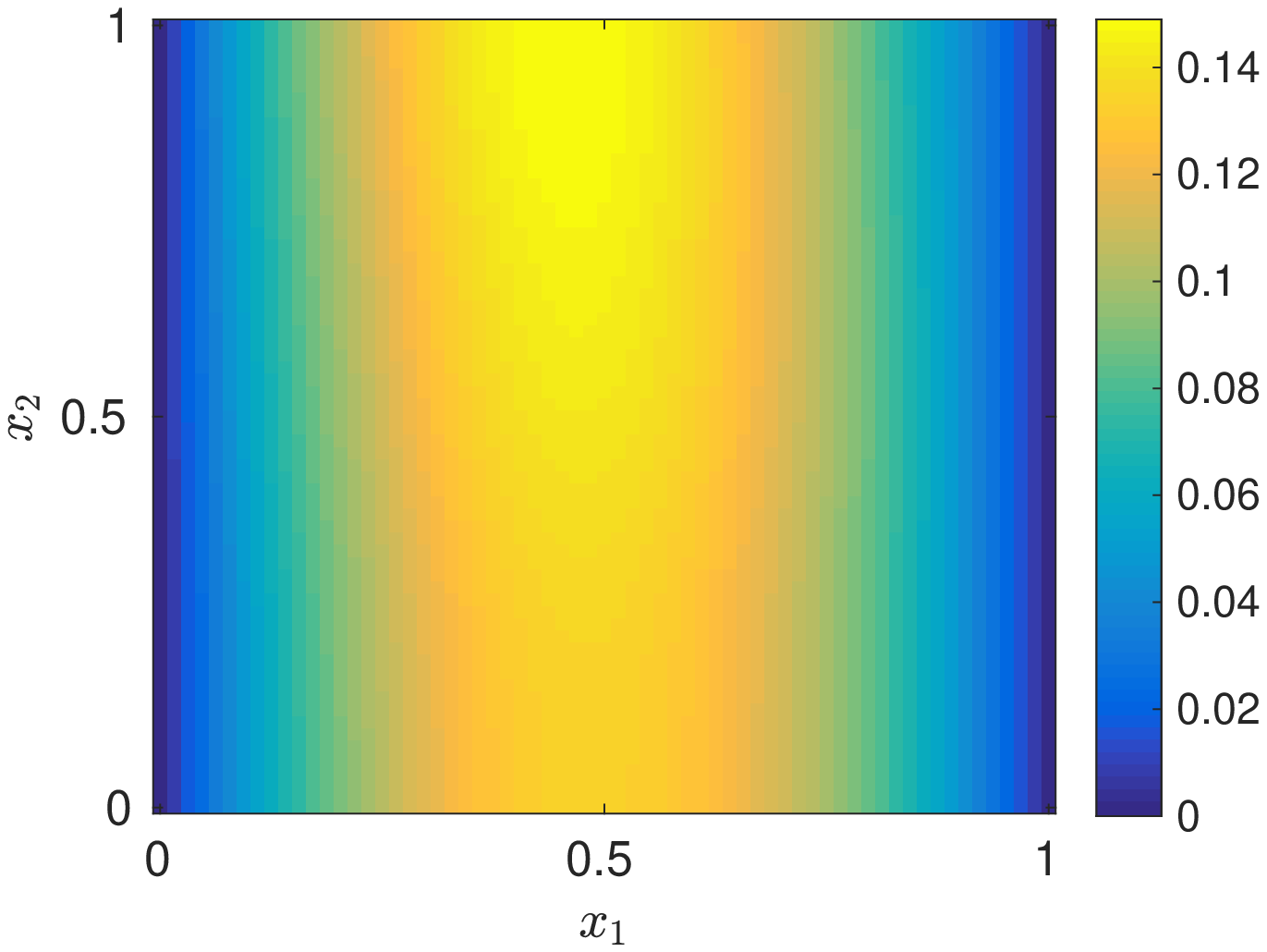}
	\hspace{0.2cm}
	\includegraphics[width=0.48\linewidth,height = 0.36\linewidth]{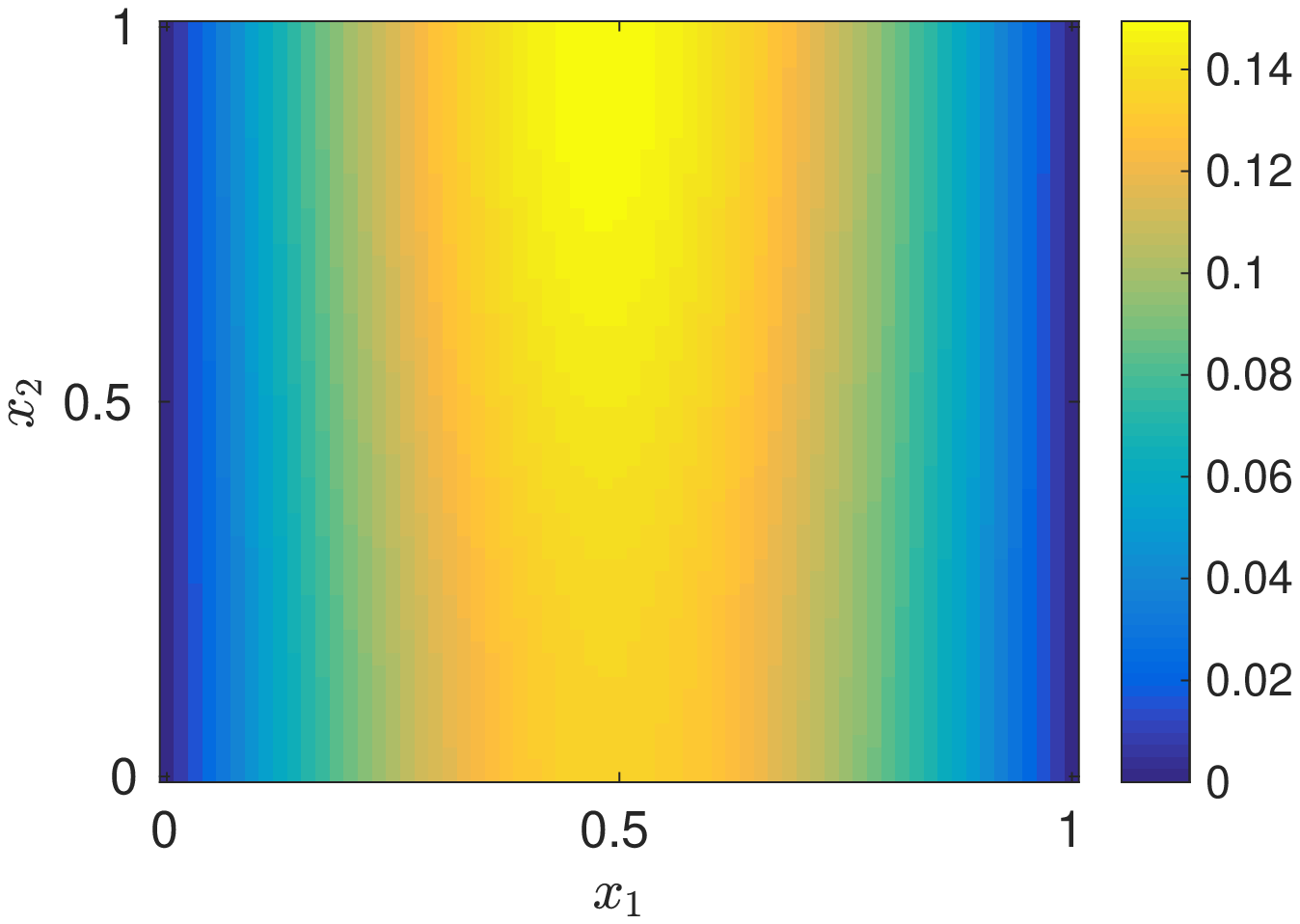}
	\subfloat[\label{fig_diffsol_s} Finite element solution]{\hspace{.5\linewidth}}
	\subfloat[\label{fig_diffsol_t}  RATR-collocation approximation]{\hspace{.5\linewidth}}
	\caption{The finite element solution and the RATR-collocation approximation 
		(with $|\Theta|=600$) responding to a given realization of $\mb{\xi}$,  test problem~1.}
	\label{fig_sol1}
\end{figure}   

\begin{figure}[!htp]
	\centering
	\includegraphics[width=0.58\linewidth, height = 0.38\linewidth]{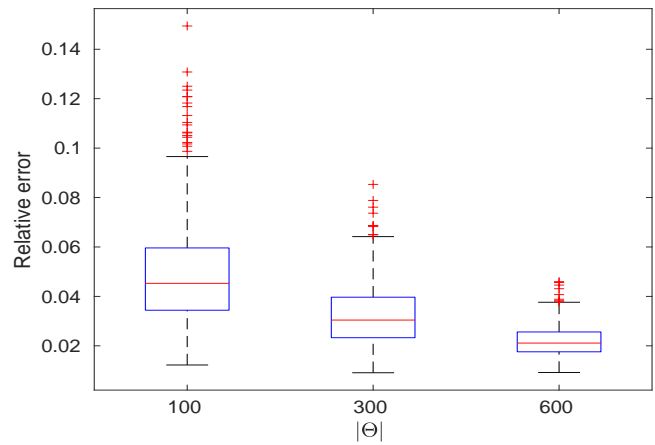}
	\caption{Relative errors of RATR-collocation approximation for $500$ test samples, test problem~1.}
	\label{fig_diff_box1}
\end{figure}

\subsection{Test problem 2: diffusion problem with $\covl=1/16$ and $d=48$} \label{sec_numexp_diff2}
For this test problem, the correlation length is very small, and the diffusion problem becomes highly non-smooth.
Following the discussion procedure in test problem 1, we first generate the corresponding data matrix $\mb{Y}$
for the three cases of $\Theta$ ($|\Theta| = 100, 300$ and $600$) and apply kPCA on it. 
For $tol_{\text{PCA}}=90\%$, the number of kPCA modes retained is $N_r=7$ for this test problem. 
Tabel~\ref{table_cprank_ftr} shows the estimated CP ranks 
of $\zX_e$ generated through Algorithm~\ref{alg_rankadaptive_cpl1} for each $e=1,\ldots,7$,
where it is clear that the estimated ranks are small (the maximum of the estimated ranks is seven).
Figure \ref{fig_errortr_iteration} shows validation errors of our RATR, FRTR-O and FRTR-N (Algorithm~\ref{alg_ftr} with
initial factor matrices generated through $\mathrm{U}(1,2)$ and $\mathrm{N}(0,1)$ respectively)
for recovering $\mb{\mathrm{X}}_1$ (see \eqref{eq_data_tensor}) with $|\Theta| = 600$ .
From Figure \ref{fig_errortr_iteration}(a), it can be seen that our RATR has the smallest validation error for each rank $R=2,\ldots,7$. 
It is also clear that, as the ranks increase, the error of RATR decreases,  while the errors of  FRTR-O and FRTR-N do not  
decrease. The other pictures in  Figure \ref{fig_errortr_iteration} show the validation errors at each iteration step of the 
alternative minimization iterative procedure for $R=1,2,3,4,7$ (since the errors for $R=5,6$ are similar to the errors for $R=7$,
they are not shown here). For the case $R=1$ (Figure \ref{fig_errortr_iteration}(b)), while RATR is the same as FRTR-O, 
 the error of  RATR is larger than the error of FRTR-N, but the errors of RATR and  FRTR-N are both very large (larger than one), 
 which implies that this rank (R=1) is too small to accurately recover the data tensor. 
 As the rank increases,  for $R=2,3,4,7$, the validation errors of RATR are clearly smaller than the errors of FRTR-O 
and FRTR-N, which is consist with the results in test problem 1.
Figure \ref{fig_difftr_cpl1coeff} shows the absolute values of the gPC coefficient 
$c_{e\mb{i}}$ (see \eqref{eq_coeff_gpc}) for $\mb{i}\in \Upsilon$  and $e=1,\ldots,4$ (the first four kPCA modes).
It is clear that absolute values of most gPC coefficients are very small.  Therefore, the gPC expansions 
for these four kPCA modes are sparse. For the other kPCA modes ($e=5,6,7$), 
since the situation is similar to that of the first four kPCA modes, 
their corresponding gPC coefficients are not shown here, while  these gPC expansions are also sparse.
Finally, Figure \ref{fig_difftr_cpl1box} shows Tukey box plots  of the relative errors \eqref{diff_error} for $500$ test samples 
for this test problem, where the central line in each box is the median, the lower and the upper edges are the first and the third
quartiles respectively, and the red crosses are the outliers. 
From Figure \ref{fig_difftr_cpl1box}, it is clear that, as the size of the observation index set ($|\Theta|$)
increases, values of the median, the first and the third quartiles of the errors decrease, which are all consistent  with 
the results in test problem 1.

\begin{table}[!htp]
	\caption{Estimated CP ranks of each data tensor $\zX_e$ for $e=1,\ldots,7$, test problem 2.} 
	\centering	
	\begin{tabular}{c|ccccccc}  
		\hline
		\diagbox{$|\Theta|$}{rank}{$e$} & 1 & 2 & 3 & 4 & 5 & 6 & 7 \\ \hline
		100 & 7 & 1 & 3 & 2 & 1 & 2 & 1\\
		300 & 6 & 3 & 1 & 1 & 7 & 1 & 1\\
		600 & 7 & 1 & 4 & 3 & 7 & 1 & 3\\
		\hline
	\end{tabular}
	\label{table_cprank_ftr}
\end{table}
\begin{figure}[!ht]
	\centering
	\subfloat[][Validation errors w.r.t CP ranks ]{\includegraphics[width=.2670\textwidth]{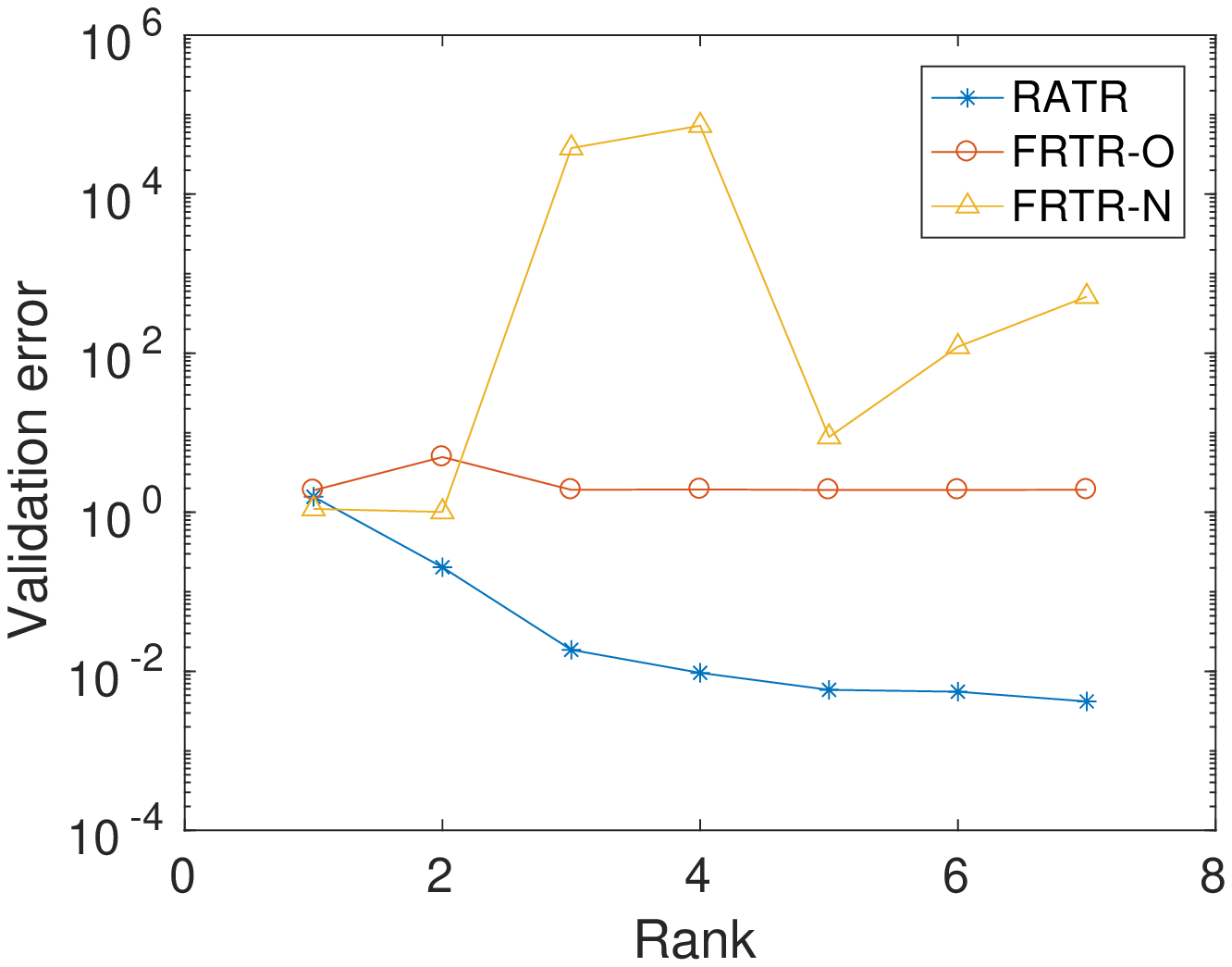}}\quad
	\subfloat[][Validation errors at $R=1$ ]{\includegraphics[width=.2670\textwidth]{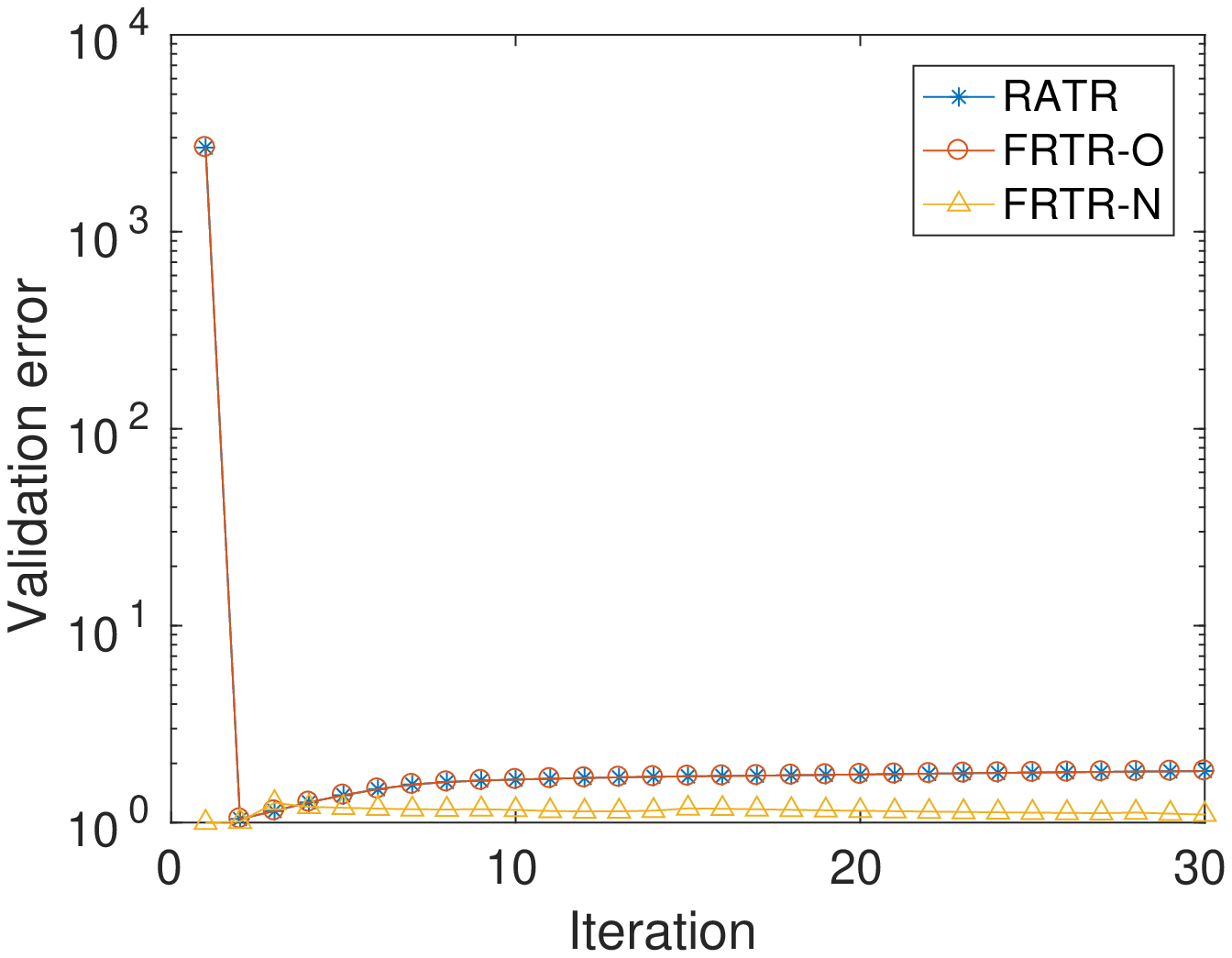}}\quad
	\subfloat[][Validation errors at $R=2$ ]{\includegraphics[width=.2670\textwidth]{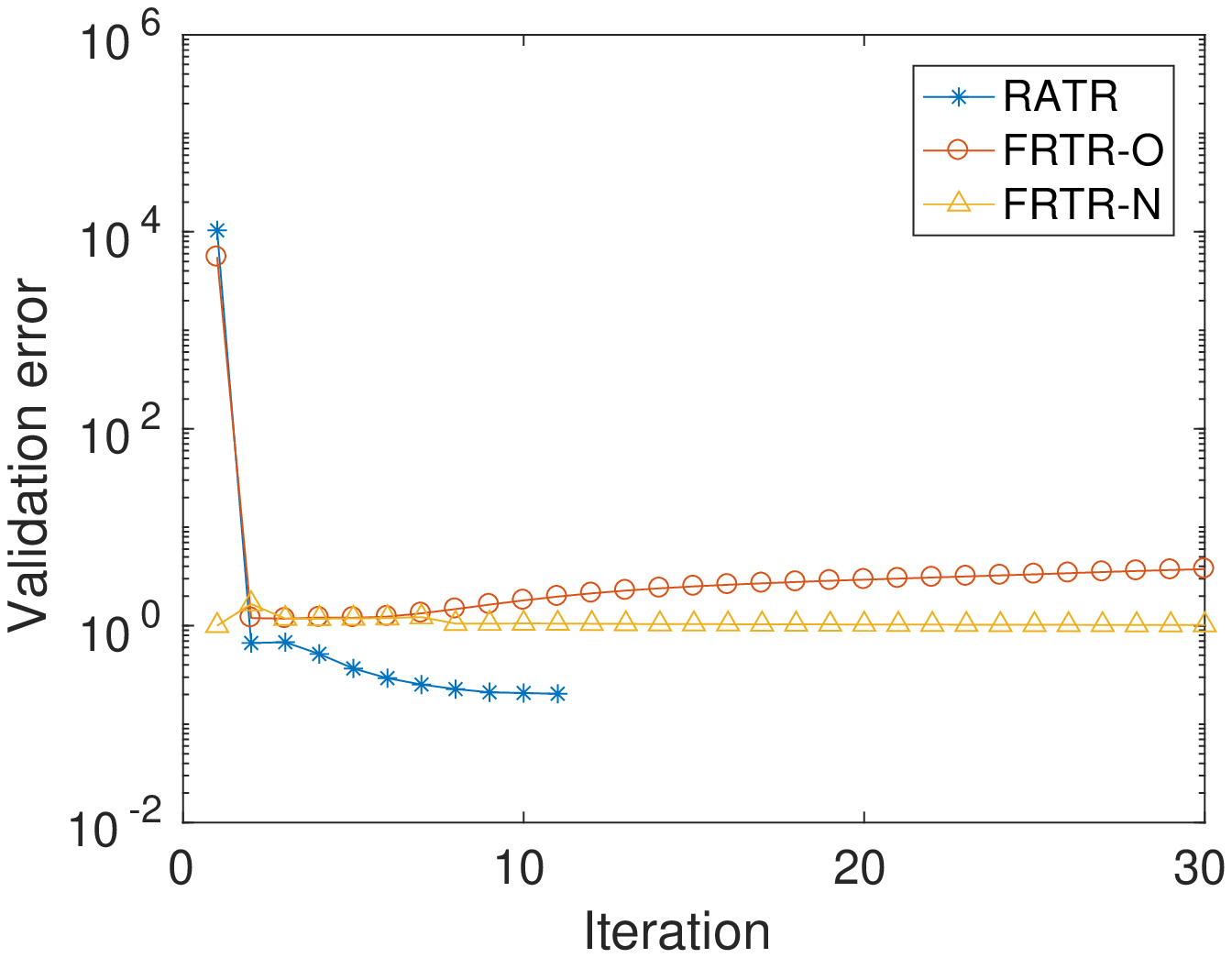}}\\
	\subfloat[][Validation errors at $R=3$ ]{\includegraphics[width=.2670\textwidth]{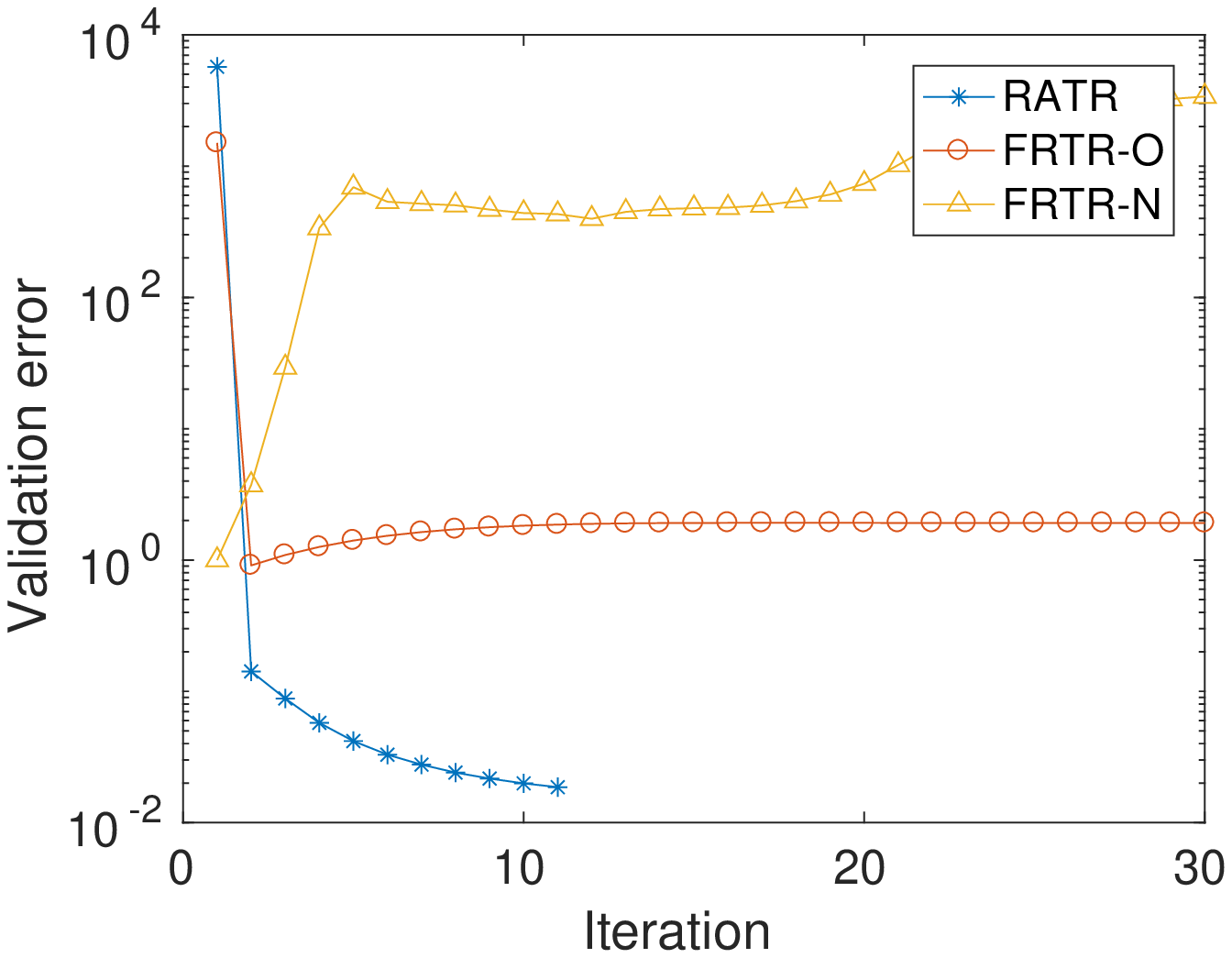}}\quad
	\subfloat[][Validation errors at $R=4$ ]{\includegraphics[width=.2670\textwidth]{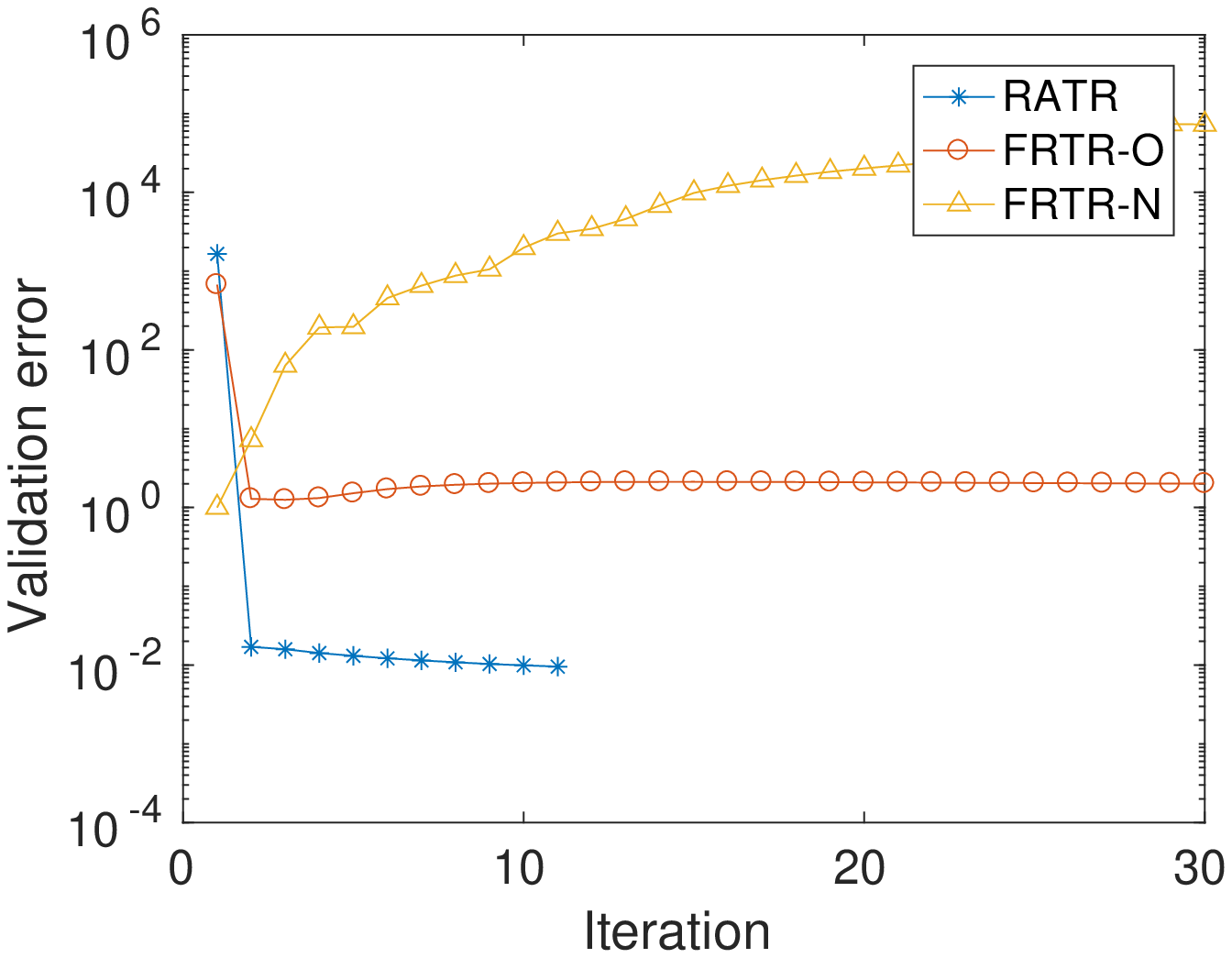}}\quad
	\subfloat[][Validation errors at $R=7$ ]{\includegraphics[width=.2670\textwidth]{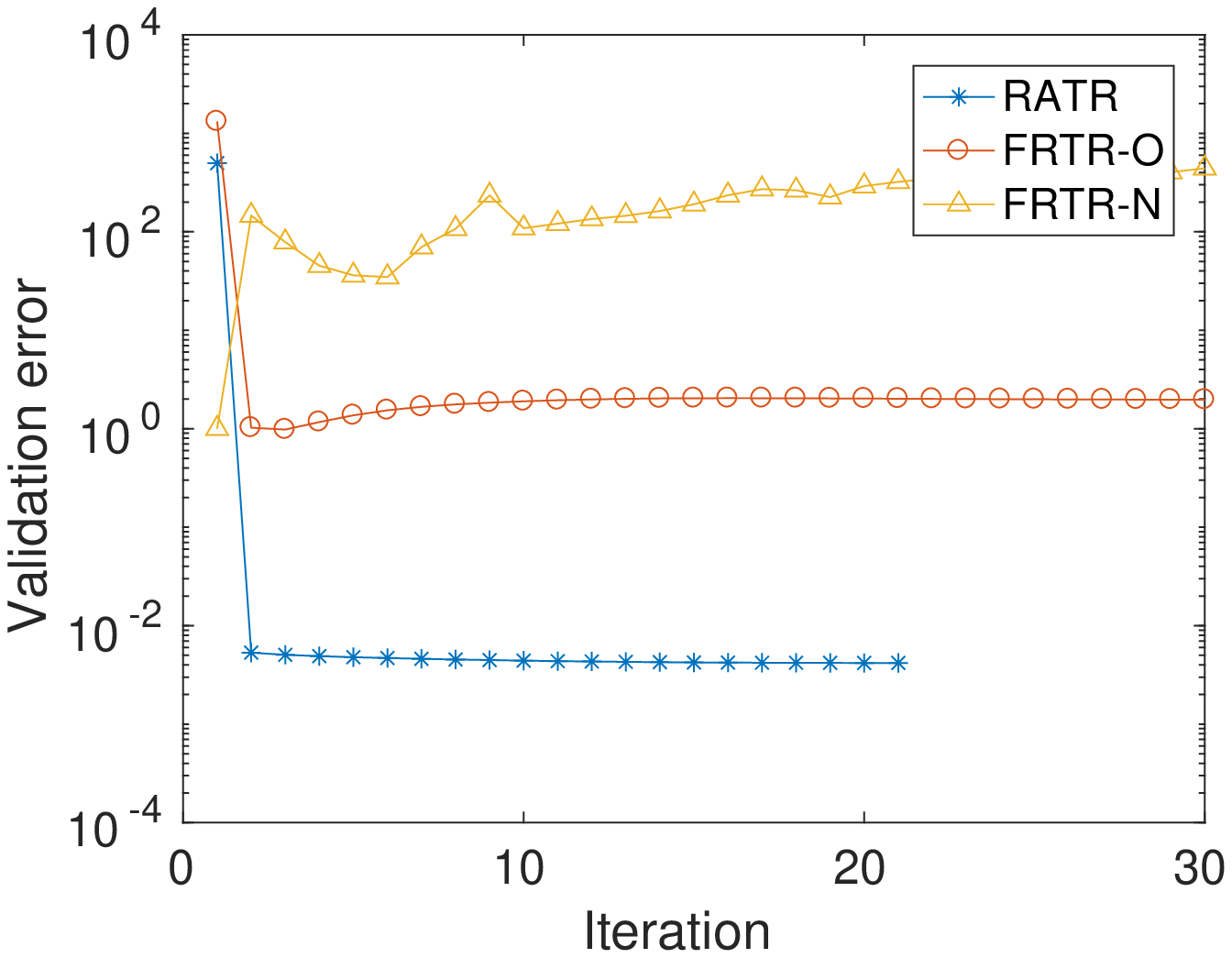}}\\
	\caption{Validation errors of rank adaptive tensor recovery (RATR),
		fixed-rank tensor recovery with initial factor matrices generated through $\mathrm{U}(1,2)$  (FRTR-O),
		and fixed-rank tensor recovery with initial factor matrices generated through $\mathrm{N}(0,1)$  (FRTR-N), 
		test problem 2. }
	\label{fig_errortr_iteration}
\end{figure}
\begin{figure}[!ht]
	\centering
	\subfloat[][1st kPCA mode]{\includegraphics[width=.4\textwidth]{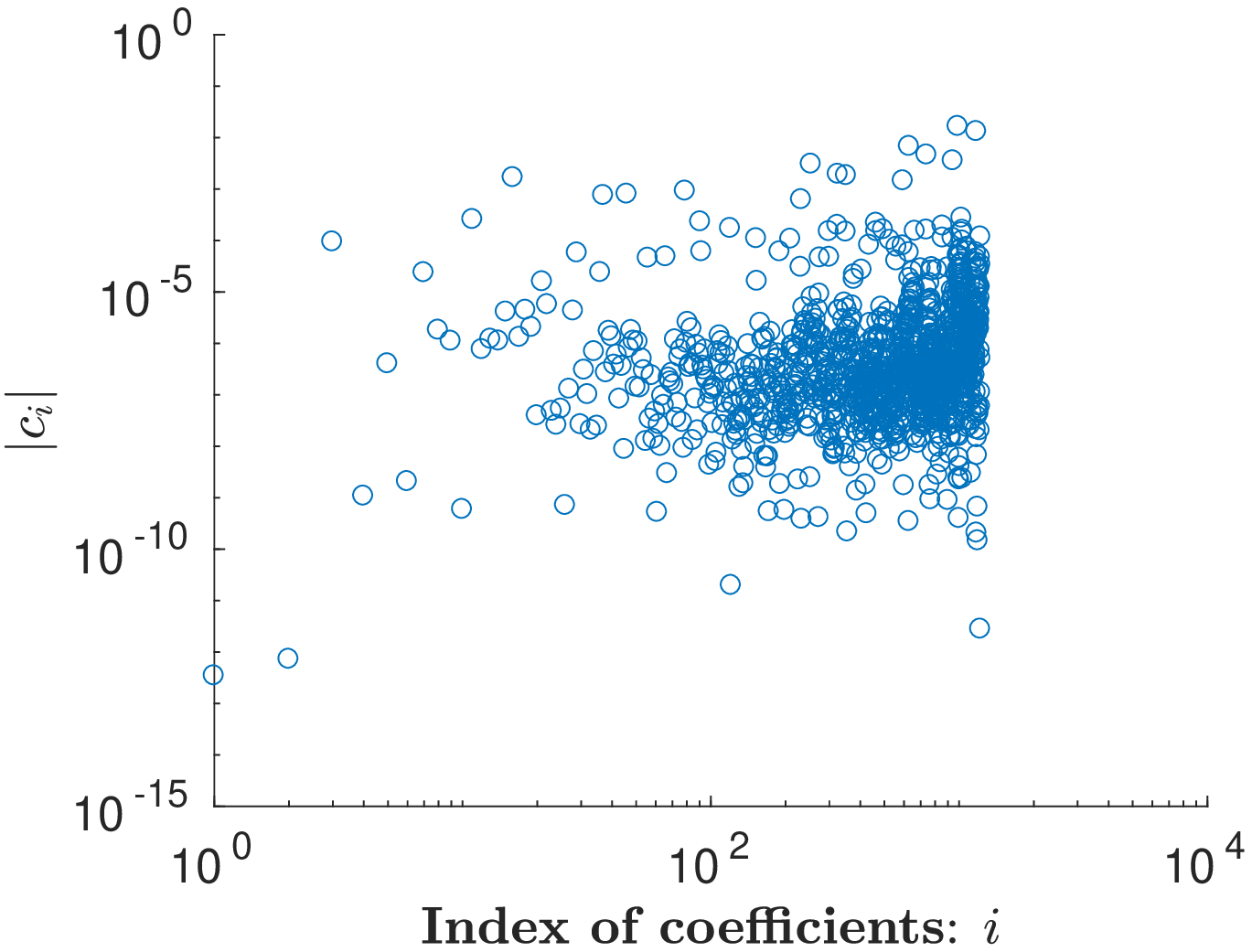}}\quad
	\subfloat[][2nd kPCA mode]{\includegraphics[width=.4\textwidth]{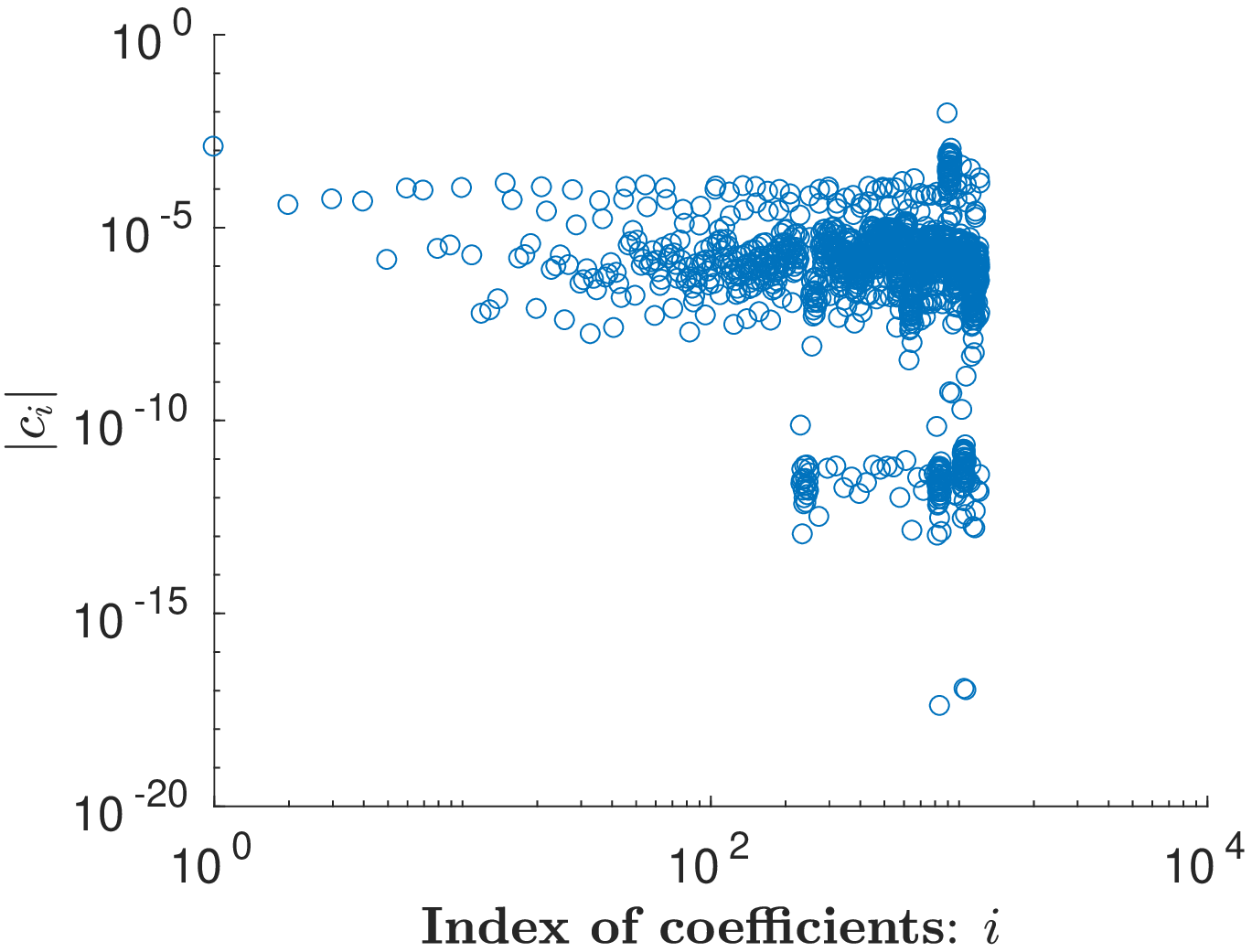}}\\
	\subfloat[][3rd kPCA mode]{\includegraphics[width=.4\textwidth]{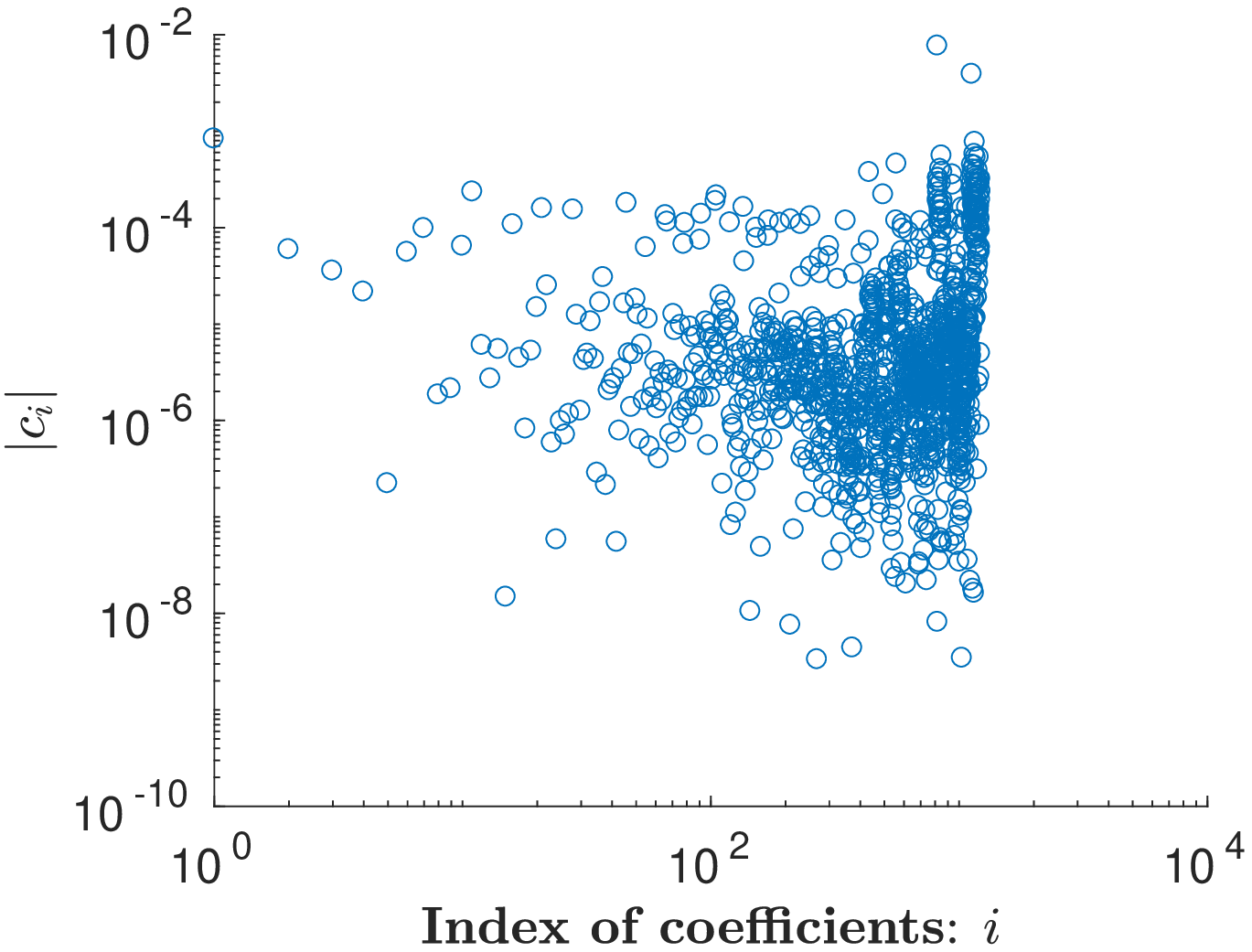}}\quad
	\subfloat[][4th kPCA mode]{\includegraphics[width=.4\textwidth]{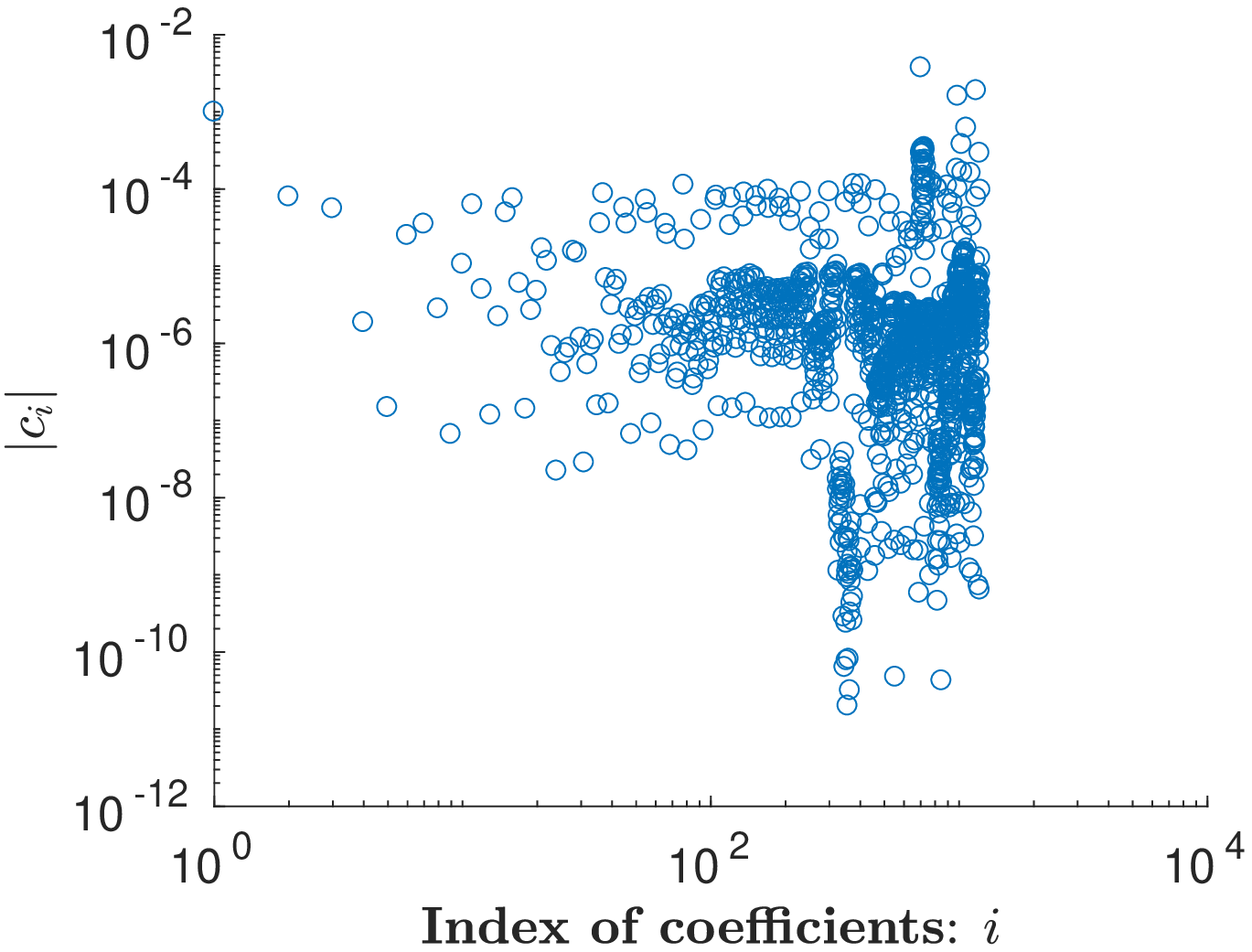}}
	\caption{Sparsity of gPC coefficients for each kPCA mode, test problem 2.}
	\label{fig_difftr_cpl1coeff}
\end{figure}
\begin{figure}[!htp]
	\centering
	\includegraphics[width=0.58\linewidth, height = 0.38\linewidth]{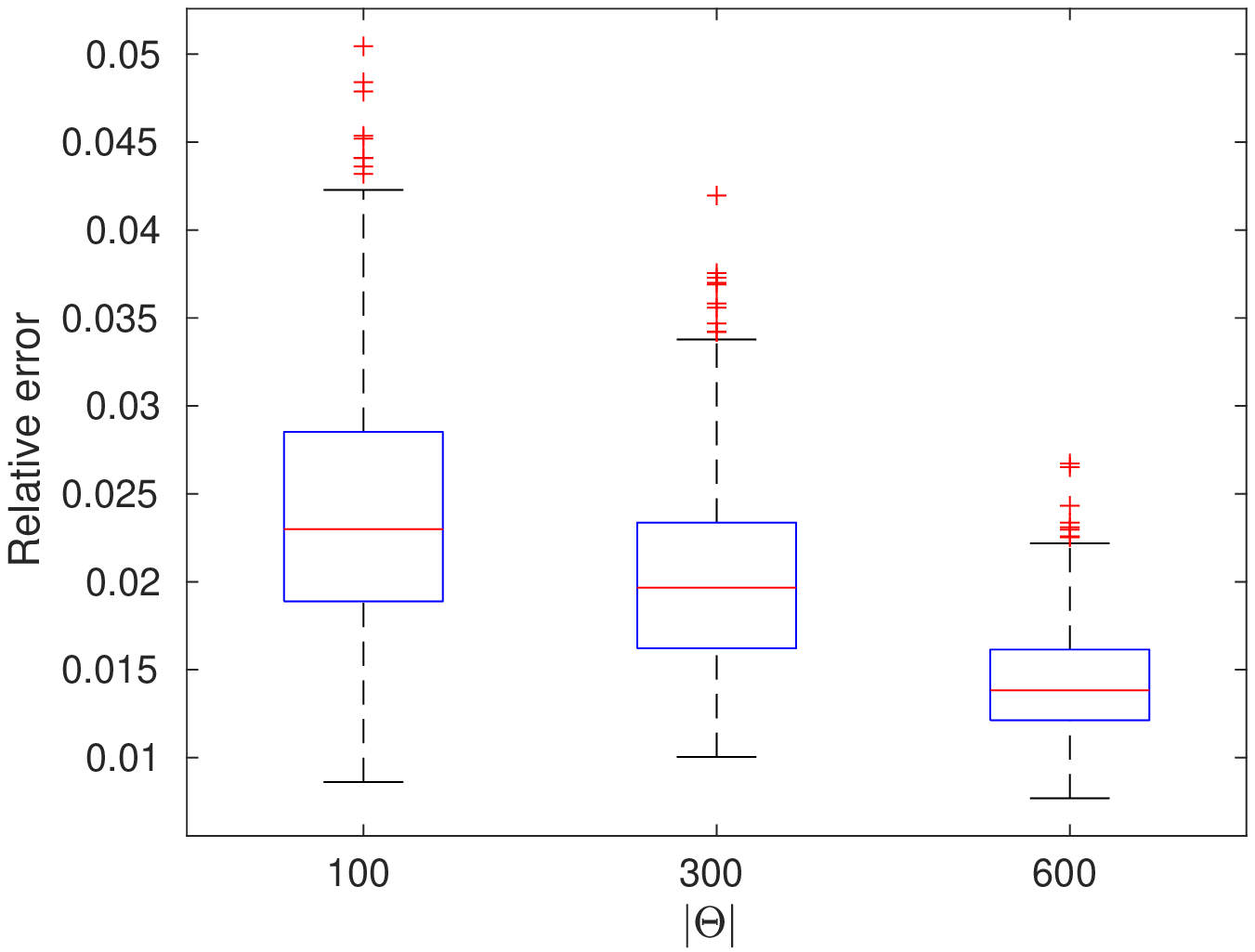}
	\caption{Relative errors of RATR-collocation approximation for 500 test samples, test problem 2.}
	\label{fig_difftr_cpl1box}
\end{figure}

\subsection{Test problem 3: the Stokes equations} \label{sec_numexp_stokes}
The governing equations for this test problem are
\begin{align} 
\nabla \cdot [a(x, \mb{\xi}) \nabla u(x, \mb{\xi})] + \nabla p(x, \mb{\xi}) &= 0 \quad \mathrm{in} \quad D  \times I^d, \label{eq_stokes1}\\
\nabla \cdot u(x, \mb{\xi}) &= 0 \quad \mathrm{in} \quad D \times I^d, \label{eq_stokes2}\\ 
u(x, \mb{\xi}) &= g(x) \quad \mathrm{on} \quad \partial D \times I^d, \label{eq_stokes3}
\end{align}
where $D\subset \dsR^2$, and $u(x, \mb{\xi}) = [u_1(x, \mb{\xi}), u_2(x, \mb{\xi})]^T$ and $p(x, \mb{\xi})$ are the flow velocity and 
the scalar pressure respectively. 
We consider the problem with  uncertain viscosity $a(x, \mb{\xi})$, which is assumed to be a random field with 
mean function $a_0(x) = 1$, variance
$\sigma^2 = 0.25$, and covariance function (\ref{eq_cov}). The correlation length is set to $l_c=0.8$, and we take $d = 48$ 
to capture $95\%$ of the total variance as in test problem 1. We here consider the driven cavity flow problem 
posed on $D = (0,1) \times (0,1)$.  For boundary conditions, the velocity profile $u = [1, 0]^T$ is imposed on the top boundary 
($x_2 = 1$ where $x = [x_1, x_2]^T$), and $u = [0, 0]^T$ is imposed on all other boundaries. 
We discretize in space using the inf-sup stable $\mb{Q}_2-\mb{P}_{-1}$ mixed finite element method (biquadratic velocity--linear 
discontinuous pressure \cite{Elman2014}) as implemented in IFISS \cite{ers14} with a uniform $33 \times 33$ grid,  
which yields the velocity degrees of freedom $N_{h,\,u} = 2178$ and  the pressure degrees of freedom $N_{h,\,p} = 768$. 
The output $\mb{y}$ here is defined to be a vector collecting both discrete velocity and pressure solutions,
and the overall dimension of $\mb{y}$ is then $N_h = N_{h,\,u} + N_{h,\,p}=2946$. 
For this test problem, the regularization parameter $\beta$ in (\ref{cp_missvalue_reg}) is set to 0.1, and the tolerance in 
Algorithm \ref{alg_ftr} is set to $\delta = 10^{-5}$. The bandwidth $\sigma_g$ of kPCA for dimension reduction is set to $10$, 
and we again set $tol_{\text{PCA}}$ to $90\%$. 

We first generate the corresponding data matrix $\mb{Y}$ for the three cases of $\Theta$ ($|\Theta| = 100, 300$ and $600$) 
and apply kPCA on it. For $tol_{\text{PCA}}=90\%$, our results show that the number of kPCA modes retained is $N_r=9$ 
for the case $|\Theta| = 100$, while $N_r=10$ for the cases $|\Theta| = 300$ and  $|\Theta| = 600$, 
which implies that the sample size $100$ may not be large enough for an accurate dimension reduction. 
Tabel~\ref{table_cprank_f_stokes} shows the estimated CP ranks 
of $\zX_e$ generated through Algorithm~\ref{alg_rankadaptive_cpl1} for each kPCA mode $e=1,\ldots,N_r$. 
Again, it is clear that, the estimated ranks are small---the maximum estimated rank is only seven. This 
shows that the rank-one update produced in Algorithm~\ref{alg_rankadaptive_cpl1} is performed seven times at most,
and it is therefore not costly. 

Figure \ref{fig_sterror_iteration} shows the validation errors of our RATR (Algorithm~\ref{alg_rankadaptive_cpl1} with 
initial factor matrices generated through $\mathrm{U}(1,2)$), FRTR-O and FRTR-N (Algorithm~\ref{alg_ftr} with
initial factor matrices generated through $\mathrm{U}(1,2)$ and $\mathrm{N}(0,1)$ respectively)
for recovering $\mb{\mathrm{X}}_1$ (see \eqref{eq_data_tensor}) with $|\Theta| = 600$.
From Figure \ref{fig_sterror_iteration}(a), it can be seen that our RATR has the smallest validation error for each rank.
It is also clear that, as the rank increases from one to three, the errors of RATR reduces significantly,
while the iterations from rank three to seven are caused by our stopping criterion on line 9 of  Algorithm \ref{alg_rankadaptive_cpl1}.
The other pictures in  Figure \ref{fig_sterror_iteration} show the validation errors at each iteration step of the 
alternative minimization iterative procedure for $R=1,2,3,6,7$ (while the errors for $R=4,5$ are similar to the errors for $R=6$,
they are not shown here). Similarly to test problem 1, for the case $R=1$, 
RATR is the same as FRTR-O, and their errors  are smaller than the error of FRTR-N. 
For $R=2,3,6,7$, the validation error of RATR is again clearly smaller than the errors of FRTR-O 
and FRTR-N, which shows that our rank-one update procedure is efficient for this Stokes problem. 

Figure \ref{fig_stokes_cpl1coeff1} shows the absolute values of the gPC coefficients of  
the first four kPCA modes for this test problem. It is clear that absolute values of most gPC coefficients are small,
and  the gPC expansions for these four kPCA modes are therefore sparse. For the other kPCA modes ($e=5,6,7$), 
while the situation is similar (the corresponding gPC expansions are also clearly sparse), they are not shown here. 
Figure \ref{fig_stokes_cpl1_sol} shows the flow streamlines and the pressure fields
generated by the mixed finite element method and RATR-collocation (see section \ref{section_ratrcollocation}) 
responding to a given realization of $\mb{\xi}$. 
It can be seen that there is no visual difference between the results obtained through finite elements and RATR-collocation. 
Finally, we generate $500$ samples of $\mb{\xi}$ and the compute the relative errors \eqref{diff_error}. 
Figure \ref{fig_difftr_cpl1box} shows Tukey box plots  of these errors, where the central line in each box is the median, the lower and the upper 
edges are the first and the third quartiles respectively. It is clear that, as the size of the observation index set ($|\Theta|$)
increases, values of the median, the first and the third quartiles of the errors decrease, which are all consistent  with 
the results of the diffusion test problems.

\begin{table}[!htp]
	\caption{Estimated CP ranks of each data tensor $\zX_e$ for $e=1,\ldots,10$, test problem 3.} 
	\centering	
	\begin{tabular}{c|cccccccccc}  
		\hline
		\diagbox{$|\Theta|$}{rank}{$e$} & 1 & 2 & 3 & 4 & 5 & 6 & 7 & 8 & 9 & 10 \\ \hline
		100 & 3 & 5 & 7 & 1 & 2 & 2 & 4 & 1 & 3 & --\\
		300 & 7 & 3 & 2 & 1 & 1 & 5 & 5 & 2 & 2 & 2\\
		600 & 7 & 7 & 7 & 4 & 2 & 2 & 2 & 1 & 4 & 7\\
		\hline
	\end{tabular}
	\label{table_cprank_f_stokes}
\end{table}   
\begin{figure}[!ht]
	\centering
	\subfloat[][Validation errors w.r.t CP ranks ]{\includegraphics[width=.2670\textwidth]{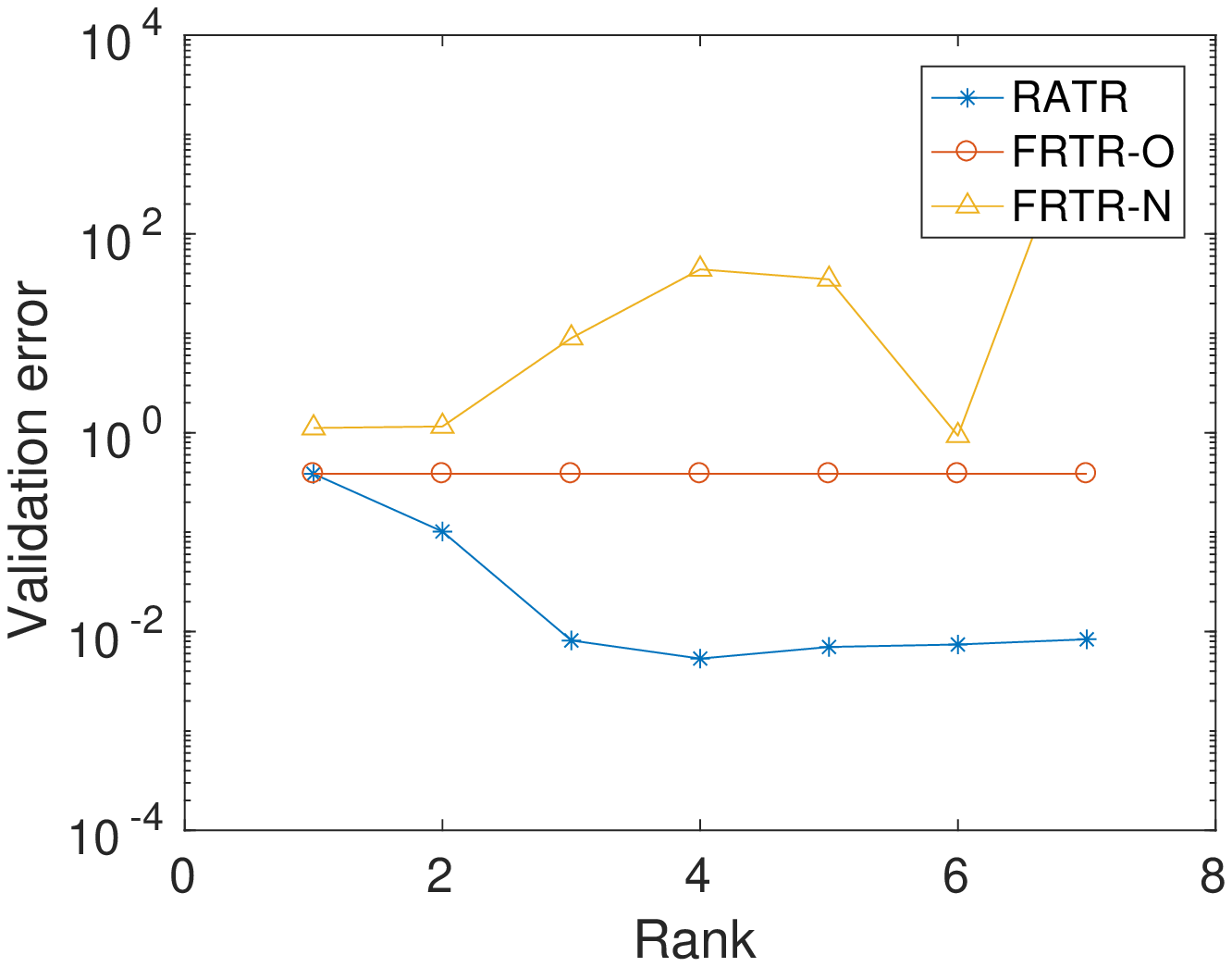}}\quad
	\subfloat[][Validation errors at $R=1$ ]{\includegraphics[width=.2670\textwidth]{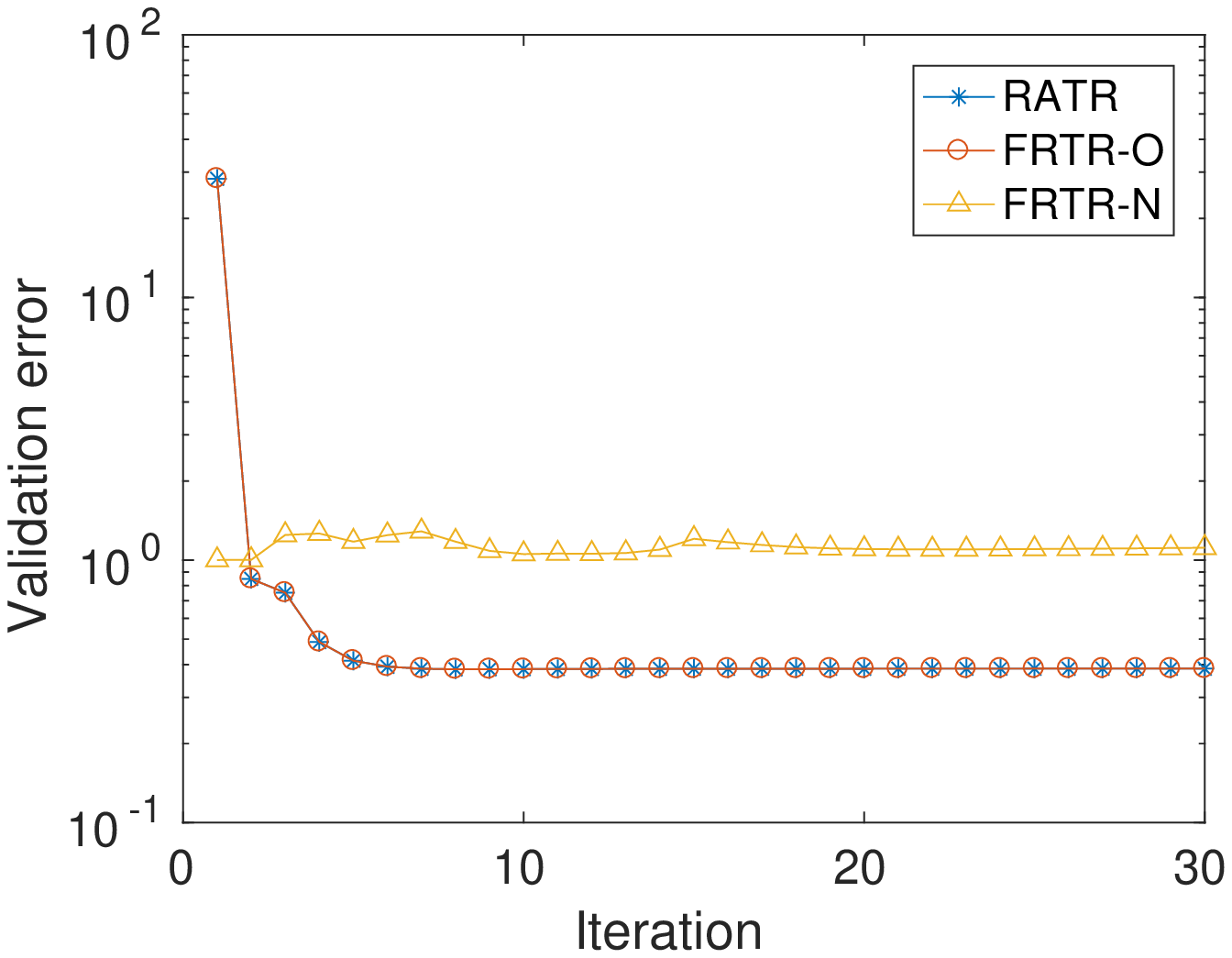}}\quad
	\subfloat[][Validation errors at $R=2$ ]{\includegraphics[width=.2670\textwidth]{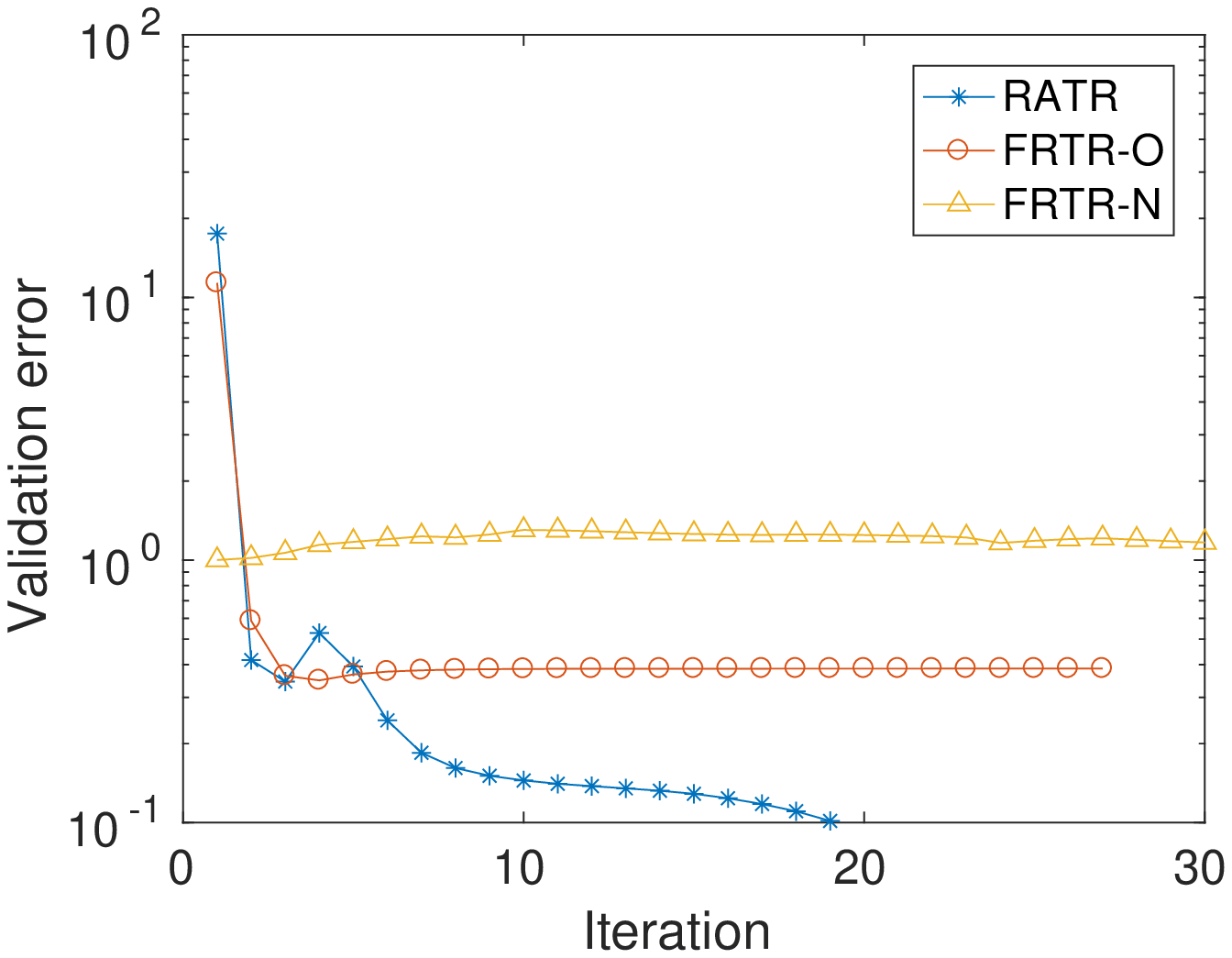}}\\
	\subfloat[][Validation errors at $R=3$ ]{\includegraphics[width=.2670\textwidth]{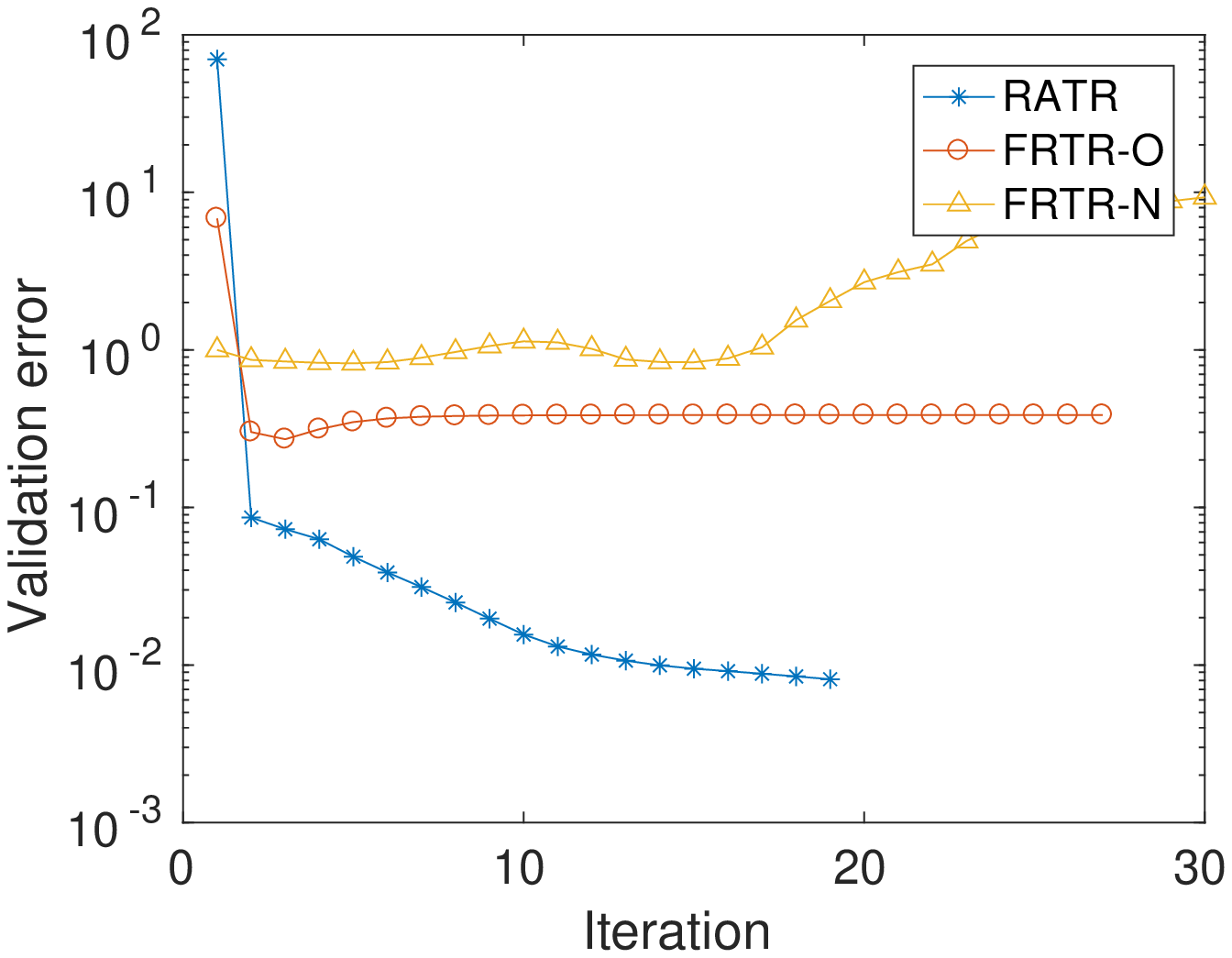}}\quad
	\subfloat[][Validation errors at $R=6$ ]{\includegraphics[width=.2670\textwidth]{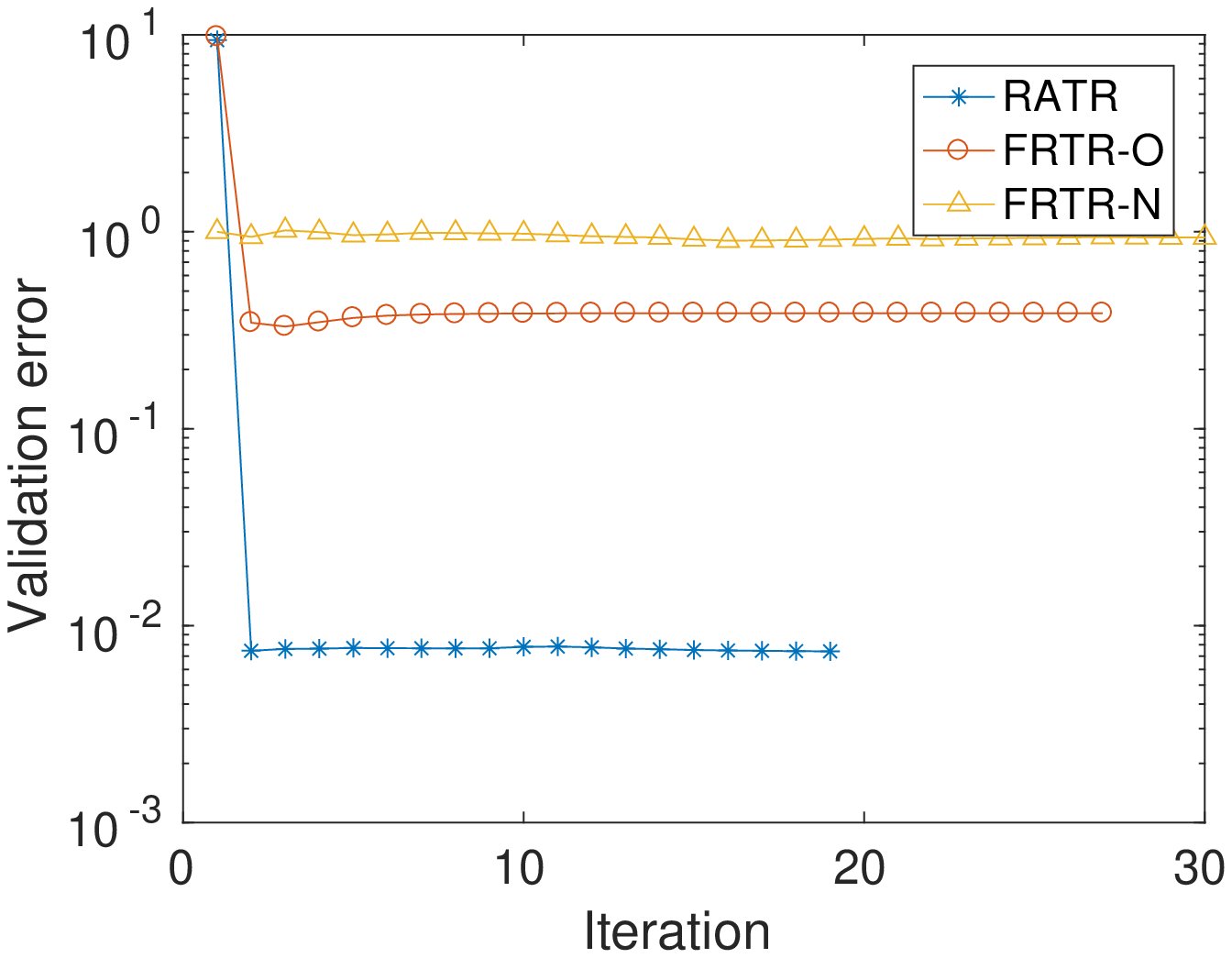}}\quad
	\subfloat[][Validation errors at $R=7$ ]{\includegraphics[width=.2670\textwidth]{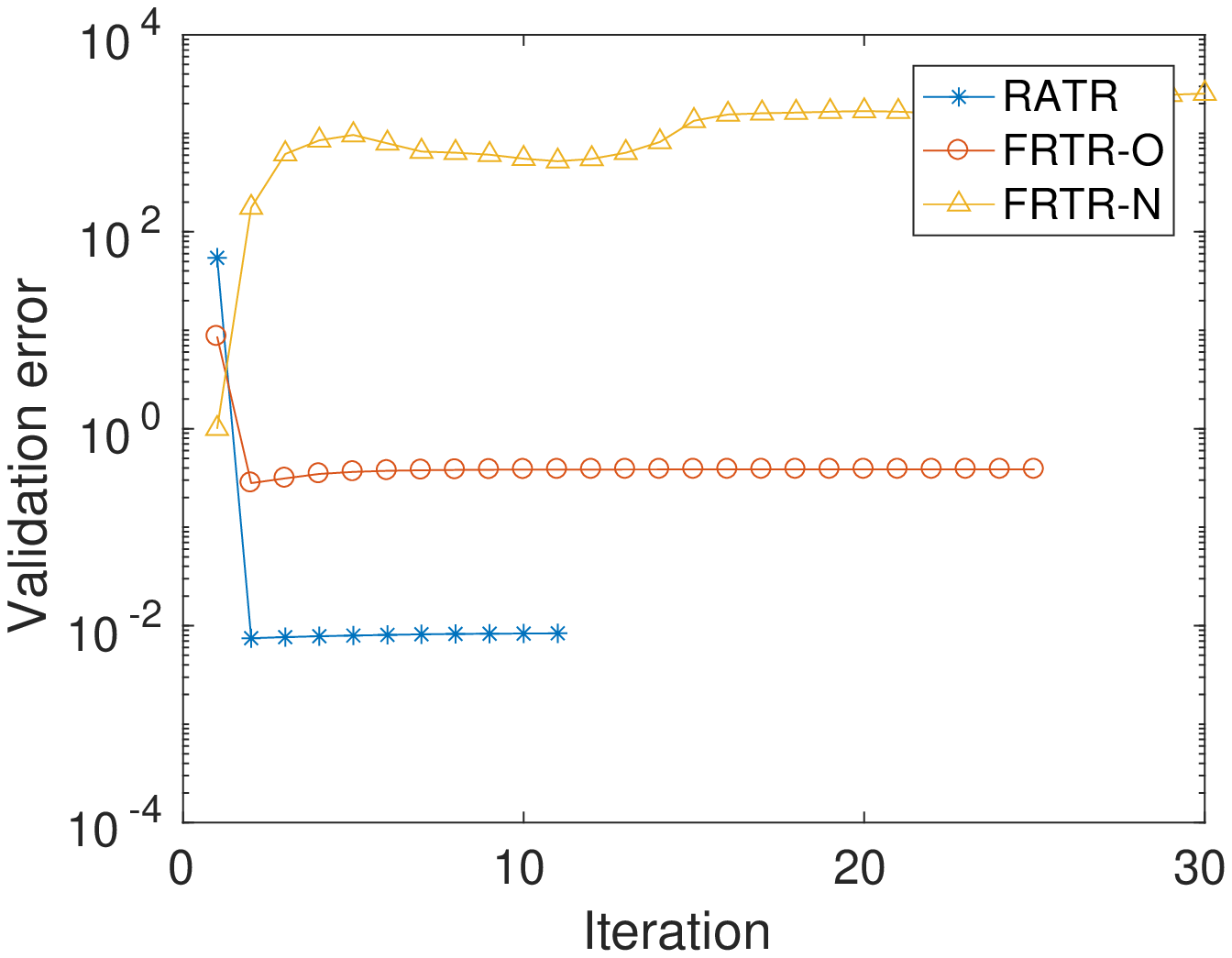}}\\
	\caption{Validation errors of rank adaptive tensor recovery (RATR),
		fixed-rank tensor recovery with initial factor matrices generated through $\mathrm{U}(1,2)$  (FRTR-O),
		and fixed-rank tensor recovery with initial factor matrices generated through $\mathrm{N}(0,1)$  (FRTR-N), 
		test problem 3.}
	\label{fig_sterror_iteration}
\end{figure}

\begin{figure}[!ht]
	\centering
	\subfloat[][1st kPCA mode]{\includegraphics[width=.4\textwidth]{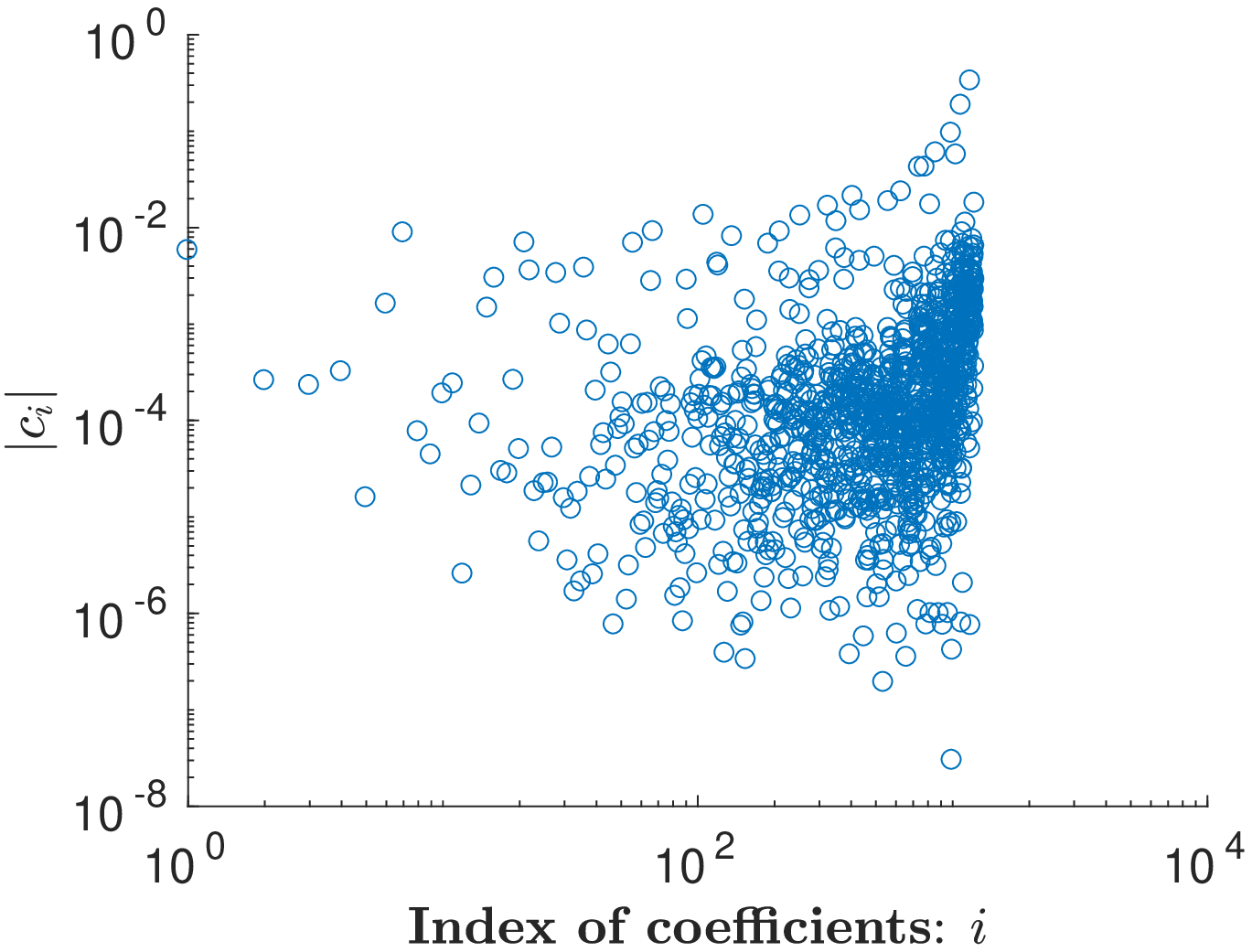}}\quad
	\subfloat[][2nd kPCA mode]{\includegraphics[width=.4\textwidth]{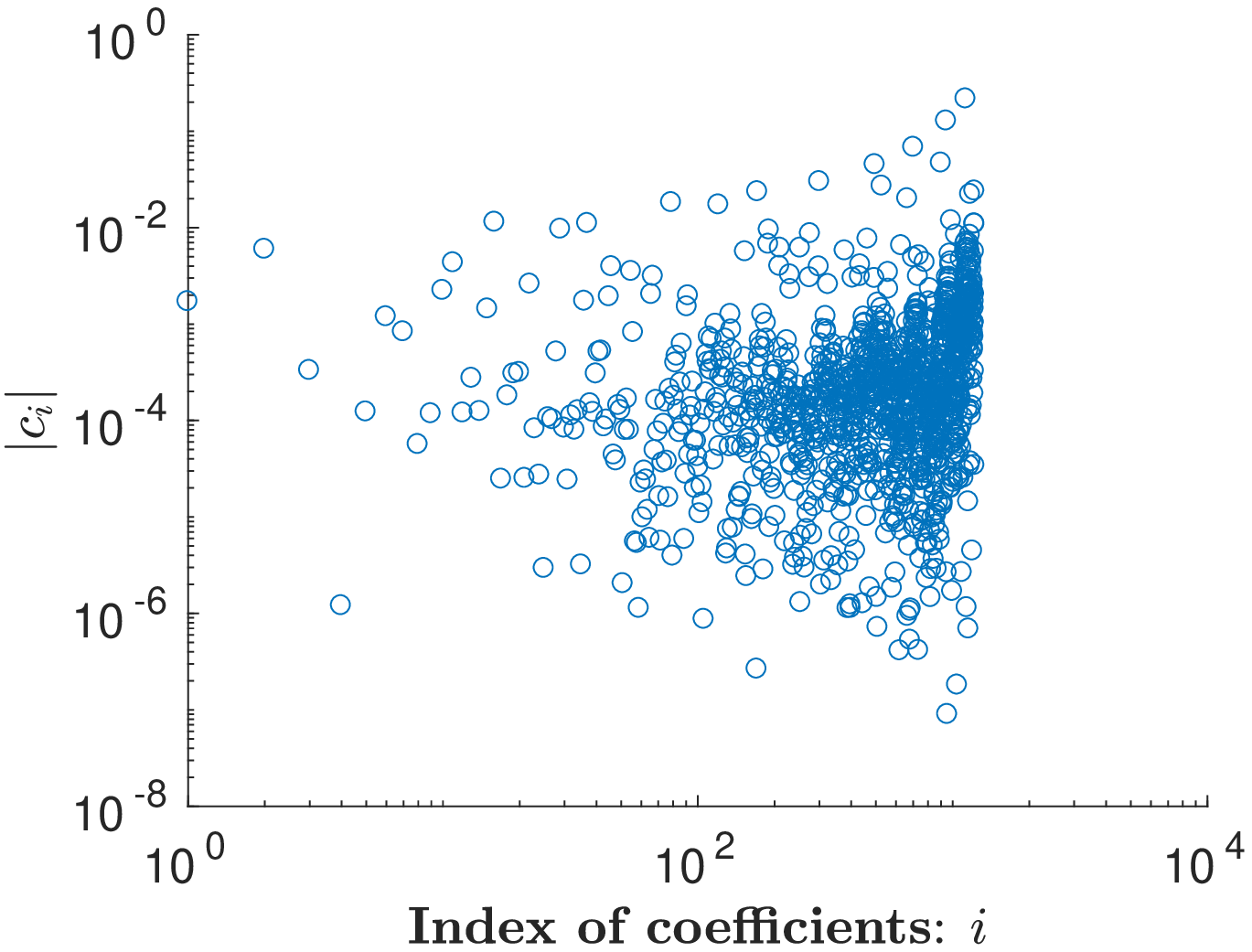}}\\
	\subfloat[][3rd kPCA mode]{\includegraphics[width=.4\textwidth]{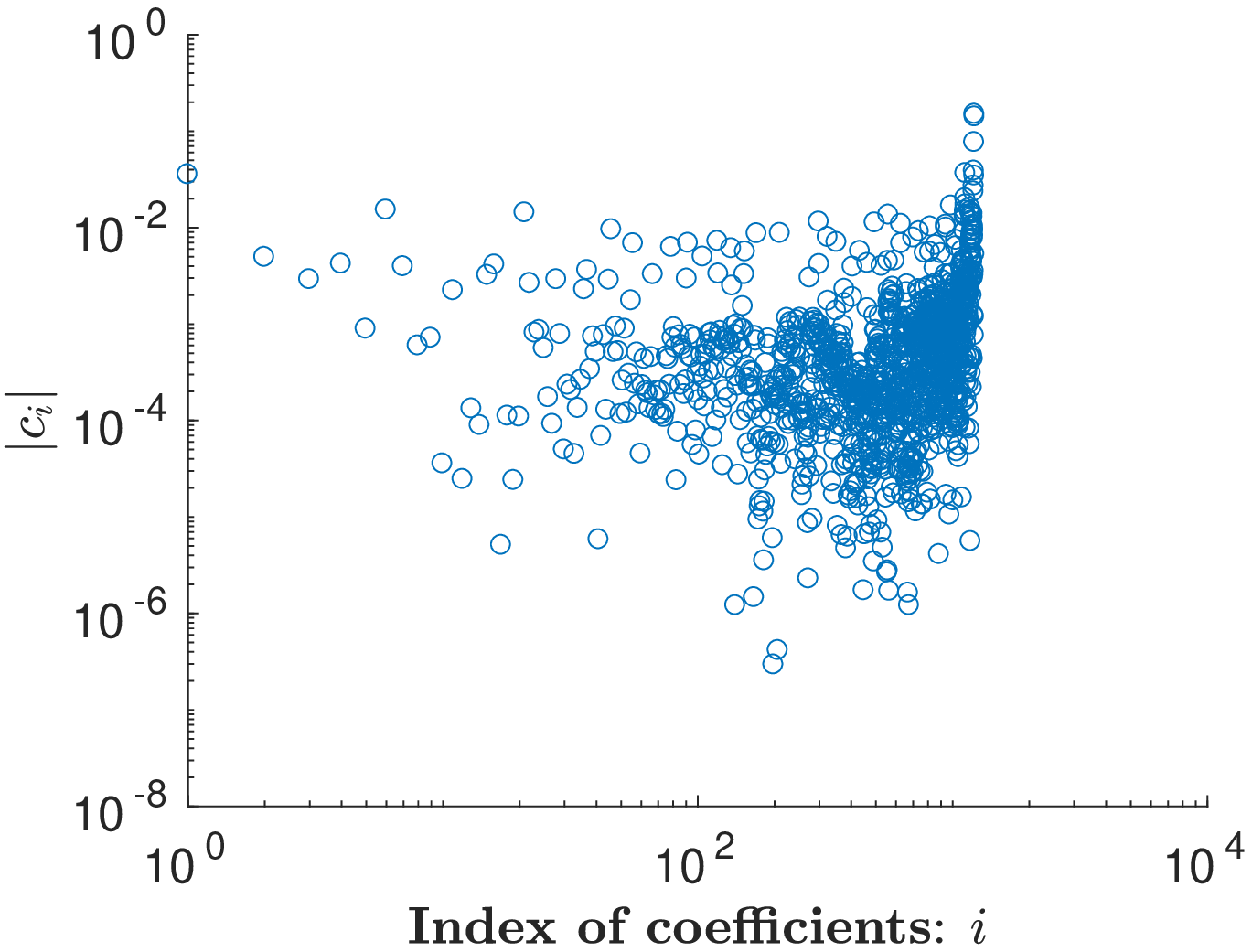}}\quad
	\subfloat[][4th kPCA mode]{\includegraphics[width=.4\textwidth]{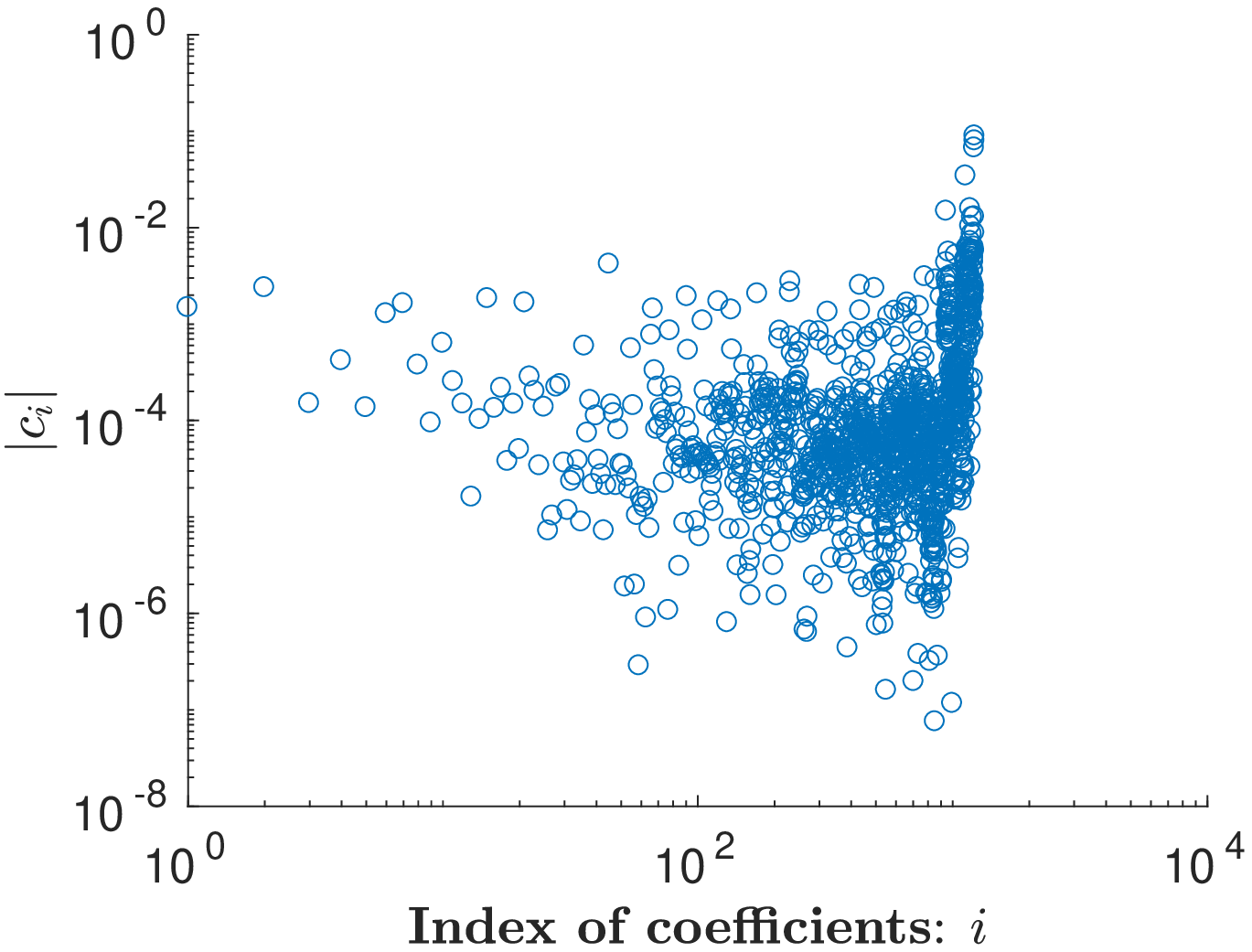}}
	\caption{Sparsity of gPC coefficients for each kPCA mode, test problem 3.}
	\label{fig_stokes_cpl1coeff1}
\end{figure}

\begin{figure}[!ht]
	\centering
	\subfloat[][Streamline: finite element solution]{\includegraphics[width=.445\textwidth]{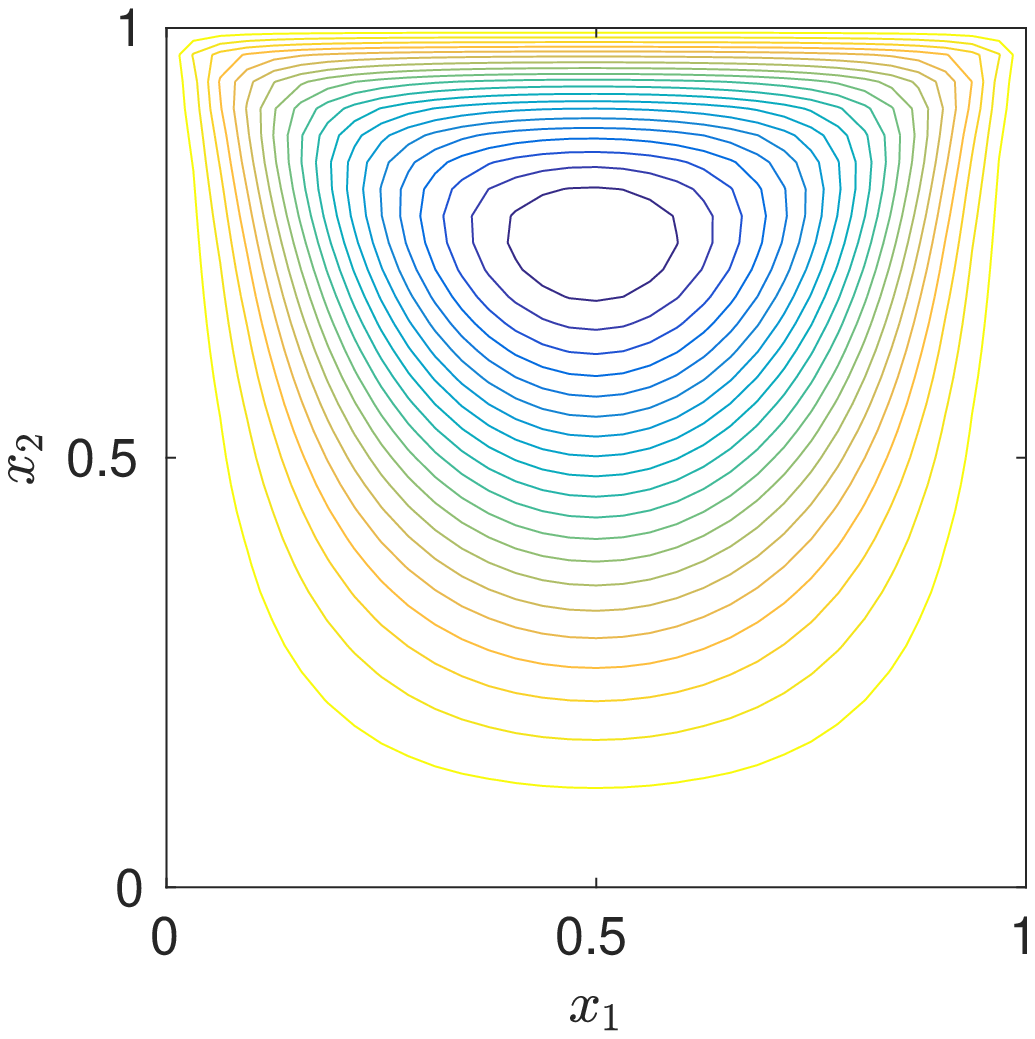}}\quad
	\subfloat[][Streamline: RATR-collocation approximation]{\includegraphics[width=.45\textwidth]{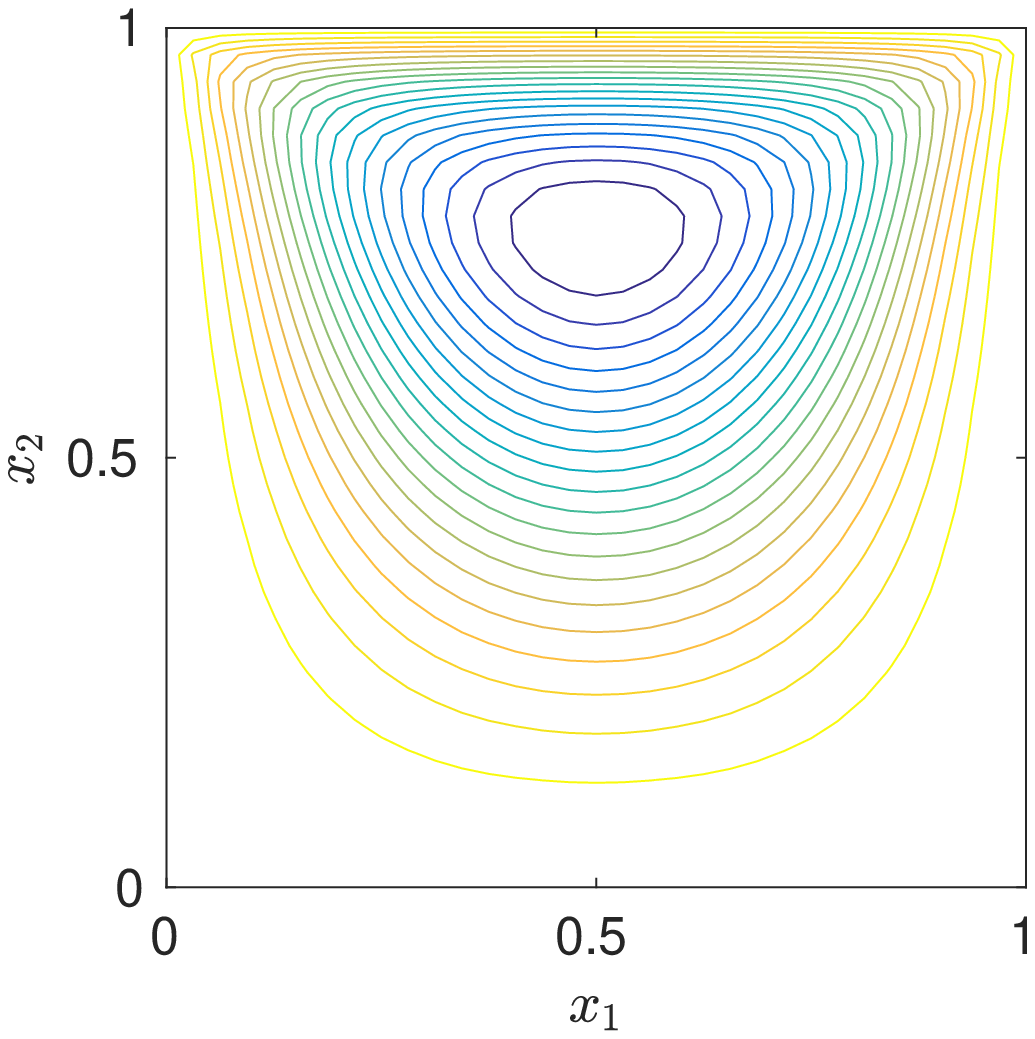}}\\
	\subfloat[][Pressure: finite element solution]{\includegraphics[width=.45\textwidth]{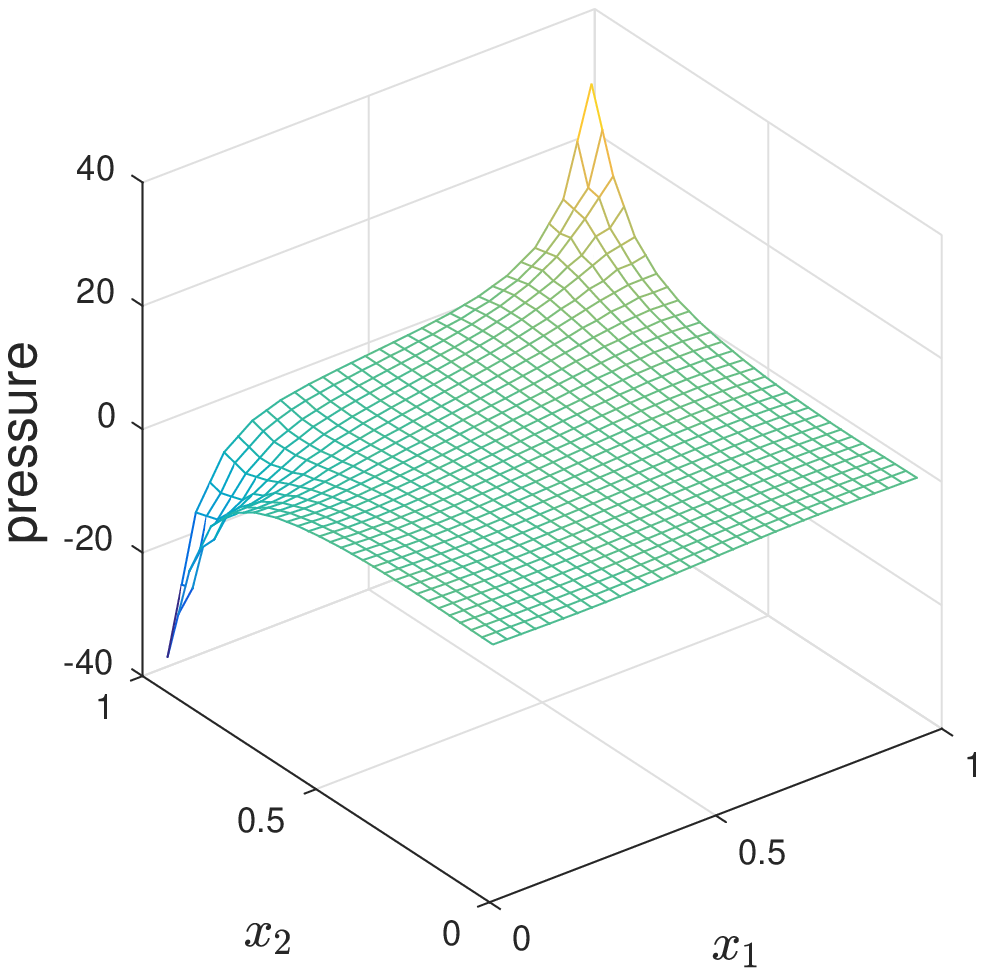}}\quad
	\subfloat[][Pressure: RATR-collocation approximation ]{\includegraphics[width=.45\textwidth]{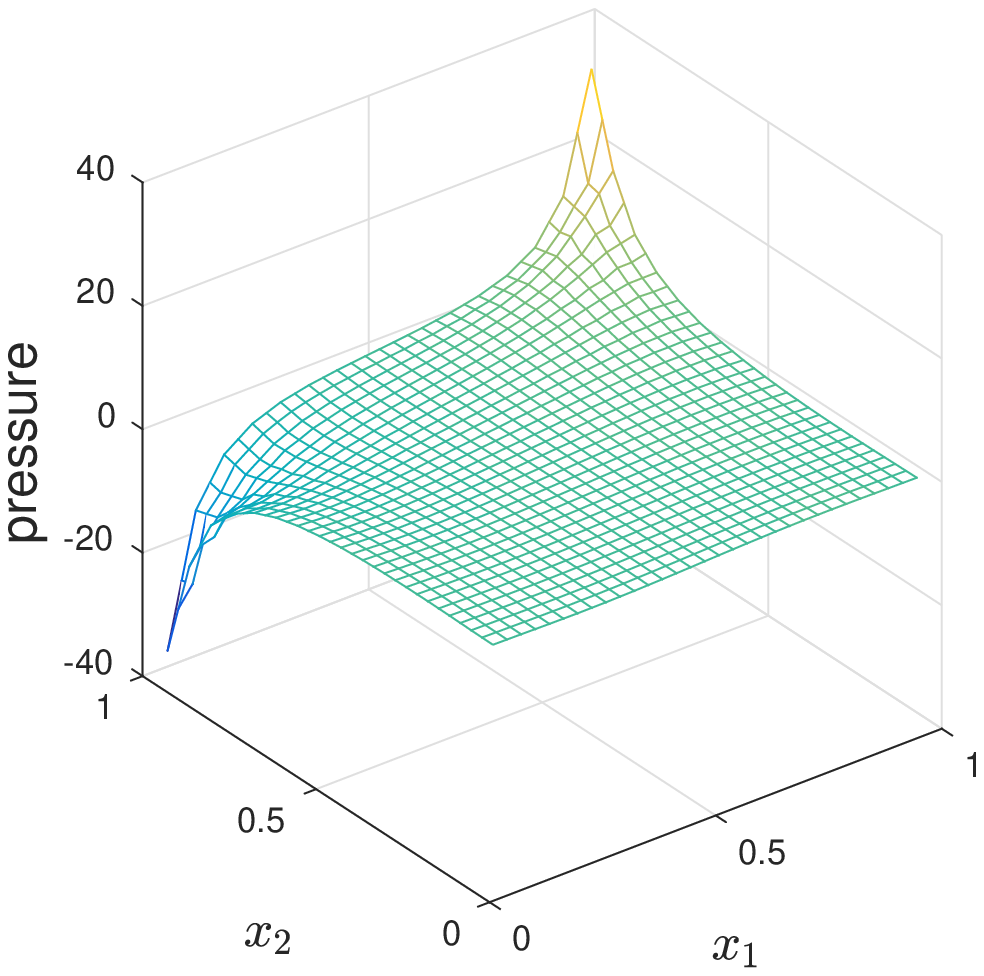}}
	\caption{The finite element solution and the RATR-collocation approximation (with $|\Theta|=600$) responding to a given realization of $\xi$,  test problem 3.}
	\label{fig_stokes_cpl1_sol}
\end{figure}
\begin{figure}[!htp]
	\centering
	\includegraphics[width=0.58\linewidth, height = 0.38\linewidth]{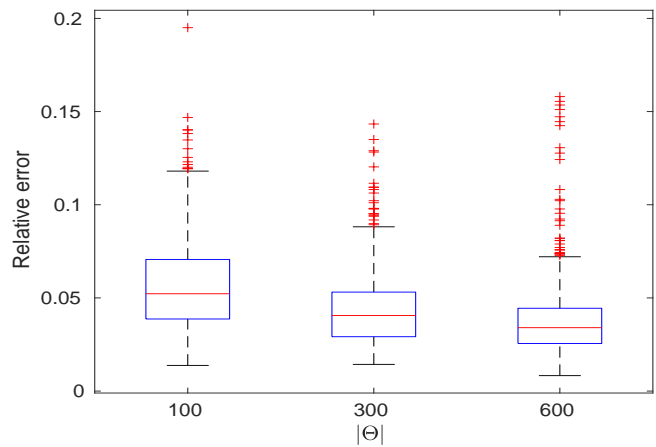}
	\caption{Relative errors of RATR-collocation approximation for 500 test samples, test problem 3.}
	\label{fig_stokes_cpl1box}
\end{figure}

\section{Conclusions}\label{section_conclude}
Exploiting potential low-dimensional  structures is a fundamental concept of efficient 
surrogate and reduced order modelling for high-dimensional UQ problems. With a focus on the tensor recovery based
stochastic collocation, our main conclusion is that our rank adaptive tensor recovery collocation (RATR-collocation)
approach  can efficiently exploit low-dimensional  structures in this challenging problem in two aspects: 
first, we reformulate stochastic colocation based on manifold learning,
where nonlinear low-dimensional  structures in the snapshots are captured through kPCA; second, 
our novel RATR algorithm automatically explores the low-rank structures in the data tensors for computing the collocation coefficients
without requiring a given tensor rank.
Moreover, another main contribution of this work is the analysis of RATR, 
where the stability of our initialization strategies and the rank-one update procedure for the non-convex optimization problems involved 
is proven theoretically, such that  a systematic guidance to initialize the the factor matrices is provided to result in efficient  and
stable recovery results. 
As the performance of RATR algorithm depends on the CP rank of the data tensor (although it does not need to be explicitly given),
our RATR-collocation is efficient  when the CP rank is small, while it may not be efficient for high-rank problems. 
A possible solution for efficiently recovering high-rank tensors is to conduct adaptivity with respect to physical properties 
of the underlying PDE models, e.g., domain decomposition methods. 
Designing and analyzing  such strategies will be the focus of our future work.

\bigskip
\textbf{Acknowledgments:}
This work is supported by the National Natural Science Foundation of China (No. 11601329) and 
the science challenge project (No. TZ2018001).

\section*{References}

\bibliography{ratrc}

\end{document}